\documentclass[11pt,twoside]{amsart}

\usepackage{pstricks}
\usepackage{multido}

\usepackage{graphicx,url}
\usepackage{amssymb}
\usepackage{amsmath}
\usepackage{amscd}
\usepackage{amstext}
\usepackage{amsbsy}
\usepackage{amsxtra}
\usepackage{amsfonts}
\usepackage{latexsym}
\usepackage{alltt}
\usepackage{bbm}

\usepackage{verbatim}
\usepackage{amscd}
\usepackage[all]{xy}

\theoremstyle{plain}
\newtheorem{theorem}{Theorem}
\newtheorem*{theorem*}{Theorem}
\newtheorem{lemma}{Lemma}[section]
\newtheorem*{corollary}{Corollary}
\newtheorem{proposition}[lemma]{Proposition}
\newtheorem{remark}[lemma]{Remark}

\newtheorem{definition}[lemma]{Definition}

\newcommand{\LSL}{\Lambda  \SL}
\newcommand{\LSU}{\Lambda  \SU}

\newcommand{\Lsu}{\Lambda \su}

\newcommand{\moduli}{\mathcal{M}}
\newcommand{\mper}{\mathcal{M}_{\mathrm{no} \mathrm{Sym}}}
\newcommand{\bigmoduli}{\mathcal{M}_{\mathrm{plane}}}
\newcommand{\mrot}{\mathcal{M}_{\mathrm{rot}}}
\newcommand{\mae}{\mathcal{M}_{\mathrm{\!Ae}}}
\newcommand{\ann}{\mathcal{C}}
\newcommand{\annrot}{\mathcal{C}_{\mathrm{rot}}}
\newcommand{\annae}{\mathcal{C}_{\mathrm{Ae}}}
\newcommand{\iso}[1]{\mathcal{I}(#1)}
\newcommand{\imm}{\mathrm{f}}
\newcommand{\cmc}{{\sc{cmc~}}}

\newcommand{\mi}{{\mathrm i}}
\newcommand{\C}{\mathbb C}
\newcommand{\R}{\mathbb R}

\newcommand{\Sp}{\mathbb S}
\newcommand{\Z}{\mathbb Z}
\newcommand{\N}{\mathbb N}

\newcommand{\CP}{\mathbb{P}^1}

\newcommand{\tr}{\mathrm{tr}}

\newcommand{\Order}{{\rm O}}

\newcommand{\ubdh}{C'}

\newcommand{\II}{\mathfrak{h}}

\newcommand{\barpot}[1]{\bar{\mathcal{P}}_{#1}}
\newcommand{\pot}[1]{\mathcal{P}_{#1}}

\newcommand{\ind}[1]{_{\mbox{\rm\scriptsize{#1}}}}

\newcommand{\SL}{\mathrm{SL}_{\mbox{\tiny{$2$}}}}

\newcommand{\Sl}{\mathfrak{sl}_{\mbox{\tiny{$2$}}}}

\newcommand{\SU}{\mathrm{SU}_{\mbox{\tiny{$2$}}}}

\newcommand{\su}{\mathfrak{su}_{\mbox{\tiny{2}}}}

\numberwithin{equation}{section}

\title[Alexandrov embedded cmc tori ]
{Mean-convex Alexandrov embedded constant mean curvature tori in the 3-sphere}

\author{L. Hauswirth}
\address{L. Hauswirth, Universit\'e de Marne la Vall\'ee, France.}
\email{hauswirth@univ-mlv.fr}

\author{M. Kilian}
\address{M. Kilian, University College Cork, Ireland}
\email{m.kilian@ucc.ie}

\author{M. U. Schmidt}
\address{M. U. Schmidt, Universit\"at Mannheim, Germany}
\email{schmidt@math.uni-mannheim.de}

\begin{document}


%=========================

\begin{abstract}
We introduce the moduli space of spectral curves of constant mean curvature (\cmc\hspace{-5pt}) cylinders of finite type in the round unit 3-sphere. The subset of spectral curves of mean-convex Alexandrov embedded cylinders is explicitly determined using a combination of integrable systems and geometric analysis techniques. We prove that these cylinders are surfaces of revolution. As a consequence all mean-convex Alexandrov embedded {\sc{cmc}} tori in the 3-sphere are surfaces of revolution.
\end{abstract}
\maketitle

%%%%%%%%%%%%%%%%%%%%%%%%%%%%%%%%%%%%%

\begin{center}
  \sc Introduction
\end{center}

The classification of constant mean curvature  (\cmc\hspace{-5pt}) tori in $\mathbb{R}^3$ by Pinkall-Sterling \cite{PinS} lead Bobenko \cite{Bob:tor} to explicit formulae for all \cmc tori in the three-dimensional spaceforms $\mathbb{H}^3,\,\mathbb{R}^3$ and $\mathbb{S}^3$. Despite these formulae it was not possible to single out the embedded or Alexandrov embedded classes. Brendle \cite{Bre2, Bre} proves that the Clifford torus is the only minimal Alexandrov embedded torus in $\Sp^3$. Andrews-Li \cite{AndL} prove that all embedded \cmc tori in $\Sp^3$ are surfaces of revolution. The problem of determining the mean-convex Alexandrov embedded \cmc tori in the round unit 3-sphere is settled here, and summarized by the following
\begin{theorem}\label{thm:main}
Every mean-convex Alexandrov embedded finite type \cmc cylinder in $\Sp^3$ is rotational. Its bounding 3-manifold is diffeomorphic to $\overline{\mathbb{D}}\times\R$.
\end{theorem}
Here $\mathbb{D}$ denotes the open unit disk in the plane; the proof is given in Section~\ref{sec:proof main theorem}. This theorem yield alternative proofs of the Lawson conjecture and the Pinkall-Sterling conjecture. In fact, let $\imm:N\to\Sp^3$ be a mean-convex Alexandrov embedded \cmc torus $M=\partial N\simeq\mathbb{T}^2$. Due to \cite[Theorem~1]{Law:unknot} the homomorphism $\pi_1(M)\to\pi_1(N)$ is surjective and due to \cite[Theorem 18.1]{pap} the boundary $\partial\tilde{N}$ of the universal covering $\tilde{N}$ of $N$ is a cylinder. Hence there exists a mean-convex Alexandrov embedded \cmc cylinder $\tilde{\imm}:\tilde{N}\to\Sp^3$, which is the composition of $\imm$ with a covering map. Due to Pinkall-Sterling \cite{PinS} $\imm$ is of finite type. Then $\tilde{\imm}$ is also of finite type and by Theorem~\ref{thm:main} a surface of revolution. This implies
\begin{corollary}
All mean-convex Alexandrov embedded \cmc tori in $\Sp^3$ are tori of revolution.
\end{corollary}
The techniques involved in obtaining our classification are very different from the methods of Brendle, and Andrews-Li. To explain our approach first recall that Pinkall-Sterling \cite{PinS}, and independently Hitchin \cite{Hit:tor} describe \cmc tori in terms of a real algebraic curve $\Sigma$ and a holomorphic line bundle $L$ on $\Sigma$. This algebro-geometric correspondence between pairs $(\Sigma,\,L)$ and geometric objects has been very useful for the construction of examples. We here strive to display, how these methods can also help to classify the corresponding geometric objects, focusing on Alexandrov embedded \cmc tori in $\Sp^3$.

Consider a doubly periodic conformal immersion $\imm: \mathbb{C} \to \Sp^3$ with constant mean curvature $H$. The Hopf differential is holomorphic, and hence without zeroes on a torus. The exponent $2\omega$ in the induced metric $ds^2=\tfrac{1}{4(H^2+1)}e^{2 \omega} dz \otimes d\bar{z}$ satisfies the sinh-Gordon equation
\[
  \Delta\,\omega + \sinh(\omega) \cosh(\omega)= 0\,.
\]
Infinitesimal variations that preserve the \cmc condition require that in the normal direction the variation is prescribed by a function $u$ in the kernel of the linearized sinh-Gordon equation
\[
   \mathcal{L} u =\Delta u+ u\,\cosh(2 \omega) = 0\,.
\]
Pinkall-Sterling provide an iteration to obtain an infinite hierarchy of solutions $u_1,\,u_2,...$ of the linearized sinh-Gordon equation $\mathcal{L}u_n = 0$. Applying this iteration to the trivial solution $u_{0} \equiv 0$ yields the sequence of Jacobi fields
\[
u_{0}=0\,,\,
u_{1}= \omega_z \,,\,
u_{2}= \omega_{zzz}-2\omega_z^3\,,\,
u_{3}= \omega_{zzzzz}-10 \omega_{zzz}\omega_{z}^3-10 \omega^2_{zz} \omega_{z} + 6\omega_{z}^5\,\ldots\,.
\]
The remarkable fact proven by Pinkall-Sterling is that for a doubly-periodic solution $\omega$ of the sinh-Gordon equation this iteration becomes stationary in the sense that there exists an integer $g\geq 0$ and complex constants $a_1,\,\ldots a_g$ such that
\[
    u_{g+1} = a_1 \,u_1 + \ldots + a_g \,u_g\,.
\]
This leads us to the definition of finite type: A solution $\omega$ of the sinh-Gordon equation is of finite type if $u_{g+1}$ is for some $g\ge 0$ a complex linear combination of $u_1,\ldots,u_g$. In particular, the solution $\omega$ is of finite type, if the kernel of ${\mathcal L}$ is finite dimensional. Thus the result by Pinkall-Sterling can be rephrased into the statement that all doubly-periodic real solutions of the sinh-Gordon equation are of finite type. In particular all \cmc tori are of finite type.

Suppose the image of a doubly periodic immersion restricted to a period parallelogram is a \cmc torus. The image of a strip obtained by extending one of the generators is then an infinite covering of the \cmc torus by a \cmc cylinder. Such \cmc cylinders are also of finite type, but the class of \cmc cylinders of finite type is much larger than such covers of \cmc tori. A conformally immersed \cmc cylinder is of finite type if there exists a conformal  parametrisation with constant Hopf-differential, and its conformal factor is a singly-periodic solution of the sinh-Gordon equation of finite type. When the curvature is uniformly bounded then the function $\omega$ is uniformly bounded. A theorem of Mazzeo and Pacard \cite{MPR} states that an elliptic operator $\mathcal{L} u=\Delta u +q u$ on a cylinder $\Sp^1 \times \R$ has for bounded continuous function $q$ a finite dimensional kernel in the space of uniformly bounded $\mathrm{C}^2$-functions on $\Sp^1 \times \R$. Hence for \cmc cylinders of bounded curvature with constant Hopf differential the metric is of finite type in the meaning of Pinkall-Sterling. Equivalently we thus have that a conformally parameterised \cmc cylinder with constant Hopf-differential and bounded curvature is of finite type.

Let $\ann$ denote the set of such \cmc cylinders in $\Sp^3$ of finite type up to ambient isometries. Our main objective is the investigation of this space $\ann$. To obtain nice parameters for $\ann$ we invoke the algebro-geometric correspondence, which identifies a geometric object with an algebraic one. If the geometric object is a finite type \cmc cylinder $\imm \in \ann$, then the finite sequence $u_1,\,\ldots,\,u_d$ and the constants $a_1,\,\ldots,\,a_d$ above are encoded into a $2 \times 2$ matrix of polynomials, called a polynomial Killing field. A polynomial Killing field $\zeta$ is real analytic in $z$ and polynomial in $\lambda$. It solves a Lax equation, and its determinant $\lambda \mapsto-\lambda^{-1}a(\lambda)$ is independent of $z$. Thus we associate with a cylinder $\imm\in\ann$ a complex polynomial $a\in\C^{2g}[\lambda]$ of degree $2g$. When all the roots of $a$ are distinct, it in turn defines a hyperelliptic Riemann surface $\Sigma$ of genus $g$, and is called the spectral curve of $\imm$. The genus $g$ of $\Sigma$ is called the spectral genus. For $g=0$, the corresponding cylinders are homogeneous, while for $g=1$, the corresponding cylinders are in the family of Delaunay associates. The curve $\Sigma$ does not uniquely determine the cylinder, but there exists a finite dimensional compact subset of $\ann$, all whose elements $\imm$ have the same spectral curve. We call this set the isospectral set.

Arguments of Hitchin~\cite{Hit:tor} encode the non-trivial topology of $\imm\in\ann$ in a meromorphic 1-form $dh$ on $\Sigma$. For given $\Sigma$ with polynomial $a$ the differential $dh$ is described by another polynomial $b$. Furthermore two unimodular complex numbers $\lambda_1$ and $\lambda_2$ parameterise the mean curvature and the Hopf-differential. We call quadruples $(a,\,b,\,\lambda_1,\,\lambda_2)$ spectral data. For given spectral data $(a,b,\lambda_1,\lambda_2)$ the corresponding set of cylinders $\imm\in\ann$ build the isospectral set $\ann(a,b,\lambda_1,\lambda_2)$.

Since coverings of \cmc tori are not embedded cylinders in $\Sp^3$, we relax the notion of embeddedness to the weaker notion of mean-convex Alexandrov embeddedness. The corresponding spectral data turn out to be stable under continuous deformations. The spectral curves of homogeneous cylinders can effectively be classified by explicit calculation. In fact we determine all spectral data of spectral genus zero corresponding to mean-convex Alexandrov embedded cylinders. Moreover, we extend the corresponding family to a two-dimensional family of spectral genus at most one. The corresponding \cmc cylinders are rotational mean-convex Alexandrov embedded cylinders. We call this family the rotational family.

The Whitham equations define vector fields on the space of quadruples $(a,\,b,\,\lambda_1,\,\lambda_2)$ whose integral curves define deformation families in $\ann$. The corresponding flows simultaneously deform $\Sigma$, $dh$ and the points $\lambda_1$ and $\lambda_2$. Along this flow the curvature can blow up. It turns out that the curvature stays uniformly bounded on the cylinder as long as the roots of $a$ stay away from the poles of $dh$. A typical limit of a curvature blow up is a chain of spheres touching each other at the limiting points of shrinking necks. Another accident of the flow are coalescing roots of $a$. The limits are higher order roots of $a$ and called singularities of $\Sigma$, since the corresponding algebraic curve is not any more a complex manifold. The dimension of the isospectral set of $\Sigma$ is the spectral genus, so equal to half the number of roots of $a$. It turns out that in case of roots of $a$ coalescing on $\Sp^1$ one dimension of the isospectral set shrinks to a point. In this case the limit of the isospectral sets coincides with the isospectral set of the desingularised curve $\Sigma$ and there is no geometric accident. In case of roots of $a$ coalescing at points away from $\Sp^1$ the limit of the isospectral sets is a union of the compact isospectral set of the desingularised curve and an extra higher-dimensional non-compact part. The lower-dimensional compact part is the closure of the higher-dimensional non-compact part and the union is still compact. The movement from the isospectral set of the desingularised curve to the extra non-compact part drastically changes the corresponding geometric cylinder. In this case there is a geometric accident.

As an application of the deformation of spectral data we establish global properties of the moduli space. Using the non-compactness of cylinders instead of compact tori gives more space for deformations. We are able to construct in Theorem~\ref{t1} for all spectral data of $\imm\in\ann$ a path in the moduli space, which starts at the given spectral data and ends at spectral data of spectral genus zero. Along this path the curvature stays bounded and no geometric accidents happen. In Theorem~\ref{isolated} we show that the rotational family is connected with spectral data outside of this family only by a movement from the lower-dimensional compact part into the higher-dimensional extra non-compact part of a singular spectral curve whose desingularised curve belongs to the rotational family. Finally we classify mean-convex Alexandrov embedded \cmc cylinders of finite type in $\Sp^3$ - they are exactly the rotational cylinders.

Let us briefly describe the organization of the paper. The first four sections recall the integrable systems theory of the construction of \cmc cylinders immersed in $\Sp^3$. In particular in Section~\ref{sec:CONF} we encounter the sinh-Gordon equation, the notion of extended frame and the Sym-Bobenko formula. In Section~\ref{sec:finite type} we describe the algebro-geometric correspondence for finite type cylinders. Afterwards we characterize in Section~\ref{sec:spectral data} those agebraic curves, which are spactral curves of finite type cylinders. In Section~\ref{sec:moduli} we intruduce the corresponding moduli spaces. We then turn to an example in Section~\ref{sec:rotational} and describe the spectral data of rotational cylinders. Section~\ref{sec:AE spectral data} considers mean-convex Alexandrov embedded cylinders. The application of a result of \cite{HKS4} yields that mean-convex Alexandrov embeddedness is preserved firstly under isospectral deformations (Proposition~\ref{embedded isospectral}) and secondly under continuous deformations of spectral data (Proposition~\ref{embedded deformation}). This leads in Theorem~\ref{characterization} to a characterization of the moduli space of mean-convex Alexandrov embedded \cmc cylinders of finite type by three properties. The purpose of section~\ref{sec:deformation of cylinders} is to adapt the isoperiodic deformation to our setting. We develop general tools to deform spectral data. In section~\ref{sec:pathconnected} we construct continuous paths from spectral data of arbitrary \cmc cylinders of finite type to spectral data of genus zero. In section~\ref{sec:isolated} we prove an isolated property of the rotational family. In the final section we show that the spectral data of mean-convex Alexandrov embedded \cmc cylinders is that of the rotational cylinders, which proves Theorem~\ref{thm:main}.
%
%%%%%%%%%%%%%%%%%%%%%%%%%%%%%%%%%%%%%%%%%%%%%%%%%%%%%%%%%%%%%%%%%
%%%%%%%%%%%%%%%%%%%%%%%%%%%%%%%%%%%%%%%%%%%%%%%%%%%%%%%%%%%%%%%%%
%
%\setcounter{section}{-1}
%\section{Statement of the main Results with Applications}
%\label{sec:proof}
%
%%%%%%%%%%%%%%%%%%%%%%%%%%%%%%%%%%%%%%%%%%%%%%%%%%%%%%%%%%%%%%%%%
%%%%%%%%%%%%%%%%%%%%%%%%%%%%%%%%%%%%%%%%%%%%%%%%%%%%%%%%%%%%%%%%%
%
\section{Conformal cmc immersions into $\Sp^3$}
\label{sec:CONF}
This section recalls the relationship between \cmc
immersed surfaces in $\Sp^3$ and solutions of the $\sinh$-Gordon
equation. Furthermore, in the special case of \cmc cylinders in $\Sp^3$ we introduce the notion of monodromy and the period problem.
%
%%%%%%%%%%%%%%%%%%%%%%%%%%%%%%%%%%%%%%%%%%%%%%%%%%%%%%%%%%%%%%%%%%
%%%%%%%%%%%%%%%%%%%%%%%%%%%%%%%%%%%%%%%%%%%%%%%%%%%%%%%%%%%%%%%%%%
%
\subsection{The $\sinh$-Gordon equation}
We identify the 3-sphere $\Sp^3 \subset \R^4$ with $\Sp^3 \cong
\SU$. The Lie algebra of the matrix Lie group $\SU$ is $\su$,
equipped with the commutator $[\,\cdot,\,\cdot\,]$.
We denote by $\langle \cdot , \cdot \rangle$ the bilinear
extension of the Ad--invariant inner product $(X,\,Y) \mapsto
-\tfrac{1}{2} \tr(XY)$ of $\su$ to $\su^{\mbox{\tiny{$\C$}}} =
\Sl(\C)$. For a conformal map $\imm:\C\to\Sp^3$ let $v^2dzd\bar{z}$ denote
the induces metric, $Qdz^2$ the Hopf differential and $H$ the mean
curvature. The double cover of the isometry group $\mathrm{SO}(4)$ is
$\SU \times \SU$ via the action $((F,G),X) \mapsto FXG^{-1}$. We define
maps $F,G:\C\to\SU$ such that
\begin{align*}
\imm&=FG^{-1}&
\imm_z&=F
\left(\begin{smallmatrix}0&v\\0&0\end{smallmatrix}\right)G^{-1}&
\imm_{\bar{z}}&=F
\left(\begin{smallmatrix}0&0\\v&0\end{smallmatrix}\right)G^{-1}&
N&=F\left(\begin{smallmatrix}\mi&0\\0&-\mi\end{smallmatrix}\right)G^{-1}.
\end{align*}
Here $N:\C\to\su\subset\R^4$ denotes the normal of $\imm$.
The maps $F,G$ are unique up to the transformation
$(F,G)\mapsto(-F,-G)$. A direct calculation shows that the
corresponding $\su$-valued 1-forms $F^{-1}dF$ and $G^{-1}dG$ can be
calculated in terms of $v$, $Q$, $H$ and their derivatives. The
Gau{\ss}-Codazzi equations are the integrability conditions of these 1-forms.

If $\imm$ has constant mean curvature, then the Hopf differential is a
holomorphic quadratic differential \cite{Hop}. On the cylinder
$\C^\ast$ there is an infinite dimensional space of holomorphic
quadratic differentials, large classes of which can be realized as
Hopf differentials of \cmc cylinders \cite{KilMS}. On a
\cmc torus the Hopf differential is constant (and non-zero).
Since we are ultimately interested in tori, we restrict our
attention to \cmc cylinders considered via
Proposition~\ref{thm:sinh} which have constant non-zero Hopf
differentials on the universal covering $\C$ of $\C^\ast$.

For such \cmc immersions we may define a $\su$-valued 1-form
$\alpha_0$ and two $\Sl$-valued 1-forms $\alpha_\pm$ on $\C$ and two
{\bf{Sym points}} $\lambda_1\not=\lambda_2\in\Sp^1$ such that
(compare \cite[Theorem~14.1]{Bob:cmc})
\begin{gather}\begin{aligned}\nonumber
F^{-1}dF&=\alpha_{\lambda_1}&\hspace{20mm}
G^{-1}dG&=\alpha_{\lambda_2}&\mbox{with}
\end{aligned}\\\label{eq:general_alpha}
\alpha_\lambda=\lambda^{-1}\alpha_-+\alpha_0+\lambda\alpha_+=
\frac{1}{4}\,\begin{pmatrix}
2\omega_z\,dz-2\omega_{\bar{z}}\,d\bar{z} &
\mi\,\lambda^{-1}e^\omega\,dz + \mi\,e^{-\omega}\,d\bar{z}\\
\mi\,e^{-\omega}\,dz + \mi\,\lambda\,e^\omega\,d\bar{z}&
-2\omega_z\,dz+2\omega_{\bar{z}}\,d\bar{z}
\end{pmatrix}\,.
\end{gather}
Here $\omega$ is a solution of the sinh-Gordon equation. Moreover, the 1-form $\alpha_\lambda$ solves the Maurer-Cartan equation $2\,d\alpha_\lambda + [\,\alpha_\lambda \wedge \alpha_\lambda \,]=0$ for all $\lambda\in\C^\ast=\C\setminus\{0\}$ and takes for $\lambda\in\Sp^1$ values in $\su$. We can integrate $\alpha_\lambda$ and obtain the corresponding {\bf{extended frame}} $F_\lambda$:
\begin{align}\label{eq:extended frame}
dF_\lambda&=F_\lambda\,\alpha_\lambda&\mbox{with}&&
F_\lambda(0)&=\mathbbm{1}.
\end{align}
For given $\alpha_\lambda$ the map $\lambda \mapsto F_\lambda$ is holomorphic on $\C^\ast$ and has essential singularities at $\lambda = 0,\,\infty$.
%Here $\omega:\C\to\R$ is a solution of the $\sinh$-Gordon equation
%
%The corresponding mean curvature $H$, Hopf differential $Qdz^2$ are as in \eqref{eq:H_mu}, and the conformal factor $v$ is
%
Putting all this together we have the following version of
\cite[Theorem~14.1]{Bob:cmc}
\begin{proposition} \cite{Bob:cmc} \label{thm:sinh}
Let $\omega:\C\to\R$ be a solution of the sinh-Gordon equation
\begin{equation}\label{eq:sinh-Gordon}
  8\omega_{z\bar{z}}+\sinh (2\omega) = 0%change of coordinate
\end{equation}
and let $\lambda_1\ne\lambda_2\in\Sp^1$ be a pair of {\bf{Sym points}}. The corresponding $\alpha_\lambda$ \eqref{eq:general_alpha} solves the Maurer-Cartan equation and defines an {\bf{extended frame}} $F_\lambda$~\eqref{eq:extended frame}. The {\bf{Sym-Bobenko-formula}}
\begin{align}\label{eq:Sym_S3}
\imm&:\C\to\Sp^3&z\mapsto\imm(z)&=F_{\lambda_1}(z)F_{\lambda_2}^{-1}(z).
\end{align}
defines a conformal \cmc immersion with induced metric, mean curvature and Hopf differential
\begin{gather}
\label{eq:conformal factor}
v^2dzd\bar{z}=2\langle \imm^{-1}\imm_z,\imm^{-1}\imm_{\bar{z}}\rangle dzd\bar{z}=
(\lambda_1^{-1}-\lambda_2^{-1})(\lambda_1-\lambda_2)\frac{e^{2\omega}}{8}dzd\bar{z}=\frac{e^{2\omega}}{4(H^2+1)} dzd\bar{z}\\
\begin{aligned}\label{eq:H_mu}
  H&=\mi\frac{\lambda_1+\lambda_2}{\lambda_2 -\lambda_1}&\hspace{25mm}
Qdz^2&=\tfrac{\mi}{16}(\lambda_1^{-1}-\lambda_2^{-1})dz^2.
\end{aligned}
\end{gather}
Conversely, given a conformal \cmc immersion $\imm:\C\to\Sp^3$ with constant Hopf differential there exist a unique solution $\omega$ of~\eqref{eq:sinh-Gordon}, $c>0$ and $\lambda_1\not=\lambda_2\in\Sp^1$, such that the immersion $z\mapsto\imm(cz)$ is up to ambient isometry equal to~\eqref{eq:Sym_S3} and has $v^2dzd\bar{z}$, $H$ and $Qdz^2$ as in~\eqref{eq:conformal factor}-\eqref{eq:H_mu}.
\end{proposition}
%
%
%%%%%%%%%%%%%%%%%%%%%%%%%%%%%%%%%%%%%%%%%%%%%%%%%%%%%%%%%%%%%%
%%%%%%%%%%%%%%%%%%%%%%%%%%%%%%%%%%%%%%%%%%%%%%%%%%%%%%%%%%%%%%
%
\subsection{Monodromy and Periodicity condition}
Let $F_\lambda$ be an extended frame or a \cmc immersion
$\imm:\C \to \Sp^3$ such that \eqref{eq:Sym_S3} holds for two
distinct unimodular numbers $\lambda_1,\,\lambda_2$. Let
$\tau:\C\to\C,\,z\mapsto z+\tau$ be a translation, and assume that
$\alpha_\lambda = F_\lambda^{-1}dF_\lambda$ has period $\tau$, so that $\tau^{\ast} \alpha_\lambda = \alpha_\lambda \circ \tau = \alpha_\lambda$. Then we define the {\bf{monodromy}} of $F_\lambda$ with respect to $\tau$ as
\begin{equation}\label{eq:monodromy}
    M_\lambda(\tau) = \tau^{\ast}(F_\lambda)\,F_\lambda ^{-1}\,.
\end{equation}
Periodicity $\tau^{\ast}\imm=\imm$ in terms of the monodromy is then
$\tau^{\ast}\imm= M_{\lambda_1}(\tau)\,\imm\, M_{\lambda_2}^{-1}(\tau)$. Hence $\tau^{\ast}\imm=\imm$
holds for surfaces not contained in the fixed point set of an isometry
if and only if
\begin{equation}\label{eq:periodicity}
    M_{\lambda_1}(\tau) = M_{\lambda_2}(\tau) = \pm \mathbbm{1}\,.
\end{equation}
%
%
%%%%%%%%%%%%%%%%%%%%%%%%%%%%%%%%%%%%%%%%%%%%%%%%%%%%%%%%%%%%%%%
%%%%%%%%%%%%%%%%%%%%%%%%%%%%%%%%%%%%%%%%%%%%%%%%%%%%%%%%%%%%%%%
%%%%%%%%%%%%%%%%%%%%%%%%%%%%%%%%%%%%%%%%%%%%%%%%%%%%%%%%%%%%%%%
%
\section{Polynomial Killing fields}
\label{sec:finite type}
In this section we discuss the solutions of the $\sinh$-Gordon equation which are called finite type solutions. These finite type solutions of the $\sinh$-Gordon equation are in one-to-one correspondence with maps called polynomial Killing fields from $\C$ to $2\times2$-matrix valued polynomials. These polynomial Killing fields themselves solve a non-linear partial differential equation. They are uniquely determined by one of their values. We shall call these values potentials and parameterise the space of finite type solutions by these potentials. The Symes method calculates the solutions in terms of the potentials with the help of a loop group splitting. The eigenvalues of these matrix polynomials define a real hyperelliptic curve. One spectral curve corresponds to a whole family of finite type solutions of the
$\sinh$-Gordon equation. We call the sets of finite type solutions (or their potentials), which belong to the same spectral curve, isospectral sets.
%
%%%%%%%%%%%%%%%%%%%%%%%%%%%%%%%%%%%%%%%%%%%%%%%%%%%%%%%%%%%%%%%
%%%%%%%%%%%%%%%%%%%%%%%%%%%%%%%%%%%%%%%%%%%%%%%%%%%%%%%%%%%%%%%
%
\subsection{Polynomial Killing fields} \label{sec:killing}
For some aspects of the theory untwisted loops are advantageous, and
avoiding the additional covering map $\lambda \mapsto \sqrt\lambda$
simplifies for example the description of Bianchi-B\"acklund
transformations by the simple factors \cite{TerU, KilSS}. For the
description of polynomial Killing fields on the other hand, the
twisted loop algebras as in
\cite{BurFPP,BurP_adl,BurP:dre,DorPW,McI:tor} are better suited, but
we remain consistent and continue working in our `untwisted'
setting.
\begin{definition}\label{def:finite type}
A solution $\omega$ of the $\sinh$-Gordon equation~\eqref{eq:sinh-Gordon}
is called a {\bf{finite type}} solution if and only if for the corresponding
$\alpha_\lambda$ in \eqref{eq:general_alpha}, the Lax equation
$$
    d\Phi=[\Phi,\alpha_\lambda]
$$
has a polynomial solution
$\lambda \mapsto \Phi = \Phi(z,\,\lambda) =\sum_{j=1}^d \lambda^d\Phi_j(z)$ with smooth $\Phi_j:\C\to\Sl(\C)$.
\end{definition}
\begin{definition}
A conformal \cmc immersion $\imm:\C^\ast\to\Sp^3$
%from the parabolic cylinder to the three-dimensional round sphere
is of finite type, if the Hopf differential is constant with respect to a conformal parameter $z\in\C/\tau\Z\simeq\C^\ast$ with $\tau\in\C^\ast$ and if it corresponds to a finite type solution $\omega$. Dividing out the ambient isometries $\mathrm{iso}(\mathbb{S}^3)$ we define
\[
    \ann = \{ \mbox{ \cmc-cylinders in $\Sp^3$ of
finite type } \} \,\slash \,\mathrm{iso}(\mathbb{S}^3)\,.
\]
\end{definition}
Finite type solutions of the sinh-Gordon equation give rise to algebraic objects in a finite dimensional vector space which we call potentials (compare \cite[Section~2]{HKS1}).
\begin{definition}\label{def pot}
For $g \in \N\cup\{0\}$ define the real $3g+2$-dimensional vector space
$$
    \barpot{g} = \left\{
    \xi=\sum\nolimits_{d=-1}^{g}\xi_d\lambda^d
    \mid\xi_{-1}\in\C\bigl( \begin{smallmatrix} 0 & 1 \\ 0 & 0
    \end{smallmatrix} \bigr),\,%\tr(\xi_0 A)\in \R,\,
    \xi_d=-\bar{\xi}^T_{g-1-d}\in\Sl(\C)\mbox{ for }d=-1,\ldots,g
    \right\}.
$$
\end{definition}
Clearly $\barpot{g}$ has up to isomorphism a unique norm $\|\cdot\|$. These Laurent polynomials define smooth mappings $\lambda\in\Sp^1\to\Sl(\C)$. Note that $\sqrt{\lambda}\mapsto\lambda^{\frac{1-g}{2}}\xi$ belongs to the loop Lie algebra $\Lsu$ of the loop Lie group $\LSU$. An element $\xi\in\barpot{g}$ is called {\bf{potential}}. Expanding a map $\zeta: \C \to \barpot{g}$ as
\begin{equation} \label{eq:zeta-expansion}
    \zeta (z) = \begin{pmatrix} 0 & \beta_{-1}(z)  \\  0 & 0 \end{pmatrix} \lambda^{-1} + \begin{pmatrix} \alpha_0 (z)
     & \beta_0  (z) \\ \gamma_0 (z) & -\alpha_0 (z)  \end{pmatrix} \lambda^0 + \ldots + \begin{pmatrix} \alpha_g (z)  & \beta_g (z)  \\ \gamma_g (z)  & -\alpha_g (z)  \end{pmatrix} \lambda^g
\end{equation}
we associate a matrix 1-form defined by
\begin{equation} \label{eq:alpha-zeta}
    \alpha(\zeta)=\begin{pmatrix}
    \alpha_0 (z)  & \beta_{-1}(z) \lambda^{-1}  \cr
    \gamma_0 (z)  & -\alpha_0 (z)
    \end{pmatrix} {\rm d}z -\begin{pmatrix}
    \bar{\alpha}_0 (z)  & \bar{\gamma}_0 (z)   \cr
    \bar{\beta}_{-1} (z)  \lambda & - \bar{\alpha}_0 (z)
    \end{pmatrix} {\rm d} \bar{z}.
\end{equation}
\begin{definition}
A {\bf{polynomial Killing field}} $\zeta:\C\to\barpot{g}$ is a solution of the Lax equation
\begin{equation} \label{eq:pKf}
    d\zeta =[\,\zeta,\,\alpha(\zeta)\,]\quad\mbox{ with }\quad
    \zeta(0)=\xi.
\end{equation}
\end{definition}
Such polynomial Killing fields give rise to solutions of the sinh-Gordon equation if the potentials $\xi\in\pot{g}$ obey the conditions $\xi_1\ne0$ and $\tr(\xi_{-1}\xi_0)\ne0$. By multiplying $\xi$ with positive numbers and by rotating $\lambda\mapsto e^{\mi\varphi}\lambda$ we may achieve $\tr(\xi_{-1}\xi_0)=-\frac{1}{16}$. These transformations induce a reparameterisation of the corresponding $\zeta$.  After conjugation by constant diagonal matrices in $\SU$ such potentials $\xi$ stay together with $\zeta$~\eqref{eq:pKf} in (see \cite[Remark~3.3]{HKS1})
\begin{align}\label{def:pot}
    \pot{g}=\left\{\xi\in\barpot{g}\mid
    \xi_{-1}\in\R^+\bigl(\begin{smallmatrix} 0 & \mi \\ 0 & 0
  \end{smallmatrix}\bigr)\mbox{ and }\tr(\xi_{-1}\xi_0)\ne 0\right\}.
\end{align}
The transformed $\alpha(\zeta)$ is of the form $\alpha_\lambda$~\eqref{eq:general_alpha} for a finite type solution $\omega$ of~\eqref{eq:sinh-Gordon}.

By the Symes method \cite{BurP_adl}, see also \cite[Proposition~3.2]{HKS1}, the extended framing $F_\lambda:\C\to\Lambda\SU$ of a \cmc immersion of finite type is given by the unitary factor of the Iwasawa decomposition
\begin{equation} \label{eq:FB}
  \exp(z\,\xi) = F_\lambda\,B
\end{equation}
for some $\xi\in\pot{g}$ with $g\in\N\cup\{0\}$. Due to
Pressley-Segal \cite{PreS}, the Iwasawa decomposition is a
diffeomorphism between the loop group $\LSL(\C)$ of $\SL(\C)$ into point-wise products of elements of $\LSU$ with elements of the loop group $\Lambda^+\SL(\C)$ of holomorphic maps from $\lambda\in\mathbb{D}$ to $\SL(\C)$, which take at $\lambda=0$ values in the subgroup of $\SL(\C)$ of upper-triangular matrices with positive real diagonal entries. If $\alpha(\xi)$ denotes for every constant map $z\mapsto\xi\in\pot{g}$ the 1-form~\eqref{eq:alpha-zeta} on $\C$, then the 1-form $\xi dz-\alpha(\xi)$ takes values in the Lie algebra of $\Lambda^+\SL(\C)$ of the right hand factor in the Iwasawa decomposition~\eqref{eq:FB}.

For each potential $\xi\in\pot{g}$, there
exists a unique polynomial Killing field given by
\begin{equation}\label{eq:solution pk}
\zeta=B\xi B^{-1}=F^{-1}_\lambda\xi F_\lambda\quad
\mbox{ with }F_\lambda\mbox{ and }B\mbox{ as in \eqref{eq:FB}.}
\end{equation}
Thus $\det \zeta (z) = \det \xi $ does not depend on the variable $z$. For a potential $\xi=\zeta(0)$ the polynomial $a(\lambda)=-\lambda\det \xi $ defines the spectral curve. For $\xi\in\pot{g}$ with $\tr(\xi_{-1}\xi_0)=-\frac{1}{16}$ the corresponding 1-form $\alpha(\zeta)$ is the $\alpha_\lambda$ in \eqref{eq:general_alpha} for that particular solution $\omega$ of the $\sinh$-Gordon equation corresponding to the extended frame $F_\lambda$ of \eqref{eq:FB}. For general $\xi\in\barpot{g}$ with $\xi_{-1}\ne0$ and $\tr(\xi_1\xi_0)\ne0$ the corresponding $\alpha(\zeta)$ differs from \eqref{eq:general_alpha} by multiplication of $\lambda$ and $dz$ with constant non-zero complex numbers. Given a polynomial Killing field $\zeta$, we set the potential $\xi=\zeta|_{z=0}$ in \eqref{eq:FB}. Thus $\zeta$, or the potential $\xi$, gives rise to an extended frame, and thus to an isometric family of finite type \cmc surfaces in $\Sp^3$.
%
%%%%%%%%%%%%%%%%%%%%%%%%%%%%%%%%%%%%%%%%%%%%%%%%%%%%%%%%%%%%%%%%%%%%%%%%
%%%%%%%%%%%%%%%%%%%%%%%%%%%%%%%%%%%%%%%%%%%%%%%%%%%%%%%%%%%%%%%%%%%%%%%%
%
\subsection{Roots of polynomial Killing fields}
If a potential $\xi$ has a root at some
$\lambda=\alpha\in\C^\ast$, then the corresponding polynomial
Killing field has a root at the same $\lambda$ for all $z\in\C$. In
this case we may reduce the order of $\xi$ and $\zeta$ without
changing the corresponding $F_\lambda$ \eqref{eq:FB}. Let
\begin{align}\label{eq:transformation}
  p(\lambda)&=\begin{cases}
  \overline{\sqrt{-\alpha}}\lambda+\sqrt{-\alpha}
  &\mbox{ for }|\alpha|=1\\
  |\alpha|^{-1}(\bar{\alpha}\lambda-1)(\lambda-\alpha)&
  \mbox{ for }\alpha\in\C^\ast\setminus\Sp^1\end{cases}&
  \lambda^{\deg(p)}\bar{p}\left(\bar{\lambda}^{-1}\right)&=p(\lambda).
\end{align}
If the polynomial Killing field $\zeta$ with potential
$\xi\in\pot{g}$ has a simple root at $\lambda=\alpha\in\C^\ast$,
then $\zeta/p$ does not vanish at $\alpha$ and is the polynomial
Killing field with potential $\xi/p\in\barpot{g-\deg(p)}$. Furthermore,
obviously $\zeta$ and $\zeta/p$ commute, and we next
show that both polynomial Killing fields $\zeta$ and $\zeta/p$ give
rise to the same extended frame $F_\lambda$ \eqref{eq:FB}
(compare \cite[Proposition~4.4]{HKS1}).
\begin{proposition}\label{th:pKf_min}
If a polynomial Killing field $\zeta$ with potential $\xi\in\pot{g}$
has zeroes in $\lambda\in\C^\ast$, then there is a
polynomial $p(\lambda)$, such that the following two conditions
hold:
\begin{enumerate}
    \item $\zeta/p$ is the polynomial Killing field with potential
    $\xi/p\in\barpot{g-\deg p}$, which gives rise to the same
    associated family as $\zeta$.
    \item $\zeta/p$ has no zeroes in $\lambda \in \C^\ast$.
\end{enumerate}
\end{proposition}
\begin{proof}
An appropriate rotation of $\lambda$ transforms
any root $\alpha\in\C^\ast$ into a negative root. For such
negative roots the corresponding potentials $\xi$ and $\xi/p$
are related by multiplication with a polynomial with respect to
$\lambda$ with positive coefficients. In the Iwasawa decomposition
\eqref{eq:FB} this factor is absorbed in $B$. Hence the
corresponding extended frames coincide.
\end{proof}
Hence amongst all polynomial Killing fields that give rise to a
particular \cmc surface of finite type there is one of
smallest possible degree (without adding further poles), and we say
that such a polynomial Killing field has \emph{minimal degree}. A
polynomial Killing field has minimal degree if and only if it has
neither roots nor poles in $\lambda \in \C^\ast$. We recall three
statements of \cite{HKS1}. The first part restates
\cite[Theorem~2.4]{HKS1}, the second part follows immediately from
Section~2 in \cite{HKS1}, and the third part is a variant of
\cite[Proposition~4.5]{HKS1}.
\begin{proposition}\label{thm:PKF}
(i) A periodic bounded solution $\omega$ of \eqref{eq:sinh-Gordon} is of finite type. Thus the finite type immersions $\imm:\C\to\Sp^3$ are characterized by properties having bounded curvature and constant Hopf differential with respect to a conformal parameter $z\in\C/\tau\Z\simeq\C^\ast$ with $\tau\in\C^\ast$.

(ii) A solution $\omega$ of \eqref{eq:sinh-Gordon} is of finite type, if and
only if there exists $\xi\in\pot{g}$ with $g\in\N\cup\{0\}$ and
$\tr(\xi_{-1}\xi_0)=-\frac{1}{16}$, whose polynomial Killing field has
the same $\alpha_\lambda$~\eqref{eq:general_alpha}.

(iii) There exists a unique polynomial Killing field of
minimal degree that gives rise to $\omega$. Thus every \cmc
immersion $\imm:\C\to\Sp^3$ of finite type is up to isometry determined
by a unique triple $(\xi,\lambda_1,\lambda_2)$ with $\xi\in\pot{g}$
without roots and $g\in\N\cup\{0\}$.
\end{proposition}
The map which associates to the data $(\xi,\lambda_1,\lambda_2)$ the
corresponding \cmc immersion is smooth.
%
%%%%%%%%%%%%%%%%%%%%%%%%%%%%%%%%%%%%%%%%%%%%%%%%%%%%%%%%%%%%%%%%%%%%%%%%%
%%%%%%%%%%%%%%%%%%%%%%%%%%%%%%%%%%%%%%%%%%%%%%%%%%%%%%%%%%%%%%%%%%%%%%%%%
%
\section{Spectral Data}\label{sec:spectral data}
\subsection{Spectral data I}
For each $\imm\in\ann$ let $\xi\in\pot{g}$ be the potential of minimal degree and define
\begin{align*}
a&\in\C^{2g}[\lambda]&&\mbox{with}&a&=-\lambda\det(\xi)
\end{align*}
as the polynomial of degree $2g$, which obeys the reality conditions
\begin{equation}\label{eq:a_reality}
    \lambda^{2g}\bar{a}(\bar{\lambda}^{-1}) = a(\lambda)
    \qquad \mbox{and} \qquad \lambda^{-g}a(\lambda) \le 0\mbox{ for all }\lambda\in\Sp^1.
\end{equation}
This polynomial defines a hyperelliptic curve with three involutions:
\begin{align}\label{eq:curve}
\Sigma^\ast&=\{(\lambda,\nu)\in\C^\ast\times\C\mid \nu^2=
\lambda^{-1}a(\lambda)\},\\
\sigma&:(\lambda,\nu)\mapsto(\lambda,-\nu),\qquad \label{eq:involutions}
\rho:(\lambda,\nu)\mapsto(\bar{\lambda}^{-1},\bar{\lambda}^{-g}\bar{\nu}),&
\eta&:(\lambda,\nu)\mapsto(\bar{\lambda}^{-1},-\bar{\lambda}^{-g}\bar{\nu}).
\end{align}
By construction we have a map $\lambda:\Sigma^* \to \C^\ast$ of degree 2, which is branched at the $2g$-pairwise distinct roots $\{ \alpha_1,\,\ldots \alpha_g,\,\bar{\alpha}_1^{-1},\,\ldots \bar{\alpha}_g^{-1}\}$ of the polynomial $a$. By declaring the points over $\lambda = 0,\,\infty$ to be two further branch points, we then have $2g+2$ branch points. This 2-point compactification $\Sigma$ is called the {\bf{spectral curve}}. The genus $g$ is called the {\bf{spectral genus}}. Due to \eqref{eq:solution pk} the zero set of the characteristic polynomial $\det(\nu\mathbbm{1}-\zeta)$ of the polynomial Killing field $\zeta$ with potential $\xi$ does not depend on $z\in\C$ and agrees with \eqref{eq:curve}.
The multiplication of $\xi$ and $\zeta$ with positive constants only changes the parameter $z$. We shall normalize $|a(0)|=\frac{1}{16}$.
%
%%%%%%%%%%%%%%%%%%% Ende der 7. Veraenderung
%%%%%%%%%%%%%%%%%%%%%%%%%%%%%%%%%%%%%%%%%%%%%%%%%%%%%%%%%%%%%%%%%%%%%
%%%%%%%%%%%%%%%%%%%%%%%%%%%%%%%%%%%%%%%%%%%%%%%%%%%%%%%%%%%%%%%%%%%%%
%
\subsection{Isospectral sets}
\label{sec:isospectral}
We recall the isospectral action and investigate the isospectral sets.
\begin{definition}{\bf(Isospectral action).}\label{groupaction}
Let $\xi\in\pot{g}$ and $t=(t_0,\ldots,t_{g-1}) \in \C^g$, and
\begin{equation*}
\exp\,\bigl( \,\xi \sum_{i=0}^{g-1}\lambda^{-i} t_i \, \bigr) = F_\lambda(t)B(t)
\end{equation*}
the Iwasawa factorisation. Define the map $\pi(t) :\pot{g}\to\pot{g}$ by
\begin{equation}\label{eq:isospectral action}
	\pi(t) \,\xi =B(t) \, \xi\, B^{-1}(t)= F_\lambda ^{-1} (t) \, \xi\, F_\lambda (t)
\end{equation}
\end{definition}
In \cite[Section~4]{HKS1} this cummutative group action is
investigated. It preserves the following sets:
\begin{definition}
For a polynomial $a$ of degree $2g$ obeying \eqref{eq:a_reality} the
isospectral set is defined as
\begin{equation*}%\label{eq:isospectral}
\iso{a}=\{\xi\in\pot{g}\mid\det\xi=-\lambda^{-1}a(\lambda)\}.
\end{equation*}
\end{definition}
Due to Proposition~\ref{thm:PKF}~(iii) \cmc immersions of finite type $\imm:\C\to\Sp^3$ are determined by triples $(\xi,\lambda_1,\lambda_2)$. The sets $\iso{a}$ build the fibres of the following map:
\begin{definition}\label{map a}
Let $\bigmoduli^{g}$ denote the space of triples
$(a,\lambda_1,\lambda_2)$ of $a\in\C^{2g}[\lambda]$ obeying
\eqref{eq:a_reality} together with two Sym points
$\lambda_1\not=\lambda_2\in\Sp^1$. Then we consider the following map:
\begin{align}\label{eq:map a}
A:\{(\xi,\lambda_1,\lambda_2)\in
\pot{g}\times\Sp^1\times\Sp^1\mid\lambda_1\not=\lambda_2\}&
\to\bigmoduli^g&
(\xi,\lambda_1,\lambda_2)&
\mapsto(-\lambda\det\xi,\lambda_1,\lambda_2).
\end{align}
\end{definition}
\begin{lemma}\label{open and proper}
The map $A$~\eqref{eq:map a} is open and proper.
\end{lemma}
\begin{proof}
The Laurent coefficients of $\xi = \sum_{d=-1}^g \lambda^d \xi_d $ are
$$\xi_d=\frac{1}{2\pi\mi}\int_{\Sp^1} \lambda ^{-d} \xi \,
\frac{d\lambda}{\lambda}\,.
$$
Using a norm yields
$$\|\xi_d\| \leq\frac{1}{2\pi\mi}\int_{\Sp^1} \|\lambda ^{-d}\xi \| \,
\frac{d\lambda}{\lambda} \leq
\sup_{\lambda \in\Sp^1}\sqrt{-\lambda^{-g}a(\lambda)}\,.
$$
Thus each entry $\xi_d$ of $\xi$ is bounded if
$\sqrt{-\lambda^{-g}a(\lambda)}$ is bounded on $\Sp^1$. For polynomials
$a$ obeying the reality conditions \eqref{eq:a_reality} this follows
from the roots of $a$ and $|a(0)|$ being bounded. Hence
$A^{-1}[\mathcal{K}]$ is a bounded subset of the real
finite-dimensional vector space
$\barpot{g}\times\C\times\C\supset\pot{g}\times\Sp^1\times\Sp^1$, if
$\mathcal{K}$ is a compact subset of $\bigmoduli^g$. By
continuity of $A$ it is also closed, and therefore compact.

Due to \cite[Proposition~4.12 and Theorem~6.8]{HKS1} the orbits of the
group action~\eqref{eq:isospectral action} are the subsets of
$\iso{a}$ of all elements $\xi$ with the same roots on $\C^\ast$
counted with multiplicities. For any polynomial $a$ of degree $2g$
obeying \eqref{eq:a_reality}, an off-diagonal potential
$$
    \xi= \begin{pmatrix}
    0 &\lambda^{-1}\beta(\lambda)\\
    \gamma(\lambda) & 0
    \end{pmatrix}
$$
belongs to $\iso{a}$, if and only if the
polynomials $\beta$ and $\gamma$ of degree $g$ obey
$\beta(\lambda)\gamma(\lambda)=a(\lambda)$ and
$\gamma(\lambda)=-\lambda^g\bar{\beta}(\bar{\lambda}^{-1})$.
The roots of $\beta$ are $g$ roots of $a$, which are mapped by
$\lambda\mapsto\bar{\lambda}^{-1}$ onto the remaining $g$ roots of
$a$. For any choice of such roots, $\beta$ and $\gamma$ are determined
up to multiplication by inverse unimodular numbers. At higher order
roots $\alpha\in\C^\ast\setminus\Sp^1$ of $a$ we can choose the
multiplicity of the root of $\beta$ at $\alpha$ between zero and the
multiplicity of the root $\alpha$ of $a$. The sum of the
multiplicities of the roots of $\beta$ at $\alpha$ and at
$\bar{\alpha}^{-1}$ has to be equal to the multiplicity of the root of
$a$ at $\alpha$. Therefore there exists in every orbit of the
isospectral group action~\eqref{eq:isospectral action} at least one
off-diagonal $\xi$. Due to the relation between the polynomials $a$
and $\beta$ and $\gamma$, the map $A$~\eqref{eq:map a} is open at
off-diagonal $\xi$. Since the isospectral
action~\eqref{eq:isospectral action} acts by
diffeomorphisms on $\pot{g}$ and preserves the fibres of the map $A$,
this map is (globally) open.
\end{proof}
As in \cite[Proposition~4.4]{HKS2} any bound on the coefficients of $a$ implies a uniform curvature bound and a bound on the covariant derivative of the second fundamental from:
\begin{proposition}\label{uniform bound}
Let $\mathcal{K}\subset\C^{2g}[\lambda]\times\Sp^1\times\Sp^1$ be a set of polynomials $a$ obeying the reality condition~\eqref{eq:a_reality} together with two Sym points $\lambda_1\ne\lambda_2$. If in $\mathcal{K}$ all roots of $a$ and $|a(0)|$ are uniformly bounded, and $\lambda_1$ and $\lambda_2$ are uniformly bounded away from each other, then all \cmc immersions $\imm:\C\to\Sp^3$ corresponding to $(\xi,\lambda_1,\lambda_2)\in A^{-1}[\mathcal{K}]$ have uniformly bounded $\mathrm{C}^{k,\alpha}$-norm for fixed $k,\alpha$. In particular all these immersions have uniformly bounded curvature and second fundamental forms with uniformly bounded covariant derivatives.
\end{proposition}
\begin{proof}
To prove this proposition it suffices to show a uniform $\mathrm{C}^{k,\alpha}$-bound on the corresponding solutions $\omega:\C\to\R$ of \eqref{eq:sinh-Gordon}. The closure of $\mathcal{K}$ in $\C^{2g}[\lambda]\times\Sp^1\times\Sp^1$ is a compact subset of $\bigmoduli^g$. Lemma~\ref{open and proper} gives a uniform bound on $\omega$. Schauder estimates bound the $\mathrm{C}^{k,\alpha}$-norms on $\C$.
\end{proof}
%
%%%%%%%%%%%%%%%%%%%%%%%%%%%%%%%%%%%%%%%%%%%%%%%%%%%%%%%%%%%%%%%%%%%%%
%%%%%%%%%%%%%%%%%%%%%%%%%%%%%%%%%%%%%%%%%%%%%%%%%%%%%%%%%%%%%%%%%%%%%
%
\subsection{Spectral data II}
We utilise the description of finite type \cmc surfaces in
$\Sp^3$ via spectral curves due to Hitchin \cite{Hit:tor}, and
relate to our previous definition of spectral curves due to Bobenko
\cite{Bob:cmc}. While Hitchin defines the spectral curve as the
characteristic equation for the holonomy of a loop of flat
connections, Bobenko defines the spectral curve as the
characteristic equation of a polynomial Killing field. We shall use
both of these descriptions, and briefly recall their equivalence:
Due to \eqref{eq:monodromy}, the monodromy $\C^\ast \to \SL
(\C),\,\lambda \mapsto M_\lambda$ is a holomorphic map with
essential singularities at $\lambda = 0,\,\infty$. By construction
the monodromy takes values in $\SU$ for $|\lambda |=1$. The
monodromy depends on the choice of base point, but its conjugacy
class and hence eigenvalues $\mu,\,\mu^{-1}$ do not.
With $\Delta(\lambda)=\tr(M_\lambda)$ the {\bf{spectral variety}} is given by
\begin{equation*} %\label{eq:characteristic1}
 \{(\lambda,\mu)\in\C^\ast\times\C^\ast\mid
\mu ^2-\Delta(\lambda) \,\mu + 1 =0\,\}.
\end{equation*}
The eigenspace of $M_\lambda$ depends holomorphically on
$(\lambda,\,\mu)$ and defines the {\bf{eigenbundle}} on the spectral
variety. Let us compare the previous definition of a spectral
curve~\eqref{eq:curve} of periodic finite type
solutions of the $\sinh$-Gordon equations. Let $\zeta$ be a
polynomial Killing field with potential $\xi\in\pot{g}$, with period
$\tau$ so that $\zeta(z+\tau) = \zeta(z)$ for all $z\in\C$. Then
also the corresponding $\alpha(\zeta)$ is $\tau$-periodic. Let
$dF_\lambda= F_\lambda \alpha(\zeta),\,F_\lambda(0)=\mathbbm{1}$ and
$M_\lambda(\tau) = F_\lambda(\tau)$ be the monodromy with respect to
$\tau$. Then for $z=0$ we have $\xi= \zeta(0) = \zeta(\tau) =
F_\lambda^{-1}(\tau) \,\xi\,F_\lambda(\tau) = M_\lambda^{-1}(\tau) \xi
\,M_\lambda(\tau)$ and thus
$$
    [\,M_\lambda(\tau),\,\xi\,] = 0\,.
$$
All eigenvalues of holomorphic $2\times 2$ matrix-valued functions
depending on $\lambda\in\CP$ and commuting point-wise with
$M_\lambda(\tau)$ or $\xi$ define the sheaf of holomorphic functions of
the spectral curve. Hence the eigenvalues of $\xi$ and $M_\lambda(\tau)$
are different functions on the same Riemann surface. Furthermore, on
this common spectral curve the eigenspaces of $M_\lambda(\tau)$ and $\xi$
coincide point-wise. Consequently the holomorphic eigenbundles of
$M_\lambda(\tau)$ and $\xi$ coincide.

Following Hitchin~\cite{Hit:tor} we encode the non-trivial topology
of $\imm\in\ann$ in a meromorphic differential 1-form $dh$ on $\Sigma$.
We write $dh$ in terms of a polynomial $b\in\C^{g+1}[\lambda]$ of degree $g+1$ such that
\begin{equation}\label{eq:b_def}
    \lambda^{g+1}\bar{b}(\bar{\lambda}^{-1}) = -b(\lambda) \qquad \mbox{and} \qquad
    dh =\frac{b(\lambda)d\lambda}{\nu\lambda^2}\,.
\end{equation}
The meromorphic differential $dh$ has only second order poles without
residues at $\lambda=0$ and $\lambda=\infty$. Moreover the exponent of the integral $e^h$ has to be a single-valued holomorphic function on $\Sigma^\ast$. Finally the two Sym points $\lambda_1\ne\lambda_2\in\Sp^1$ parameterise the mean curvature and the Hopf differential~\eqref{eq:H_mu}. Putting this all together we obtain the spectral data of $\imm\in\ann$:
\begin{definition}\label{spectral data}
For all spectral genera $g\in\N\cup\{0\}$ the spectral data of
$\imm\in\ann$ consists of complex polynomials $a$ and $b$ of degree $2g$
and $g+1$, and Sym points $\lambda_1\not=\lambda_2\in\Sp^1$ such that:
\begin{enumerate}
\item[(i)] Reality conditions \eqref{eq:a_reality} and \eqref{eq:b_def} hold.
\item[(ii)] ${\rm Re}\left( \int_{\alpha_i}^{1/\bar{\alpha}_i}\frac{b\,
  d\lambda}{\nu \lambda^2} \right) =0$ for all roots $\alpha_i$ of $a$
  where the integral is computed on the straight segment $[\alpha_i, 1/
  \bar{\alpha}_i]$.
\item[(iii)] Let $\gamma_i\subset\Sigma$ be the closed cycles over the
  straight segments $[\alpha_i, 1/ \bar{\alpha}_i]$ and
  $\tilde\Sigma=\Sigma\setminus\cup\gamma_i$. There exists a unique
  meromorphic function $h$ on $\tilde\Sigma$ such that
\begin{align*}
    dh&=\frac{b\,d\lambda}{\nu \lambda^2}&\mbox{and}&&
\sigma^{\ast} h(\lambda)&= -h(\lambda).
\end{align*}
This function continuously extends to boundary cycles over the
segments $[\alpha_i,\bar{\alpha}_i^{-1}]$ and is assumed to take at
all roots of $a$ values in $\pi\mi \Z$.
\item[(iv)] There exist holomorphic functions $f$ and $g$ depending on
$\lambda\in\C^\ast$ such that
\begin{align*}
\mu=e^h&=f\nu+g\quad\mbox{ on }\tilde\Sigma\setminus\{0,\infty\}&
f(\lambda_1)&=f(\lambda_2)=0&g(\lambda_1)&=g(\lambda_2)=\pm1.
\end{align*}
Moreover at all roots $\alpha_i$ of $a(\lambda)$ we have $g(\alpha_i) = \pm 1$.
\item[(v)] $|a(0)|=\frac{1}{16}$ %change of coordinate
and $\lambda_2=\lambda_1^{-1}$ with $\mathrm{Im}(\lambda_1)<0$.
\end{enumerate}
\end{definition}
\begin{remark}
Multiplying $a$ with positive constants and the rotations $\lambda\mapsto e^{\mi\varphi}\lambda$ only changes the parametrisation of the cylinder. Condition~{\rm{(v)}} fixes these two trivial transformations.
\end{remark}
Definition~\ref{spectral data} is analogous to
\cite[Definition~5.10]{HKS1}. Conditions~(i)-(ii) determine for all
polynomials $a$ of degree $2g$ obeying~\eqref{eq:a_reality} a real
two-dimensional space of polynomials $b$ of degree
$g+1$. Condition~(iii) is implicitly a condition on the polynomial
$a$ and characterises spectral curves of periodic solutions of the
corresponding integrable equation. Condition~(iv) is an extra
condition ensuring the periodicity of the immersed cylinder.
The following Proposition~\ref{thm spectral data} (compare with \cite{Bob:tor,Bob:cmc,HKS1}) states the algebro-geometric correspondence for elements $\imm\in\ann$:
\begin{proposition}\label{thm spectral data}
Let the spectral data $(a,b,\lambda_1,\lambda_2)$ obey
{\rm{(i)--(iv)}} in Definition~\ref{spectral data}. Then all
$\xi\in \iso{a}$ together with $\lambda_1$, $\lambda_2$ correspond
to \cmc immersions $\imm:\C/\tau\Z\to\Sp^3$ of finite type.

Conversely, if $\xi\in\pot{g}$ without roots on $\lambda\in\C^\ast$
together with $\lambda_1\not=\lambda_2\in\Sp^1$ corresponds to a
\cmc immersion $\imm:\C/\tau\Z\to\Sp^3$ of finite type with period
$\tau\in\C$, then besides $a=-\lambda\det\xi$ there exists a unique
polynomial $b$, such that $(a,b,\lambda_1,\lambda_2)$ obey {\rm{(i)--(iv)}}.
\end{proposition}
%
%%%%%%%%%%%%%%%%%%%%%%%%%%%%%%%%%%%%%%%%%%%%%%%%%%%%%%%%%%%%%%%%%%
%%%%%%%%%%%%%%%%%%%%%%%%%%%%%%%%%%%%%%%%%%%%%%%%%%%%%%%%%%%%%%%%%%
%
\section{The moduli spaces}\label{sec:moduli}
In this section we introduce the moduli space of spectral data.
\begin{definition}\label{spectral data 2}
For $g\in\N\cup\{0\}$ denote
\begin{align*}
    \moduli^g &= \{(a,\,b,\,\lambda_1,\,\lambda_2)\in\C^{2g}[\lambda]\times\C^{g+1}[\lambda]
    \times\Sp^1\times\Sp^1 \mid \mbox{(i)-(v) of definition \ref{spectral data} hold} \}\,,\\
    \moduli &= \bigcup_{g\in\N\cup\{0\}}\moduli^g\,.
\end{align*}
The subsets of spectral data of \cmc cylinders with non-negative
mean curvature are denoted by $\moduli^g_+\subset\moduli^g$ and
$\moduli_+\subset\moduli$. Due to
\eqref{eq:H_mu} the mean curvature is non-negative, if the arc of
$\Sp^1$  starting at
$\lambda_1$ in the anti-clockwise direction and ending at $\lambda_2$ has
length not larger than $\pi$. We call this arc the
{\bf{short arc}}. The arc starting at $\lambda_2$ in the
anti-clockwise direction and ending at $\lambda_1$ is called the {\bf{long arc}}.
\end{definition}
In the introduction we denote for given spectral data $(a,b,\lambda_1,\lambda_2)\in\moduli^g$ by $\ann(a,b,\lambda_1,\lambda_2)$ the set of all cylinders $\imm\in\ann$ corresponding to the triples $\{(\xi,\lambda_1,\lambda_2)\mid\xi\in\iso{a}\}$.
%For given spectral data $(a,b,\lambda_1,\lambda_2)$ the set of cylinders $\imm\in\ann$ corresponding to tripples $(\xi,\lambda_1,\lambda_2)$ with $\xi\in\iso{a}$ is denoted by $\ann(a,b,\lambda_1,\lambda_2)$.
%
\begin{remark}
The complex structure determines together with an orientation of
$\Sp^3$ a unique normal on the cylinder and the sign of $H$. The
antipodal map is an isometry of $\Sp^3$ which reverses the
orientation. Due to \eqref{eq:Sym_S3} it corresponds to an interchange
of the Sym points. Therefore spectral data with interchanged
Sym points have in $\ann$ the same isospectral sets. In particular,
each $\imm\in\ann$ is contained in $\ann(a,b,\lambda_1,\lambda_2)$ for some
$(a,b,\lambda_1,\lambda_2)\in\moduli_+$.
\end{remark}
\subsection{Higher order roots of $a$}
\label{subsection higher order roots}
For some $\imm\in\ann$ the
polynomial $a$ of the corresponding spectral data
$(a,b,\lambda_1,\lambda_2)$ has higher order roots at some
$\lambda=\alpha\in\C^\ast$. In this case
$\Sigma^\ast$ (defined in \eqref{eq:curve}) is not a submanifold of
$\C^\ast\times\C$ at $(\lambda,\nu)\in(\alpha,0)$, and $(\alpha,0)$ is
called a singularity of $\Sigma$. Hyperelliptic curves $\Sigma$ and
$\tilde{\Sigma}$, whose polynomials $\tilde{a}=p^2a$ differ by the
square of a polynomial $p$ have the same meromorphic functions. For
all $\tilde{a}$ with higher order roots, there exists a polynomial $p$
and a polynomial $a$ with $\tilde{a}=p^2a$ such that $a$ has only
simple roots. In this way $\Sigma$ is considered as the desingularised
$\tilde{\Sigma}$ with the same meromorphic functions.

We next characterise pairs $(a,b,\lambda_1,\lambda_2)$,
$(\tilde{a},\tilde{b},\lambda_1,\lambda_2)$ of spectral data in
$\moduli$ for which $\tilde{a}=p^2a$ and $\tilde{b}=pb$. We
decorate the objects corresponding to
$(\tilde{a},\tilde{b},\lambda_1,\lambda_2)$ with a tilde.

Suppose first $(\tilde{a},\tilde{b},\lambda_1,\lambda_2)\in\moduli^g$.
Choose any polynomial $p$ such that $p^2$ divides $\tilde{a}$, and
\begin{align}\label{eq:p_reality}
\lambda^{\deg p}\bar{p}(\bar{\lambda}^{-1})&=p(\lambda)&
|p(0)|&=1.
\end{align}
Condition~(iv) in Definition~\ref{spectral data} implies that
$\tilde{h}$ is holomorphic on $\tilde{\Sigma}^\ast$ and $p$ divides
$\tilde{b}$. For $a=\tilde{a}/p^2$ and $b=\tilde{b}/p$ we have
$dh=d\tilde{h}$ with $\lambda=\tilde{\lambda}$. Then
$(a,b,\lambda_1,\lambda_2)$ obeys
conditions~(i)-(v) with $f=p\tilde{f}$ and $g=\tilde{g}$.

Conversely, for $(a,b,\lambda_1,\lambda_2)\in\moduli^g$ we set $\tilde{a}=p^2a$
and $\tilde{b}=pb$. This implies $\tilde{h}=h$ with
$\tilde{\lambda}=\lambda$. The relations $\sigma^\ast h=-h$ and
$\sigma^\ast\nu=-\nu$ imply
\begin{align*}
g&=\cosh(h)=\cosh(\tilde{h})=\tilde{g},&
\frac{f}{p}&=\frac{\sinh(h)}{\nu p}=\frac{\sinh(\tilde{h})}{\tilde{\nu}}
=\tilde{f}.
\end{align*}
For $(\tilde{a},\tilde{b},\lambda_1,\lambda_2)$ to satisfy
condition (iv), $p$ must divide
$\frac{f(\lambda)}{(\lambda-\lambda_1)(\lambda-\lambda_2)}=
\frac{\sinh(h)}{\nu(\lambda-\lambda_1)(\lambda-\lambda_2)}$.
Thus $\sinh(h)$ vanishes at the roots of $p$. Differentiation gives
that the following meromorphic function is either holomorphic or has first
order poles at the roots of $p$:
\begin{align}\label{eq:poles}
\frac{dh}{(\lambda-\lambda_1)(\lambda-\lambda_2)\nu p(\lambda)}=
\frac{b(\lambda)}
{(\lambda-\lambda_1)(\lambda-\lambda_2)\lambda^2 a(\lambda)p(\lambda)}.
\end{align}
For such $p$ we have indeed
$(\tilde{a},\tilde{b},\lambda_1,\lambda_2)\in\moduli^g$. We summarise
the discussion in the following
\begin{lemma}\label{jump}
For $(\tilde{a},\tilde{b},\lambda_1,\lambda_2)\in\moduli^g$ choose any
polynomial $p$ obeying \eqref{eq:p_reality} such that $p^2$ divides
$\tilde{a}$.
Then $p$ divides $\tilde{b}$ and
$(a,b,\lambda_1,\lambda_2)\in\moduli^{g-deg p}$ with $\tilde{a}=p^2a$
and $\tilde{b}=pb$.

Conversely, suppose $(a,b,\lambda_1,\lambda_2)\in\moduli^g$ and $p$
obeys~\eqref{eq:p_reality}. If in addition
$\sinh(h)$ vanishes at the roots of $p$, and the function
\eqref{eq:poles} has at the roots of $p$ at worst
simple poles, then $(p^2a,pb,\lambda_1,\lambda_2)\in\moduli^{g+\deg p}$.\qed
\end{lemma}
In \cite[Sections~4 and 6]{HKS1} the corresponding sets
$\ann(p^2a,pb,\lambda_1,\lambda_2)$ and
$\ann(a,b,\lambda_1,\lambda_2)$ are investigated. For pairs
$(a,b,\lambda_1,\lambda_2)\in\moduli^g$ and
$(p^2a,pb,\lambda_1,\lambda_2)\in\moduli^{g+\deg p}$ as in
Lemma~\ref{jump} we have
$\ann(a,b,\lambda_1,\lambda_2)\subset\ann(p^2a,pb,\lambda_1,\lambda_2)$
with equality if all roots of $p$ belong to $\Sp^1$.

We interpret the supplementation of non-real singularities (away from $\Sp^1$) as an enrichment of the complexity and the removal as a reduction of the complexity. Geometrically this corresponds to adding or removing bubbletons by a suitable Bianchi-B\"acklund transform. Adding or removing a unimodular singularity does not change the complexity. It will turn out, that the enrichment of complexity destroys Alexandrov embeddedness, while the reduction of complexity preserves Alexandrov embeddedness.
\begin{definition}\label{def piecewise continuous}
A path in $\moduli$ is called piecewise continuous, if the path is
continuous within one $\moduli^g$ besides finitely many jumps from
spectral data $(a,b,\lambda_1,\lambda_2)\in\moduli^g$ to
$(p^2a,pb,\lambda_1,\lambda_2)\in\moduli^{g+\deg p}$ or vice versa
from $(p^2a,pb,\lambda_1,\lambda_2)\in\moduli^{g+\deg p}$ to
$(a,b,\lambda_1,\lambda_2)\in\moduli^g$ as in Lemma~\ref{jump}.
\end{definition}
%
%%%%%%%%%%%%%%%%%%%%%%%%%%%%%%%%%%%%%%%%%%%%%%%%%%%%%%%%%%%%%%%%%%
%%%%%%%%%%%%%%%%%%%%%%%%%%%%%%%%%%%%%%%%%%%%%%%%%%%%%%%%%%%%%%%%%%
%
\section{Example: Spectral data of rotational cylinders}\label{sec:rotational}
We next construct a two-dimensional family of spectral data
$(a,b,\lambda_1,\lambda_2)$, for which all triples
$(\xi,\lambda_1,\lambda_2)\in A^{-1}[\{(a,\lambda_1,\lambda_2)\}]$
correspond to mean-convex Alexandrov embedded cylinders. We first construct
in Theorem~\ref{thm:onesided revolution} a one-parameter
family of homogeneous mean-convex Alexandrov embedded \cmc cylinders,
and a two-parameter family of mean-convex Alexandrov embedded rotational
\cmc cylinders of spectral genus $g=1$. We then prove that all
mean-convex Alexandrov embedded cylinders of spectral genus zero belong
to the one-dimensional family of homogeneous cylinders.
\begin{theorem}\label{thm:onesided revolution}
There exists a one-dimensional family $\mrot^0\subset\moduli_+^0$ of spectral data of mean-convex Alexandrov embedded \cmc cylinders parameterised by the mean-curvature $H\in[0,\infty)$. For all elements of $\mrot^0$ the function $\mu^2-1$ vanishes at the two Sym points $\lambda_1$ and $\lambda_2$ and at $\lambda=-1$, and has no other root on $\Sp^1$. This family contains all spectral data of mean-convex Alexandrov embedded \cmc cylinders of spectral genus zero.

The one-dimensional family
$\{(\lambda+1)^2a,(\lambda+1)b,\lambda_1,\lambda_2)\mid(a,b,\lambda_1,\lambda_2)\in\mrot^0\}$ extends to a two-dimensional family $\mrot^1\subset\moduli_+^1$ parameterised by $(H,\alpha)\in[0,\infty)\times[2,\infty)$. The one-dimensional subfamily of $\mrot^1$ which is isomorphic to $\mrot^0$ corresponds to $\alpha=2$. For $\alpha>2$ the function $\mu^2-1$ vanishes on $\Sp^1$ only at $\lambda_1$ and $\lambda_2$.
\end{theorem}
\begin{proof}
For spectral genus zero, the polynomial $a$ is
equal to $a(\lambda)=-\frac{1}{16}$. In this case $h=\ln\mu$ in Definition~\ref{spectral data}~(iii) is a meromorphic function with
simple poles at $\lambda=0,\lambda=\infty$:
\begin{align*}
h&=\frac{\beta_0\lambda+\bar{\beta}_0}{4\lambda\nu},&
\nu^2&=-\frac{1}{16\lambda},&
dh&=\frac{\beta_0\lambda-\bar{\beta}_0}{8\lambda^2\nu}d\lambda,&
b(\lambda)&=\frac{\beta_0\lambda-\bar{\beta}_0}{8}.%change of coordinate
\end{align*}
We chose the Sym points as complex conjugate
$\lambda_2=\bar{\lambda}_1$. They are determined by $H$ \eqref{eq:H_mu}:
\begin{align}\label{eq:sym points}
%\mi\frac{1+\lambda_1^2}{1-\lambda_1^2}&=H&
\lambda_1&=\frac{H-\mi}{\sqrt{H^2+1}},&
\lambda_2&=\frac{H+\mi}{\sqrt{H^2+1}}&\mbox{or}&&
\lambda_1&=-\frac{H-\mi}{\sqrt{H^2+1}},&
\lambda_2&=-\frac{H+\mi}{\sqrt{H^2+1}}.
\end{align}
At $\lambda=\lambda_1,\lambda_2$ we require the
closing conditions that $h\in\pi\mi\Z$. Thus we have
\begin{equation} \begin{split} \label{eq:b_eq}
  \left. h\right|_{\lambda_1} \in \pi\mi \Z &\Longleftrightarrow
  \beta_0^2\lambda_1+\bar{\beta}_0^2\bar{\lambda}_1=
  \pi^2m^2-2\beta_0\bar{\beta}_0,\\
  \left. h\right|_{\lambda_2} \in \pi\mi \Z &\Longleftrightarrow
  \beta_0^2\lambda_2+\bar{\beta}_0^2\bar{\lambda}_2=
  \pi^2n^2-2\beta_0\bar{\beta}_0,
\end{split}
\end{equation}
for some integers $m,n\in\Z$ whose difference $m-n\in2\Z$ is even.
We make the following claim: If a cylinder is mean-convex Alexandrov
embedded then $m = n =\pm 1$ (for this ensures that the surface is
simply wrapped with respect to the rotational period). To prove this
claim first note that any spectral genus zero cylinder is a covering
of a homogeneous embedded torus \cite{KilSS:equi}. The complement
of this homogeneous embedded torus  with respect to $\Sp^3$ consists of
two connected components $\mathcal{D}_{\pm}$, both diffeomorphic to
$\mathbb{D}\times\Sp^1$. Assume the mean curvature vector points into
$\mathcal{D}_+$. For a mean-convex Alexandrov embedded \cmc
cylinder the extension $\imm:N \to \Sp^3$ is a surjective immersion
onto the closure $\bar{\mathcal{D}}_+$ of $\mathcal{D}_+$. Hence this
map is a covering map. The fundamental group of
$\bar{\mathbb{D}}\times\Sp^1$ is isomorphic to
$\Z$. Now the covers of a topological space are in one-to-one
correspondence with subgroups of the fundamental group
\cite[\S14]{St}, and all proper subgroups of
$\pi_1(\bar{\mathbb{D}}\times\Sp^1) \cong \Z$
correspond to compact covers. Hence the only non-compact cover of
$\bar{\mathbb{D}}\times\Sp^1$ is the universal cover
$\bar{\mathbb{D}}\times\R$. Therefore $\imm$ is the universal covering
map. In particular the period of the cylinder is a primitive period in
the kernel of
$$
H_1(\Sp^1 \times \Sp^1,\,\Z) \to
H_1(\bar{\mathbb{D}}\times\Sp^1,\,\Z )\,.
$$
These are primitive periods of closed one-dimensional subgroups of the
isometry group of $\Sp^3$, which fixes a great circle (rotation axis)
in $\Sp^3$. In the group $\SU \times \SU$ such subgroups belong to the
diagonal. This implies $m=n=\pm 1$ and proves the claim.

Returning to the proof of the theorem, then \eqref{eq:b_eq} with
$m^2=n^2=1$ and \eqref{eq:sym points} simplifies to
\begin{align*}
\beta_0^2&=\frac{\pi^2\sqrt{H^2+1}}{2\sqrt{H^2+1}+2H}&\mbox{or}&&
\beta_0^2&=\frac{\pi^2\sqrt{H^2+1}}{2\sqrt{H^2+1}-2H}.
\end{align*}
Now we claim that only the first solution corresponds to
mean-convex Alexandrov embedded cylinders. In fact, we showed that
the period of the cylinder corresponds to an element in the kernel of
$H_1(\partial\mathcal{D}_+,\Z)\to H_1(\bar{\mathcal{D}}_+,\Z)$.
The transformation $\lambda\mapsto-\lambda$ interchanges the
two solutions. Hence both families of homogeneous cylinders are
two copies of one family of tori considered as cylinders with respect to
different periods. In the limit $H\to\infty$ the length of the period
of the first solution stays bounded, and the length of the period of
the second solution tends to infinity. Hence the period of the second
solution are the rotation period of $\bar{\mathcal{D}}_-$, i.e.\ a
primitive element of the kernel of
$H_1(\partial\mathcal{D}_-,\Z)\to H_1(\bar{\mathcal{D}}_-,\Z)$ and
therefore the translation period of $\bar{\mathcal{D}}_+$. The period
of the first solution is the rotation period of $\bar{\mathcal{D}}_+$.
Therefore only the first solution corresponds to
mean-convex Alexandrov embedded cylinders. The corresponding family of
spectral data $(a,b,\lambda_1,\lambda_2)$ parameterised by
$H\in[0,\infty)$ is denoted by $\mrot^0$.

A root of $\mu^2-1$ is determined by the equation
\begin{align*}
\beta_0^2\lambda+\bar{\beta}_0^2\lambda^{-1}&=\pi^2l^2-2\beta_0\bar{\beta}_0
&\mbox{with }l&\in\Z.
\end{align*}
The first solution obeys $\beta_0^2\le\frac{\pi^2}{2}$. Hence the first
solution has on $\Sp^1$ only the solutions
$\lambda=\lambda_1,\lambda_2$ with $l=\pm 1$ and $\lambda=-1$ with
$l=0$. Due to condition~(iii) in Definition~\ref{spectral data} the
former has to be preserved. If we deform the latter into two different
roots of $a$, then $h$ remains a meromorphic function on the
genus one spectral curve. We thus obtain another family $\mrot^1$
parameterised by $H\in[0,\infty)$ and $\alpha\in[2,\infty)$:
\begin{align*}
a(\lambda)&=-\frac{\lambda^2+\alpha\lambda+1}{16},&
\nu^2&=-\frac{\lambda+\alpha+\lambda^{-1}}{16},\nonumber\\
h&=\frac{\beta_1\nu}{4},&
dh&=\frac{\beta_1(1-\lambda^2)}{8\nu\lambda^2}d\lambda,&
b(\lambda)&=\frac{\beta_1(1-\lambda^2)}{8}\,.
\end{align*}
At $\lambda_1,\lambda_2$ the closing conditions $h=\pm\pi\mi$ must
hold and thus
\begin{align*}
\beta_1^2&=\frac{\pi^2\sqrt{H^2+1}}{2H+\alpha\sqrt{H^2+1}}
\le\frac{\pi^2}{\alpha}&\mbox{for }H\ge 0.
\end{align*}
Therefore roots of $\mu^2-1$ at $\lambda$ have to satisfy
\begin{align*}
\beta_1^2(\lambda+\alpha+\lambda^{-1})&=\pi^2l^2&\mbox{with }l&\in\Z.
\end{align*}
For $\alpha\in(2,\infty)$ and $\lambda\in\Sp^1$ we have
$0<\beta_1^2(\lambda+\alpha+\lambda^{-1})\le2\pi^2$. Thus there are
no solutions on $\Sp^1$ besides $\lambda=\lambda_1,\lambda_2$ with
$l=\pm1$.
\end{proof}
The corresponding frames can be easily calculated, and the corresponding surfaces are surfaces of revolution around a closed geodesic \cite{KilSS:equi}. The boundary of $\mrot^1$ consists of
\begin{description}
  \item[homogeneous cylinders in $\Sp^3$] $H \in [0,\,\infty),\,\alpha=2$,
  \item[minimal cylinders in $\Sp^3$] $H=0,\,\alpha\in[2,\infty)$,
\end{description}
In addition to these boundary components there exists two limiting
cases:
\begin{itemize}
\item When $H=\infty,\,\alpha \in [2,\infty)$ we obtain unduloidal
\cmc cylinders in $\R^3$.
\item When $H\in[0,\infty),\, \alpha=\infty$,
the resulting surfaces are {\emph{chain of spheres}}.
\end{itemize}
%As a consequence of Theorem~\ref{thm:onesided revolution} for
%$2<\alpha$ the only way to increase the genus is to open two
%conjugate pairs of double roots of $a$ in $\C^\ast\setminus\Sp^1$.
%
%%%%%%%%%%%%%%%%%%%%%%%%%%%%%%%%%%%%%%%%%%%%%%%%%%%%%%%%%%%%%%%%%%%%
%%%%%%%%%%%%%%%%%%%%%%%%%%%%%%%%%%%%%%%%%%%%%%%%%%%%%%%%%%%%%%%%%%%%
%
\section{Spectral data of mean-convex Alexandrov embedded cylinders}
\label{sec:AE spectral data}
In this section we consider mean-convex Alexandrov embedded \cmc cylinders in $\ann$.
\begin{definition}\label{mean-convex}
A {\bf{mean-convex Alexandrov embedding}} in $\Sp^3$ is a smooth immersion $\imm:M \to \Sp^3$ from a connected surface $M$ which extends as an immersion to a connected 3-manifold $N$ with boundary $M=\partial N$ with the following properties:
\begin{enumerate}
  \item The mean curvature of $M$ in $\Sp^3$ with respect to the
  inward normal is non-negative everywhere.
  \item The manifold $N$ is complete with respect to the metric
  induced by $\imm$.
\end{enumerate}
An immersion $\imm:M\to\Sp^3$ just obeying condition~{\rm{(ii)}} is called an {\bf{Alexandrov embedding}}.

Let
\[
    \annae = \{ \imm \in \ann \mid \imm \mbox{ is mean-convex Alexandrov embedded } \} \,.
\]
\end{definition}
In \cite{HKS4} conditions on deformations of mean-convex Alexandrov embedded surfaces in $\Sp^3$ are established, which ensure that the property beeing mean-convex Alexandrov embedded is preserved. Here we consider deformations of \cmc immersions $\imm:\C^\ast\to\Sp^3$ with principal curvatures bounded by $\kappa\ind{max}$ and with second fundamental form $\II$ obeying for some $\ubdh>0$
\begin{align}\label{eq:upper bound derivative of h}
|(\nabla_X \II)(Y,Z)|&\leq \ubdh\,|X|\,|Y|\,|Z|&
\mbox{for all }p\in\C^\ast\mbox{ and all }X,Y,Z\in T_p\C^\ast.
\end{align}
Due to Proposition~\ref{uniform bound} these two conditions are satisfied as long as the corresponding tripples $(a,\lambda_1,\lambda_2)$ stay in compact subsest of $\bigmoduli^g$. Now \cite[Propostion~2.1 and Proposition~5.1]{HKS4} yield
\begin{proposition}\label{ae deformation}
There exists constants $\epsilon>0$ and $R>0$ depending only on $\kappa\ind{max}$ and $\ubdh$, with the following property: Let $\tilde{\imm}:\C^\ast\to\Sp^3$ be an immersion with non-negative mean curvature and principal curvatures bounded by $\kappa\ind{max}$ such that for all $w\in\C^\ast$ there exists a \cmc mean-convex Alexandrov embedding $\imm_w:\C^\ast\to\Sp^3$ with  principal curvatures bounded by $\kappa\ind{max}$ and second fundamental form obeying~\eqref{eq:upper bound derivative of h}. If $\|\imm_w-\tilde{\imm}\|_{\mathrm{C}^1(B(w,R))}<\epsilon$ for all $w\in\C^\ast$ on the ball $B(w,R)$ with respect to the metric induced by $\tilde{\imm}$, then $\tilde{\imm}$ extends to a mean-convex Alexandrov embedding.
\end{proposition}
The first application shows that isospectral deformations preserve the property of mean-convex Alexandrov embeddedness (compare \cite[Proposition~8.1 and 8.2]{HKS2}).
\begin{proposition}\label{embedded isospectral}
Let $\xi\in\pot{g}$ have no roots on $\lambda\in\C^\ast$ and let $(\xi,\lambda_1,\lambda_2)$ correspond to a mean-convex Alexandrov embedded cylinder $\imm:\C/\tau\Z\to\Sp^3$. Then we have:
\begin{enumerate}
\item If $a(\lambda)=\!-\lambda\det\xi$ has only simple roots, then $\{\pi(t)\xi\mid t\in\C^g\}=\iso{a}$ and all $(\tilde\xi,\lambda_1,\lambda_2)$ with $\tilde\xi\in\iso{a}$ correspond to a mean-convex Alexandrov embedded cylinder.
\item If $\tilde{a}(\lambda)=\!-\lambda\det\xi$ has higher order roots, then $\iso{\tilde{a}}$ is the closure of $\{\pi(t)\xi\mid t\in\C^g\}$ and all $(\tilde\xi,\lambda_1,\lambda_2)$ with $\tilde\xi\in\iso{\tilde{a}}$ correspond to mean-convex Alexandrov embedded cylinders.
\end{enumerate}
\end{proposition}
\noindent{\it Proof of }(i): %\begin{proof}
Due to Proposition~\ref{thm spectral data}, all $(\tilde{\xi},\lambda_1,\lambda_2)$ with $\tilde{\xi}\in\iso{a}$ correspond to \cmc cylinders $\tilde{\imm}:\C/\tau\Z\to\Sp^3$. All these \cmc cylinders $\tilde{\imm}$ have due to Proposition~\ref{uniform bound} uniform bounded principal curvatures and uniform bounded covariant derivatives of the second fundamental form \eqref{eq:upper bound derivative of h}.

The continuity and the commutativity of the isospectral action in Definition~\ref{groupaction} \cite[Section~4]{HKS1}
$$
    \pi (z+t)\xi= \pi(z)\pi(t)\xi= \pi(t)\pi(z)\xi
$$
and the compactness of $\iso{a}$ implies that for all $\epsilon>0$ there exists a $\delta>0$ such that
$$
    \|\pi(z)\pi(t)\tilde\xi-\pi(z)\tilde\xi\|=\|\pi(t)\pi(z)\tilde\xi-\pi(z)\tilde\xi\|\leq\sup\limits_{\xi\in\iso{a}}\|\pi(t)\xi-\xi\|\leq\epsilon\mbox{ for all }\tilde\xi\in\iso{a}\mbox{ and }|t|<\delta\,.
$$
Hence the immersions corresponding to $(\pi(t)\xi,\lambda_1,\lambda_2)$ and $(\xi,\lambda_1,\lambda_2)$ are globally $\mathrm{C}^1$-close for small $|t|$. Due to Proposition~\ref{ae deformation} there exists a $\delta>0$, such that for all $t\in B(0,\delta)$ the \cmc cylinders corresponding to $(\pi(t)\tilde{\xi},\lambda_1,\lambda_2)$ are mean-convex Alexandrov embedded, if $(\tilde{\xi},\lambda_1,\lambda_2)$ corresponds to a mean-convex Alexandrov embedded \cmc cylinder. Hence the set of all $t\in\C^g$ such that $(\pi(t)\xi,\lambda_1,\lambda_2)$ corresponds to a mean-convex Alexandrov embedding is $\C^g$. If $a$ has only simple roots, then due to \cite[Proposition~4.12]{HKS1} $\C^g$ acts transitively on $\iso{a}$ and all $(\tilde{\xi},\lambda_1,\lambda_2)$ with $\tilde{\xi}\in\iso{a}$ corresponds to mean-convex Alexandrov embeddings. This proves~(i).

\noindent{\it Proof of }(ii): Let $\tilde{a}=p^2a$ with a polynomial $p$ obeying~\ref{eq:p_reality}. Since $\xi \in \iso{\tilde{a}}$ would vanish at all roots of $p$ on $\Sp^1$, $p$ has no roots on $\Sp^1$ and $\deg p$ is even. For $\deg p=2$ we parameterise in \cite[Section~6]{HKS1} $\iso{\tilde{a}}$ by pairs $(L,\xi)$ of lines $L\in\CP$ together with $\xi\in\iso{a}$. The elements $\tilde\xi\in\iso{\tilde{a}}$ without roots correspond to pairs such that $L^\perp$ is not an eigenline of the value of $\xi$ at a root of $p$ \cite[Proposition~6.6]{HKS1}. Such $\tilde\xi$ form a dense orbit in $\iso{\tilde{a}}$. By induction in $\frac{\deg p}{2}$ we conclude for general $p$ without roots on $\Sp^1$ that $\{\pi(t)\tilde\xi\mid t\in \C^{g+\deg p}\}$ is dense in $\iso{\tilde{a}}$, if $\tilde\xi\in\iso{\tilde{a}}$ has no roots. The first assertion together with Proposition~\ref{ae deformation} implies (ii).\qed%\end{proof}

Let us now define the spectral data of $\imm\in\annae$. Any $\imm\in\ann$ belongs to $\ann(a,b,\lambda_1,\lambda_2)$ of many $(a,b,\lambda_1,\lambda_2)\in\moduli_+$. If we replace $(a,b,\lambda_1,\lambda_2)\in\moduli_+^g$ by $(p^2a,pb,\lambda_1,\lambda_2)\in\moduli_+^{g+\deg p}$ as in Lemma~\ref{jump}, then $\ann(a,b,\lambda_1,\lambda_2)\subset\ann(p^2a,pb,\lambda_1,\lambda_2)$. Due to Proposition~\ref{thm:PKF}~(iii) this is the only ambiguity of $(a,b,\lambda_1,\lambda_2)$. Indeed all $\imm\in\ann$ correspond to unique $(a,b,\lambda_1,\lambda_2)\in\moduli_+$ and $\xi\in\iso{a}$ without roots such that all $(\tilde{a},\tilde{b},\lambda_1,\lambda_2)\in\moduli_+$ with $\imm\in\ann(\tilde{a},\tilde{b},\lambda_1,\lambda_2)$ are of the form $(\tilde{a},\tilde{b},\lambda_1,\lambda_2)=(p^2a,pb,\lambda_1,\lambda_2)$ in Lemma~\ref{jump}. Proposition~\ref{embedded isospectral} shows that for all $\imm\in\annae$, these minimal sets $\ann(a,b,\lambda_1,\lambda_2)$ are completely contained in $\annae$. Therefore any $\imm\in\annae$ is contained in the set $\ann(a,b,\lambda_1,\lambda_2)$ of an element $(a,b,\lambda_1,\lambda_2)\in\mae\subset\moduli_+$ defined as follows:
\begin{definition}\label{spectral data 3}
Let $\mae^g$ denote the space of $(a,b,\lambda_1,\lambda_2)\in\moduli^g_+$ with $\ann(a,b,\lambda_1,\lambda_2)\subset\annae$.
\end{definition}
In a second application of Proposition~\ref{ae deformation} we show that continuous deformations preserve such spectral data (compare \cite[Proposition~2.7 and 2.8]{HKS2}):
\begin{proposition}\label{embedded deformation}
For all $g\in\N\cup\{0\}$ the space $\mae^g$ is an open and closed subset of $\moduli^g_+$.
\end{proposition}
\begin{proof}
We conclude from Proposition~\ref{ae deformation} and the properness and openess of the map $A$~\eqref{eq:map a} that $\mae^g$ is open and closed in $\moduli_+^g$. For $(a,b,\lambda_1,\lambda_2)\in\moduli_+^g$ let $N\subset\bigmoduli^g$ be an open neighbourhood of $(a,\lambda_1,\lambda_2)\in\bigmoduli^g$ with compact closure in $\bigmoduli^g$. Due to Proposition~\ref{uniform bound} the \cmc immersions of all $(\tilde{\xi},\tilde{\lambda}_1,\tilde{\lambda}_2)\in A^{-1}[N]$ have curvatures bounded by some $\kappa\ind{max}$ and obey \eqref{eq:upper bound derivative of h}.

For $(a,b,\lambda_1,\lambda_2)\in\mae^g$ all $(\xi,\lambda_1,\lambda_2)\in A^{-1}[\{(a,\lambda_1,\lambda_2)\}]$ correspond to $\imm\in\annae$ and have in $A^{-1}[N]$ an open neighbourhood, whose \cmc immersions $\tilde{\imm}$ obey $\|\imm-\tilde{\imm}\|_{\mathrm{C}^1(B(0,R))}<\epsilon$ on $B(0,R)\subset\C$ with the constants $\epsilon>0$ and $R>0$ of Proposition~\ref{ae deformation}. The union $U$ of these open neighbourhoods is an open neighbourhood of the compact subset $A^{-1}[\{(a,\lambda_1,\lambda_2)\}]$ of $A^{-1}[N]$. Proposition~\ref{ae deformation} implies that the following set $O$ is a subset of $\mae^g$:
$$O=\{(\tilde{a},\tilde{b},\tilde{\lambda}_1,\tilde{\lambda}_2)\in\moduli^g_+\mid(\tilde{\xi},\tilde{\lambda}_1,\tilde{\lambda}_2)\in U\mbox{ for all }\tilde{\xi}\in\iso{\tilde{a}}\}.
$$
We claim that $O$ is an open neighbourhood of $(a,b,\lambda_1,\lambda_2)$. Let a sequence $(a_n,b_n,\lambda_{1,n},\lambda_{2,n})_{n\in\mathbb{N}}$ in $\moduli^g_+\setminus O$ converge to $(\tilde{a},\tilde{b},\tilde{\lambda}_1,\tilde{\lambda}_2)\in\moduli^g_+$. Then there exists a sequence $(\xi_n,\lambda_{1,n},\lambda_{2,n})_{n\in\mathbb{N}}$ with $(\xi_n,\lambda_{1,n},\lambda_{2,n})\in A^{-1}[\{(a_n,\lambda_{1,n},\lambda_{2,n})\}]\setminus U$. Due to Lemma~\ref{open and proper} $A$ is proper. The preimeage
$$A^{-1}[\{(a_n,\lambda_{1,n},\lambda_{2,n})\mid n\in\mathbb{N}\}\cup
\{(\tilde{a},\tilde{\lambda}_1,\tilde{\lambda}_2)\}]$$
is compact and a subsequence of $(\xi_n,\lambda_{1,n},\lambda_{2,n})_{n\in\mathbb{N}}$ converges to $(\tilde{\xi},\tilde{\lambda}_1,\tilde{\lambda}_2)\in A^{-1}[\{(\tilde{a},\tilde{\lambda}_1,\tilde{\lambda}_2)\}]$. If $(\tilde{a},\tilde{\lambda}_1,\tilde{\lambda}_2)\not\in N$, then $(\tilde{a},\tilde{b},\tilde{\lambda}_1,\tilde{\lambda}_2)\not\in O$. Otherwise a subsequence of $(\xi_n,\lambda_{1,n},\lambda_{2,n})_{n\in\mathbb{N}}$ is mapped by $A$ to a convergent sequence in the open set $N$. This subsequence stays in the closed subset $A^{-1}[N]\setminus U$ of $A^{-1}[N]$ and has limit $(\tilde{\xi},\tilde{\lambda}_1,\tilde{\lambda}_2)\not\in U$. This again implies $(\tilde{a},\tilde{b},\tilde{\lambda}_1,\tilde{\lambda}_2)\not\in O$. Therefore $\moduli^g_+\setminus O$ is closed. This shows the claim and $\mae^g$ is open in $\moduli_+^g$.

Now we show that $\mae^g$ is closed in $\moduli^g_+$. Let $(a_n,b_n,\lambda_{1,n},\lambda_{2,n})$ be a sequence in $\mae^g$ converging in $\moduli^g_+$ to $(a,b,\lambda_1,\lambda_2)$. We have to show that any $(\xi,\lambda_1,\lambda_2)\in A^{-1}[\{(a,\lambda_1,\lambda_2)\}]$ corresponds to a mean-convex Alexandrov embedded cylinder $\imm$. By Lemma~\ref{open and proper} the map $A$ is open. Therefore every neighbourhood of $(\xi,\lambda_1,\lambda_2)$ contains elements of $A^{-1}[\{(a_n,\lambda_{1,n},\lambda_{2,n})\}]$. Therefore $\imm$ % the \cmc cylinder corresponding to $(\xi,\lambda_1,\lambda_2)$
obeys the condition of Proposition~\ref{ae deformation} and is mean-convex Alexandrov embedded.
\end{proof}
We summarize the results of this section in the following
\begin{theorem}\label{characterization}
The sets $\mae^g\subset\moduli^g_+$ have the following properties:
\begin{enumerate}
\item\label{prop1} $\mae^0=\mrot^0$.
\item\label{prop2} Let $(\tilde{a},\tilde{b},\lambda_1,\lambda_2)
=(p^2a,pb,\lambda_1,\lambda_2)$,
$(a,b,\lambda_1,\lambda_2)\in\moduli_+$ and $p$ be as in
Lemma~\ref{jump}.
\begin{itemize}
\item If
$(\tilde{a},\tilde{b},\lambda_1,\lambda_2)\in\mae^{g+\deg p}$,
  then $(a,b,\lambda_1,\lambda_2)\in\mae^g$.
\item If $(a,b,\lambda_1,\lambda_2)\!\in\!\mae^g$ and all
roots of $p$ belong to $\Sp^1$, then
$(\tilde{a},\tilde{b},\lambda_1,\lambda_2)\!\in\!\mae^{g+\deg p}$.
\end{itemize}
\item\label{prop3} For all $g\in\N\cup\{0\}$ the subset $\mae^g$ is closed
and open in $\moduli^g_+$.
\end{enumerate}
\end{theorem}
\begin{proof}
Due to Theorem~\ref{thm:onesided revolution} all mean-convex Alexandrov embedded annuli of spectral genus $0$ belong to $\annrot=\bigcup_{(a,b,\lambda_1,\lambda_2)\in\mrot}\ann(a,b,\lambda_1,\lambda_2)$. This shows property~\eqref{prop1}:
\begin{align*}%\label{eq:Ae genus zero}
\mrot^0&=\{(a,b,\lambda_1,\lambda_2)\in\moduli_+^0\mid
\ann(a,b,\lambda_1,\lambda_2)\cap\annae\ne\emptyset\}\\&=
\{(a,b,\lambda_1,\lambda_2)\in\moduli_+^0\mid
\ann(a,b,\lambda_1,\lambda_2)\subset\annae\}\,.\nonumber
\end{align*}
Since $\ann(a,b,\lambda_1,\lambda_2)\subset\annae$ for all
$(a,b,\lambda_1,\lambda_2)\in\mrot^0$ the first equality implies the second.

Property~\eqref{prop2} follows from the properties of $\ann(a,b,\lambda_1,\lambda_2)$ and $\ann(p^2a,pb,\lambda_2,\lambda_2)$ in the situation of Lemma~\ref{jump}. Due to Propostion~\ref{th:pKf_min} (compare \cite[Lemma~4.7]{HKS1}) we have in all cases $\ann(a,b,\lambda_1,\lambda_2)\subset\ann(p^2a,pb,\lambda_1,\lambda_2)$ with equality, if all roots of $p$ lie in $\Sp^1$.

Finally property~\eqref{prop3} is proven in Proposition~\ref{embedded deformation}.
\end{proof}
%
%We do not make use of the fact that all cylinders $\imm\in\bigcup_{(a,b,\lambda_1,\lambda_2)\in\mrot^1}\ann(a,b,\lambda_1,\lambda_2)$ are mean-convex Alexandrov embedded. In fact, this follows from Theorem~\ref{thm:main}.
%
%%%%%%%%%%%%%%%%%%%%%%%%%%%%%%%%%%%%%%%%%%%%%%%%%%%%%%%%%%%%%%%
%%%%%%%%%%%%%%%%%%%%%%%%%%%%%%%%%%%%%%%%%%%%%%%%%%%%%%%%%%%%%%%
%%%%%%%%%%%%%%%%%%%%%%%%%%%%%%%%%%%%%%%%%%%%%%%%%%%%%%%%%%%%%%%
%
\section{Deformation of spectral data}
\label{sec:deformation of cylinders}
We now derive vector fields on the space of spectral data and construct integral curves of these vector fields. We parameterise such families by one or more real parameters, which we will denote by $t$. We view the functions on the corresponding families of spectral curves $\Sigma_t$ locally as two-valued functions depending on $\lambda$ and $t$. In a first step we consider spectral data of periodic solutions of periodic finite type solutions of the $\sinh$-Gordon equation and ignore the Sym points. So let $\mper^g$ denote the set of pairs $(a,b)\in\C^{2g}[\lambda]\times\C^{g+1}[\lambda]$, which obey the conditions~(i)-(v) in Definition~\ref{spectral data} without the statements involving $\lambda_1$ and $\lambda_2$. From conditions~(ii)-(iii) we conclude that $\partial_th$ is meromorphic and anti-symmetric with respect to the hyperelliptic involution $\sigma$~\eqref{eq:involutions}. Furthermore, this function $\partial_th$ can only have poles at the branch points, or equivalently at the zeroes of $a$, and at $\lambda=0$ and $\lambda=\infty$. If we assume that for such a family of spectral curves the genus $g$ is constant, then $\partial_th$ can have at most first order poles at simple roots of $a$. In general we have
\begin{equation}\label{eq:def_c}
\partial_th =\frac{c(\lambda)}{\nu\lambda}
\end{equation}
with a polynomial $c$ of degree $g+1$ obeying the reality condition
\begin{equation}\label{eq:c_reality}
\lambda^{g+1}\bar{c}(\bar{\lambda}^{-1})=c(\lambda).
\end{equation}
The corresponding vector field on the space of spectral data is
derived from the equality of both mixed second derivatives with
respect to $\lambda$ and $t$. For this purpose we write the derivative of
$h$ with respect to $\lambda$ as
(compare Definition~\ref{spectral data}~(iii))
\begin{equation}\label{eq:def_b}
dh=\frac{b(\lambda)}{\nu\lambda^2}d\lambda,
\end{equation}
\begin{equation}\label{eq:derivative c}
 \partial^2_{t\lambda}h =\partial_{\lambda}\frac{c}{\nu\lambda}=
\frac{c'}{\nu \lambda}-\frac{c}{\nu \lambda ^2}-
\frac{c \nu '}{\nu ^2 \lambda}=
\frac{2\lambda a(\lambda)c'(\lambda)-a(\lambda)c(\lambda)-
\lambda a'(\lambda)c(\lambda)}{2\nu\lambda^2a(\lambda)}\,,
\end{equation}
\begin{equation*}
    \partial^2_{\lambda t }h =
    \partial_{t} \frac{ b}{\nu \lambda^2}=
    \frac{\dot b}{\nu \lambda^2}-\frac{b \dot \nu }{\nu ^2 \lambda ^2}=
    \frac{2a\dot b -b \dot a }{2\nu\lambda ^2 a}\,.
\end{equation*}
Second partial derivatives commute if and only if
\begin{equation} \label{eq:integrability_1}
2\dot{b}a - b\dot{a} = 2 \lambda ac' - ac - \lambda a'c.
\end{equation}
Both sides in the last formula are polynomials of at most degree $3g+1$ which satisfy a reality condition. This corresponds to $3g+2$ real equations. Choosing a polynomial $c$ which satisfies the reality condition~\eqref{eq:c_reality} we thus obtain a vector field on polynomials $a$ and $b$. If $a$ and $b$ don't have common roots, then~\eqref{eq:integrability_1} determines at the roots of $a$ the values of $\dot{a}$ and its derivatives with respect to $\lambda$ up to one order less than the multiplicity of the root. The normalization~(v) in Definition~\ref{spectral data} uniquely determines $a$ in terms of its roots. Consequently $\dot{a}$ and $\dot{b}$ depend smoothly on $a$, $b$ and $c$, as long as $a$ and $b$ don't have common roots. If $a$ and $b$ have common roots, the vector field defined in \eqref{eq:integrability_1} by a polynomial $c$ has a singularity. In \cite[Proposition~9.5]{HKS2} we construct integral curves passing through these singularities for polynomials $c$, which does not vanish at the common roots of $a$ and $b$. Here we prove:
\begin{lemma}\label{common roots}
Let $c$ be a polynomial of degree at most $g+1$ obeying~\eqref{eq:c_reality}. Any $(a,b)\in\mper^g$ with simple roots of $a$ is the value at $t=0$ of a smooth path $(-\epsilon,\epsilon)\ni t\mapsto(a_t,b_t)\in\mper^g$ such that the lowest non vanishing $t$-derivative of the corresponding $h_t$ is at $t=0$ given by~\eqref{eq:def_c}.
\end{lemma}
\begin{proof}
We use the methods of \cite[Section~9.2]{HKS2}. First we modify $(a,b)$ into $(\tilde{a},\tilde{b})\in\mper^{g+\deg p}$ as described in Section~\ref{subsection higher order roots}. At all common roots of $a$ and $b$ we add as many double roots of $a$ and simple roots of $b$ such that the function $\tilde{f}$ in Definition~\ref{spectral data}~(iv) corresponding to $(\tilde{a},\tilde{b})$ does not vanish at the roots of $\tilde{a}$. We remove small disjoint discs $\bigcup_{m\in M}V_m\subset\C^\ast$ around the roots of $\tilde{b}$ from the compactified $\lambda$-plane $\CP$. Each $V_m$ contains only one root of $\lambda\mapsto\tilde{a}(\lambda)\tilde{b}(\lambda)$. On each $V_m$ we choose a single-valued branch of $h$ and an integer $n_m$ such that $h-n_m\pi\mi$ vanishes only at the common roots of $\tilde{a}$ and $\tilde{b}$ in $V_m$ (if there is one). Due to~\eqref{eq:poles} $\frac{h-n_m\pi\mi}{\tilde{\nu}}$ is holomorphic on $V_m$ without roots. Therefore the roots of $h-n_m\pi\mi$ coincide with the roots of $\tilde{a}$ in $V_m$. The discs $V_m$ have unique coordinates $z_m$ which vanish at the root of $\tilde{b}$ in $V_m$ with
\begin{align*}%\label{eq:undeformed_1}
A_m(z_m)&=z_m^{d_m}+a_{m,d_m}=(h-n_m\pi\mi)^2.
\end{align*}
Here $d_m-1$ is the order of the root of $\tilde{b}$ in $V_m$. For common roots, $d_m$ is the order of the root of $\tilde{a}$. $V_m$ is mapped by $z_m$ biholomorphically onto small discs $W_m\subset\C$.

We  describe  spectral curves in a neighbourhood of the given spectral curve by small perturbations $(\tilde{A}_m)_{m\in M}$ of the polynomials $(A_m)_{m\in M}$. More precisely, we consider polynomials $(\tilde{A}_m)_{m\in M}$ of the following form with coefficients $\tilde{a}_{m,2},\ldots,(\tilde{a}_{m,m}-a_m)$ nearby zero:
\begin{equation}\label{eq:deformed pol A}
\tilde{A}_m(z_m)=z_m^{d_m}+\tilde{a}_{m,2}z_m^{d_m-2}+ \tilde{a}_{m,3}z_m^{d_m-3}+\ldots+\tilde{a}_{m,m}.
\end{equation}
By a shift $z \to z+z_0$, we can always assume that the sum of the roots is zero and then $\tilde{a}_{m,1}=0$. We glue each $W_m$ to $\CP\setminus\bigcup_{m\in M}V_m$ along the boundary of $V_m$ in such a way that for all $m\in M$  the polynomial $\tilde{A}_m $ coincides with the unperturbed function $A_m $ in a tubular neighborhood of the boundary $\partial W_m$. We obtain a new copy of $\CP$. By uniformization, there exists a new global parameter $\tilde{\lambda}$, which is equal to $0$ and $\infty$ at the two points corresponding to $\lambda=0$ and $\lambda=\infty$, respectively. This new parameter is unique up to multiplication with elements of $\C^\ast$. There exists a biholomorphic map $\tilde{\lambda}= \phi (\lambda)$ which changes the parameter $ \lambda \in\CP\setminus\bigcup_{m\in M}V_m$ to the global parameter $\tilde{\lambda}$. Furthermore there are biholomorphic maps $\tilde{\lambda}=\phi_m (z_m)$ which change the local parameter $z_m \in W_m$ into $\tilde{\lambda}$. Let $\tilde{\lambda}\to\tilde{a}(\tilde{\lambda})$ be the polynomial whose roots (counted with multiplicities) coincide on each $V_m$ with the zero set of $\tilde{A}_m(\tilde{\lambda})$ and on $\CP\setminus\bigcup_{m\in M}V_m$ with the roots of $\tilde{\lambda}\to a\circ\phi ^{-1}(\tilde{\lambda})$. Now $\tilde{\Sigma} =\{(\tilde{\nu},\tilde{\lambda})\in\C^2\mid\nu^2=\frac{\tilde{a}(\tilde{\lambda})}{\tilde{\lambda}}\}$ is the hyperelliptic curve associated to the set of polynomials $(\tilde{A}_m)_{m\in M}$. We say that polynomials $(\tilde{A}_m)_{m\in M}$ respect the reality condition if the involutions $\sigma$, $\rho$ and $\eta$ lift to involutions of $\tilde{\Sigma}$ and then define a spectral curve. In this case, the parameter $\tilde{\lambda}$ is determined up to a rotation $\tilde{\lambda}\mapsto e^{\mi\varphi}\tilde{\lambda}$. The equations
\begin{equation}\label{eq:deform_1}
(\ln\mu-n_m\mi\pi)^2=\tilde{A}_m(\tilde{\lambda})= \tilde{A}_m \circ  \phi _m ^{-1}  (\tilde{\lambda} )=\tilde{A}_m (z_m) \quad\hbox{ for }\quad m\in M
\end{equation}
define a function $\mu$ on the pre-image of $\phi_m(W_m) \cap \CP$ by the map $\tilde{\lambda}$ into  $\tilde{\Sigma}$. The function $\mu$ extends to the pre-image of $\C^* \setminus\bigcup_{m\in M}V_m$ by $\tilde{\lambda}=\phi (\lambda)$ and coincides with the unperturbed $\mu$ on this set. On $\tilde{\Sigma}$ the differential $d \ln \mu$ is again meromorphic and takes the form $d \ln \mu = \frac{\tilde{b}(\tilde{\lambda}) \,d \tilde{\lambda}}{\tilde{\nu}\tilde{\lambda}^2}$ with a unique polynomial $\tilde{b}$. By taking the derivative of \eqref{eq:deform_1} we have
\begin{align*}
2(\ln\mu-n_m\pi \mi)\tfrac{d}{d \tilde{\lambda}}\ln\mu&=
\tilde{A}_m'(z_m(\tilde{\lambda}))z'_m( \tilde{\lambda})\,.
\end{align*}
The roots of $\tilde{b}$ are the roots of the derivatives of $(\tilde{A}_m)_{m\in M}$. Alltogether the small coefficients $\tilde{a}_{m,2},\ldots,(\tilde{a}_{m,m}-a_m)$ of~\eqref{eq:deformed pol A} parameterise an open set of spectral data $(\tilde{a},\tilde{b},\lambda_1,\lambda_2)\in\mper^{g+\deg p}$. 

Now we determine the intersection of this open subset with the image of $\mper^g\hookrightarrow\mper^{g+\deg p}$. It corresponds to all $(\tilde{A}_m)_{m\in M}$ which have the same number of odd order roots as $(A_m)_{m\in M}$. This means that all those $A_m$ have exactely one odd order root, whose $V_m$ contain a root of $a$. We denote the corresponding subset of indices by $l\in L\subset M$. Hence the small coefficients of
\begin{align*}
\tilde{A}_l(z_l) &=(z_l-2b_{l,1})B_l^2(z_l)&&\hbox{with}&
B_l(z_l)&=z_l^{\ell_l}+b_{l,1}z^{\ell_l-1}+\ldots+b_{l,\ell_l}&&\text{and}&d_l&=2\ell_l+1.
\end{align*}
together with the small coefficients of the remaining polynomials $(\tilde{A}_m)_{m\in M\setminus L}$ parameterise a neighbourhood of $(a,b)\in\mper^g$. For a smooth family of such polynomials $(\tilde{A}_m)_{m\in M}$ we fix the degree of roations $\tilde{\lambda}\mapsto e^{\mi\varphi}\tilde{\lambda}$ by assuming that $\mu$ is constant along the path at $\lambda_0\in\Sp^1$ in the complement of $\bigcup_{m\in M}V_m$. Consequently the corresponding polynomial $c$ vanishes at $\lambda_0$. We obtain a smooth path in $\mper^g$. Due to \eqref{eq:def_c},\eqref{eq:def_b} and \eqref{eq:deform_1} the corresponding $t$-derivatives obey
\begin{align}\label{eq:no common root}
\frac{\partial_th}{\partial_\lambda h}=
\frac{\lambda c}{b}&=
\frac{\dot{\tilde{A}}_m}{\tilde{A}'_m}\frac{d\tilde{\lambda}}{dz_m}
&\mbox{on }&V_m&\mbox{for }m&\in M\setminus L&\\
\label{eq:common root}
\frac{\partial_th}{\partial_\lambda h}=
\frac{\lambda c}{b}&=
\frac{\dot{\tilde{A}}_l}{\tilde{A}'_l}\frac{d\tilde{\lambda}}{dz_l}=
\frac{-2\dot{b}_{l,1}B_l(z_l)+2(z_l-2b_{l,1})\dot{B}_l(z_l)}
     {B_l(z_l)+2(z_l-b_{l,1})B'_l(z_l)}\frac{d\tilde{\lambda}}{dz_l}
&\mbox{on }&V_l&\mbox{for }l&\in L.
\end{align}
For all $m\in M$ the quotient $\frac{d\tilde{\lambda}}{dz_m}$ does not vanish on $W_m$, since $\tilde{\lambda}=\phi_m(z_m)$ is biholomorphic. Therefore the common lowest non-vanishing $t$-derivative of $(\tilde{A}_m)_{m\in M}$ determines the singular part of $\frac{c}{b}$ and therefore the values of $c$ at all roots of $b$. Due to our normalization $c$ vanishes at $\lambda_0$ and is completely determined by the lowest non-vanishing $t$-derivative of $(\tilde{A}_m)_{m\in M}$. The rotation $\lambda\mapsto e^{\mi t}\lambda$ acts on $\mper^g$ as $a_t(\lambda)=a(e^{\mi t}\lambda)$ and $b_t(\lambda)=b(e^{\mi t}\lambda)$ and correspond to $c=\mi b$. Therefore it suffices to show that the lowest non-vanishing $t$-derivatives in the numerators on the right hand sides of~\eqref{eq:no common root} and \eqref{eq:common root} can take all polynomials of degree less than $d_m$ and $\ell_l$, repsectively. For $m\in M\setminus L$ in \eqref{eq:no common root} this is clear even for first order $t$-derivatives. For $l\in L$ we have $b_{l,1}=0$ and $B_l(z)=z^{\ell_l}$ at $t=0$. Therfore the first order $t$-derivatives in the numerator on the right hand side of~\eqref{eq:common root} can take at $t=0$ only all polynomials of degree less than $\ell_l$ which vanish at $z_l=0$. For the remaining constant polynomials we define (compare~\cite[Lemma~9.6]{HKS2}):
$$C_l(w)=\mbox{polynomial part of }w^{\ell_l}\left(1-\tfrac{2}{w}\right)^{-\frac{1}{2}}=\sum_{k=0}^{\ell_l}a_kw^{\ell_l-k}=w^{\ell_l}\left(1-\tfrac{2}{w}\right)^{-\frac{1}{2}}-\tfrac{a_{\ell_l+1}}{w}+\Order(w^{-2})$$
with $a_k=(-2)^k\binom{-\frac{1}{2}}{k}$. The polynomial $B_l(z_l)=t^{\ell_l}C(\frac{z_l}{t})$ with $b_{l,1}=a_1t=t$ yields
$$\tilde{A}_l(z_l)=(z_l-2t)B_l^2(z_l)=z_l^{2\ell_l+1}-2t^{\ell_l+1}a_{\ell_l+1}z_l^{\ell_l}+t^{\ell_l+2}\Order(z_l^{\ell_l-1}).$$
In this case the lowest non-vanishing $t$-derivative in the numerator on the right hand side of \eqref{eq:common root} is constant. By adding to $B_l$ products of $t^{\ell_l+1}$ with polynomials of degree less than $\ell_l$ this lowest non-vanishing $t$-derivative can take arbitrary polynomials of degree less than $\ell_l$.
\end{proof}
Now we present for later use two general constructions of deformations in $\moduli^g$. In addition to the polynomials $a$ and $b$ we have to deform  the two Sym points such that the conditions of Definition~\ref{spectral data} are preserved. As long as $\lambda_1\neq \lambda_2$, and thus $| H | <\infty$, we preserve the closing condition if $\frac{d}{dt}h(\lambda_j(t),\,t) = 0$, which holds precisely when $\partial_\lambda h(\lambda_j(t),\,t)\,\dot{\lambda}_j+\partial_th(\lambda_j(t),\,t) = 0$. Using equations \eqref{eq:def_b} and \eqref{eq:def_c}, the closing conditions are therefore preserved if and only if
\begin{align}\label{eq:integrability_2}
\frac{\dot{\lambda}_1}{\lambda_1}&=-\frac{c(\lambda_1)}{b(\lambda_1)}&
\mbox{and}&&
\frac{\dot{\lambda}_2}{\lambda_2}&=-\frac{c(\lambda_2)}{b(\lambda_2)}.
\end{align}
The equations \eqref{eq:integrability_1}-\eqref{eq:integrability_2}
define  rational vector fields on the
space of spectral data $(a,b,\lambda_1,\lambda_2)$ obeying
conditions~(i)-(iv) in Definition~\ref{spectral data}. These
vector fields are smooth as long as $a$ and $b$ do not have common
roots, and $b$ does not vanish at $\lambda_1$ and $\lambda_2$. The
normalisation $\lambda_2=\lambda_1^{-1}$ in
Definition~\ref{spectral data}~(v) is preserved if
\begin{align}\label{eq:normalisation}
\frac{c(\lambda_1)}{b(\lambda_1)}+\frac{c(\lambda_2)}{b(\lambda_2)}&=0.
\end{align}
The function $\mu=e^h$ is defined on the spectral curve, which is a
two-sheeted covering over $\lambda\in\CP$. Hence $\mu$ is a two-valued
function depending on $\lambda\in\C^\ast$. Instead of $\mu$ the function
\begin{align}\label{eq:delta}
\Delta:\C^\ast&\to\C,&\lambda\mapsto\Delta&=\mu+\mu^{-1}
\end{align}
is single-valued and determines the corresponding values
$\mu_{1,2}=\frac{1}{2}(\Delta\pm\sqrt{\Delta^2-4})$ of the function
$\mu$. The range of \eqref{eq:delta} is called $\Delta$-plane and the
domain is called $\lambda$-plane. Then $\mu\in\Sp^1$ is equivalent to
$\Delta\in[-2,2]$. We shall construct in
Proposition~\ref{global existence}  a one-dimensional
family of spectral data, such that the values of
$\Delta$ at the $g+1$ simple roots of $b$ are prescribed. These values
lie on given curves in the $\Delta$-plane. If these
curves do not intersect $\Delta=\pm2$, then the roots of $b$ stay
away from the roots of $a$.

Let $\Delta_0$ be the function
\eqref{eq:delta} corresponding to the initial spectral data
$(a_0,b_0,\lambda_{1,0},\lambda_{2,0})\in\moduli^g$.
We choose curves $\beta_1,\ldots,\beta_{g+1}$ in the
$\lambda$-plane. They define the prescribed curves
$t\mapsto\Delta_0(\beta_1(t)),\ldots,t\mapsto\Delta_0(\beta_{g+1}(t))$
in the $\Delta$-plane. We impose the following conditions:
\begin{enumerate}
\item\label{cond embedding}
$\beta_1,\ldots,\beta_{g+1}\!:\![-1,1]\!\to\!
\{\lambda\in\C^\ast\mid a_0(\lambda)\ne0,\lambda\ne\lambda_{1,2}\}$
are either constant (see Figure~1) or smooth embedded curves (see Figures~2-3).
\item\label{cond initial} The initial values
$\beta_1(0),\ldots,\beta_{g+1}(0)$ are the roots of the initial
polynomial $b_0$.
\item\label{cond disjoint}
Each curve $\beta_i([-1,1])$ is either disjoint from the other (see
Figure~1-2) or intersects exactly one other curve $\beta_j([-1,1])$ in
one point $\beta_i(1)=\beta_j(1)$ with
$\dot\beta_i(1)\ne\pm\dot\beta_j(1)$ (see Figure~3). We
assume that there exists a single-valued injective branch of $h$ on
$\beta_i[0,1]$ and $\beta_i[0,1]\cup\beta_j[0,1]$, respectively, which
obeys \eqref{eq:symmetry} along $\beta_i$ and $\beta_j$:
\begin{align}\label{eq:symmetry}
h(\beta_i(-t))&=h(\beta_i(t))&\mbox{ for all }&t\in[0,1].
\end{align}
Such curves $\beta_i$ exist since the simple root $\beta_i(0)$ of
$b$ is a first order branch point of $h$.
\item\label{cond reality}
The curves respect the reality condition, so that two not necessarily
different curves $\beta_i$ and $\beta_j$ with conjugated initial roots
$\beta_j(0)=\bar{\beta}_i^{-1}(0)$ obey
$\beta_j(t)=-\bar{\beta}_i^{-1}(t)$ for all $t\in[0,1]$.\vspace{-5mm}
\end{enumerate}

\hspace{1mm}\setlength{\unitlength}{5mm}
\begin{picture}(10,10)
\linethickness{0.05mm}
\put(4,5){\circle{2.5}}
\put(4,5){\circle*{.2}}
\put(3.3,5.3){$\beta_i(t)$}
\put(3.5,4.3){$S_m$}
\put(3.5,2.8){$V_m$}
\put(0,1){Figure~1. Constant curve}
\end{picture}\hspace{0mm}
\setlength{\unitlength}{5mm}
\begin{picture}(10,10)
\linethickness{0.05mm}
\qbezier(1,5)(2,5.5)(4,5)
\qbezier(4,5)(6,4.5)(7,5)
\put(4,5){\oval(8,2)}
\put(1,5){\circle*{.2}}
\put(.4,5.3){$\beta_i(-1)$}
\put(4,5){\circle*{.2}}
\put(3.3,5.3){$\beta_i(0)$}
\put(7,5){\circle*{.2}}
\put(6.0,5.3){$\beta_i(1)$}
\put(3.5,3){$V_m$}
\put(2.2,4.5){$S_m$}
\put(0,1){Figure~2. Single curve}
\end{picture}
\setlength{\unitlength}{5mm}
\begin{picture}(10,10)
\linethickness{0.05mm}
\qbezier(1,8)(2,8.5)(4,8)
\qbezier(4,8)(6,7.5)(7,8)
\qbezier(7,8)(6.5,7)(7,5.5)
\qbezier(7,5.5)(7.5,4)(7,3)
\put(5,8){\oval(10,2)[l]}
\put(7,6){\oval(2,8)[b]}
\qbezier(5,7)(6,7)(6,6)
\put(5,9){\line(1,0){2}}
\put(8,6){\line(0,1){2}}
\qbezier(7,9)(8,9)(8,8)
\put(1,8){\circle*{.2}}
\put(.3,8.3){$\beta_i(-1)$}
\put(4,8){\circle*{.2}}
\put(3.3,8.3){$\beta_i(0)$}
\put(7,8){\circle*{.2}}
\put(6.0,8.3){$\beta_i(1)$}
\put(7.2,7.5){$\beta_j(1)$}
\put(7,5.5){\circle*{.2}}
\put(7.2,5.5){$\beta_j(0)$}
\put(7,3){\circle*{.2}}
\put(7.2,2.8){$\beta_j(-1)$}
\put(4.8,6){$V_m$}
\put(5.6,7){$S_m$}
\put(0,1){Figure~3. Two curves}
\end{picture}\vspace{-3mm}
\begin{proposition}\label{global existence}
Let $(a_0,b_0,\lambda_{1,0},\lambda_{2,0})\in\moduli^g$ and let
$\Delta_0$~\eqref{eq:delta} take $g+1$ pairwise different values in
$\C\setminus\{-2,2\}$ at the roots of $b_0$. For
given smooth curves $\beta_1,\ldots,\beta_{g+1}$ obeying
\eqref{cond embedding}-\eqref{cond reality} there
exists a continuous family % of spectral data
$(a_t,b_t,\lambda_{1,t},\lambda_{2,t})_{t\in[0,1]}$ in $\moduli^g$
such that $\Delta_t$~\eqref{eq:delta} takes at the roots of the
polynomial $b_t$ the values
$\Delta_0(\beta_1(t)),\ldots,\Delta_0(\beta_{g+1}(t))$.
\end{proposition}
\begin{proof}
Again we use the methods of Lemma~\ref{common roots}
(compare \cite[Section~9.2]{HKS2}). We compactify the
$\lambda$-planes of the functions $\Delta_t$ \eqref{eq:delta}
corresponding to the spectral data
$(a_t,b_t,\lambda_{1,t},\lambda_{2,t})$ to $\CP$. For any $t\in(0,1]$
we remove from the compactified $\lambda$-plane of $\Delta_0$ the
curves $\beta_1,\ldots,\beta_{g+1}$. Along this open subset of $\CP$ we
glue deformed small tubular neighbourhoods $(W_{m,t})_{m\in M}$ of the
curves $\beta_1,\ldots,\beta_{g+1}$ as explained below. The result is
a compact simply connected Riemann surface. By the uniformisation
theorem there exists a global parameter $\lambda_t$ on this compact
Riemann surface, which takes the values $\lambda_t=0$ and
$\lambda_t=\infty$ at the two points $\lambda=0$ and $\lambda=\infty$
in the $\lambda$-plane of $\Delta_0$. The appropriately normalised
parameter $\lambda_t$ identifies the compact Riemann surface with the
compactified $\lambda$-plane of $\Delta_t$. Due to the gluing rules,
$\Delta_0$ extends to a holomorphic function $\Delta_t$ with respect
to $\lambda_t\in\C$ and defines the spectral data
$(a_t,b_t,\lambda_{1,t},\lambda_{2,t})$.

Due to conditions~\eqref{cond embedding} and \eqref{cond disjoint},
$\beta_1([-1,1])\cup\ldots\cup\beta_{g+1}([-1,1])$ is the disjoint
union $\bigcup_{m\in M}S_m$ of connected curves in the
$\lambda$-plane of $\Delta_0$. Each $S_m$ contains either one curve
like in Figure~1 and 2, or contains two connected curves like in
Figure~3. For each $m\in M$ the branch of $h$ specified in
condition~\eqref{cond disjoint} has on $S_m$ vanishing derivative only
at the roots of $b_0$. This branch takes due to \eqref{eq:symmetry}
on all curves $\beta_i$ in $S_m$ the same values at $\beta_i(t)$ and
$\beta_i(-t)$.

We choose simply connected tubular neighbourhoods $(V_m)_{m\in M}$ of
$(S_m)_{m\in M}$. Each $V_m$ is sufficiently small not to contain a
root of $a_0$. Therefore each
$U_m=\{(\lambda,\nu)\in\Sigma^\ast\mid\lambda\in V_m\}$ has two
connected components. The branch of $h$ in condition~(iii) extends to
one of these components. We choose integers $(n_m)_{m\in M}$ such that
$h-2n_m\pi\mi$ does not vanish on this component of $U_m$. We extend
the branch of $h$ to the other component of $U_m$ by setting
$\sigma^\ast(h-2n_m\pi\mi)=-h+2n_m\pi\mi$. With this choice
$(h-2n_m\pi\mi)^2$ is invariant with respect to $\sigma$ and
therefore is well defined on $V_m$. The only critical
points of $(h-2n_m\pi\mi)^2$ on $V_m$ are the roots of $b_0$ in $S_m$.

We define polynomials $A_m$ with the same critical values as $(h-2n_m\pi\mi)^2$ on $S_m$. After a change of coordinate
$z\mapsto cz+x_0$ the highest coefficient is one and the second highest
coefficient vanishes. So $A_m$ takes one of the following forms:
\begin{equation}\label{eq:pol}\begin{aligned}
A_m(z_m)&=z_m^2+a_m&
\mbox{crit.pt. }z_m&=0&
\mbox{crit.value }a_m\\
A_m(z_m)&=z_m^3+b_mz_m+a_m&
\mbox{crit.pts. }z_m&=\!\pm(\tfrac{-b_m}{3})^{1/2}&
\mbox{crit.values }a_m&\!\mp2(\tfrac{-b_m}{3})^{3/2}.
\end{aligned}\end{equation}
On each simply connected $V_m$ there exists a holomorphic function $z_m$ such that $(h-2n_m\pi\mi)^2$ is equal to a $A_m(z_m)$. The roots of $b_0$ are the critical points of $(h-2n_m\pi\mi)^2$ and correspond to the critical points of $A_m$. Due to condition~\eqref{cond disjoint} on a sufficiently small tubular neighbourhood $V_m$ the function $z_m$ is a biholomorphic map onto a simply connected open subset $W_m\subset\C$. The image of $S_m$ in this subset is denoted by $T_m$. For all $m\in M$ we obtain a biholomorphic map $S_m\subset V_m\simeq T_m\subset W_m$ (left hand side of \eqref{diagram}).

Now we deform these discs $(W_m)_{m\in M}$. For this purpose we deform
the polynomials $A_m$ into a continuous family of polynomials
$(A_{m,t})_{t\in[0,1]}$ of the form~\eqref{eq:pol} with
coefficients $a_{m,t}$ and $b_{m,t}$, respectively. Their critical
values are the values of $(h-2n_m\pi\mi)^2$ at $\beta_i(t)$ and
$\beta_i(-t)$ and eventually at $\beta_j(t)$ and $\beta_j(-t)$. Due to
\eqref{eq:pol} the coefficient $a_{m,t}$ depend smoothly on
$t\in[0,1]$. In case of two intersecting curves like in Figure~3 the
difference of the critical values has at $t=1$ a first order root with
respect to $t$. In this case $b_{m,t}$ depends smoothly on
$\tilde{t}=1-(1-t)^\frac{1}{2}$. For any $z_m\in W_m\setminus T_m$
we consider a curve
\begin{align}\label{eq:curves}
[0,1]&\to\C&t&\mapsto z_{m,t}&\mbox{ with}&&
A_{m,t}(z_{m,t})&=A_m(z_m)\,.
\end{align}
Since $A_{m,t}(z_{m,t})$ is constant along the curve, it solves the
differential equation
\begin{align*}
\dot{z}_{m,t}&=-\frac{\dot{A}_{m,t}(z_{m,t})}{A'_{m,t}(z_{m,t})}&
\mbox{with}&& z_{m,0}&=z_m.
\end{align*}\vspace{-12mm}

\noindent\hspace{10mm}\setlength{\unitlength}{5mm}
\begin{picture}(10,10)
\linethickness{0.05mm}
\qbezier(1,5)(2,5.5)(4,5)
\qbezier(4,5)(6,4.5)(7,5)
\put(4,5){\oval(8,2)}
\put(1,5){\circle*{.2}}
\put(4,5){\circle*{.2}}
\put(7,5){\circle*{.2}}
\put(3.5,3){$W_m$}
\put(5,5){$T_m$}
\put(1.3,1){Figure~4. Deformation of a}
\end{picture}\hspace{0mm}
\setlength{\unitlength}{5mm}
\begin{picture}(10,10)
\linethickness{0.05mm}
\qbezier(2,4.75)(3,5)(4,5)
\qbezier(4,5)(5,5)(6,5.25)
\qbezier(4.25,7)(4,6)(4,5)
\qbezier(4,5)(4,4)(3.75,3)
\put(2.5,4.75){\oval(3,2)[l]}
\put(5.5,5.25){\oval(3,2)[r]}
\put(4.25,6.5){\oval(2,3)[t]}
\put(3.75,3.5){\oval(2,3)[b]}
\qbezier(2.5,3.75)(2.75,3.75)(2.75,3.5)
\qbezier(4.75,3.5)(4.75,4.25)(5.5,4.25)
\qbezier(5.5,6.25)(5.25,6.25)(5.25,6.5)
\qbezier(3.25,6.5)(3.25,5.75)(2.5,5.75)
\put(2,4.75){\circle*{.2}}
\put(6,5.25){\circle*{.2}}
\put(4.25,7){\circle*{.2}}
\put(3.75,3){\circle*{.2}}
\put(.9,3){$W_{m,\frac{1}{2}}$}
\put(4.2,5.7){$T_{m,\frac{1}{2}}$}
\put(0,1){single curve}
\end{picture}
\setlength{\unitlength}{5mm}
\begin{picture}(8,10)
\linethickness{0.05mm}
\qbezier(3,1.5)(3.5,2.5)(3,4.5)
\qbezier(3,4.5)(2.5,6.5)(3,7.5)
\put(3,4.5){\oval(2,8)}
\put(3,1.5){\circle*{.2}}
\put(3,4.5){\circle*{.2}}
\put(3,7.5){\circle*{.2}}
\put(0.2,3.8){$W_{m,1}$}
\put(2.75,6.1){$T_{m\hspace{-.3mm},\hspace{-.3mm}1}$}
\end{picture}
\vspace{-4mm}

As long as $z_{m,t}$ does not meet a critical point of $A_{m,t}$ it is
smooth. Along the interval $t\in[0,1]$ the critical values of
$A_{m,t}$ move along the values of $A_m$ on $T_m$. This shows that for
$z_m\in W_m\setminus T_m$ these curves are smooth. Furthermore, for
any $t\in[0,1]$ the maps $z_m \mapsto z_{m,t}$ are biholomorphic maps from
$W_m\setminus T_m$ onto the complement $W_{m,t}\setminus T_{m,t}$ of
the union of finitely many compact curves $T_{m,t}$ in an open subset
$W_{m,t}$. On the deformations $T_{m,t}$ of $T_m$ the deformed
polynomial $A_{m,t}$ takes the same values as the undeformed
polynomial $A_m$ on $T_m$.  For $t\neq0,1$ two curves of $T_{m,t}$
intersect each other at the critical points of $A_{m,t}(z_{m,t})$ as
in Figure~4.

We  glue the deformed
subsets $(W_{m,t})_{m\in M}$ of $\C$ parameterised by $z_{m,t}$ in
such a way with the undeformed $\CP\setminus\bigcup_{m\in M}S_m$, that
the values of the polynomials $A_{m,t}(z_{m,t})$ coincide on
$W_{m,t}\setminus T_{m,t}$ with the values of the initial polynomials
$A_m(z_m)$ on $W_m\setminus T_m\simeq V_m\setminus S_m$ :\vspace{1mm}
\begin{equation}\label{diagram}
\setlength{\unitlength}{5mm}
\begin{picture}(28,1.8)
\linethickness{0.05mm}
\put(5,.7){\vector(1,-2){.5}}
\put(6.6,.7){\vector(-1,-2){.5}}
\put(5.5,1.3){$\simeq$}
\put(1,1.2){$S_m\!\subset\!V_m\!\ni\!\lambda$}
\put(6.6,1.2){$z_m\!\in\!T_m\!\subset\!W_m$}
\put(0.7,0){$(h\!-\!2n_m\pi\mi)^2$}
\put(6.6,-0.1){$A_m$}
\put(5.5,-1.3){$\C$}
\put(20,.7){\vector(1,-2){.5}}
\put(21.6,.7){\vector(-1,-2){.5}}
\put(20.5,1.3){$\simeq$}
\put(15.9,1.2){$W_m\!\setminus\!T_m\!\ni\!z_m$}
\put(21.6,1.2){$z_{m,t}\!\in\!W_{m,t}\!\setminus\!T_{m,t}$}
\put(19,-0.1){$A_m$}
\put(21.6,-0.1){$A_{m,t}$}
\put(20.5,-1.3){$\C$}
\end{picture}\vspace{5mm}
\end{equation}
The initial function $\Delta_0$ on
$\CP\setminus\bigcup_{m\in M}T_m$ is on $W_m\setminus T_m$ equal
to $2\cosh(\sqrt{A_m(z_m)})$. It extends as
$2\cosh(\sqrt{A_{m,t}(z_{m,t})})$ holomorphically to
$W_{m,t}$. We obtain a new copy of $\CP$ with two
marked points $\lambda=0$ and $\lambda=\infty$ and a holomorphic map
$\Delta_t$ from the complement of these marked points to $\C$. Due to
condition~\eqref{cond reality} the anti-holomorphic involution
$\lambda\mapsto\bar{\lambda}^{-1}$ of
$\CP\setminus\bigcup_{m\in M}T_m$ extends to a
global involution of this copy of $\CP$. By uniformisation there
exists a global parameter $\lambda_t$ on this Riemann surface, which
takes at the marked points the values $0$ and $\infty$ and transforms
with respect to the anti-holomorphic involution like
$\lambda_t\mapsto\bar{\lambda}^{-1}_t$. This parameter is unique up to
rotation $\lambda_t\mapsto e^{\mi\varphi}\lambda_t$. For any
$t\in[0,1]$ and any choice of the global parameter $\lambda_t$ the
function $\Delta_0$ on $\CP\setminus\bigcup_{m\in M}T_m$ extends to
a holomorphic function $\Delta_t$ depending on the global parameter
$\lambda_t\in \C^\ast$, which is on $V_{m,t}$ equal to
$2\cosh(\sqrt{A_{m,t}(z_{m,t})})$. We may fix the degree of freedom of
rotations, by assuming that the values $\lambda_{1,t}$ and
$\lambda_{2,t}$ of $\lambda_t$ at the Sym points $\lambda_1$ and
$\lambda_2$ in the $\lambda$-plane of $\Delta_0$ obey
$\lambda_{1,t}\lambda_{2,t}^{-1}=\lambda_1^{-1}\lambda_2$. This
normalisation preserves condition~(v) in Definition~\ref{spectral data}.
Let $a_t$ be the unique polynomial obeying \eqref{eq:a_reality}, whose
roots are the values of $\lambda_t$ at the roots of $a_0$ in the
$\lambda$-plane $\CP\setminus\bigcup_{m\in M}T_m$ of
$\Delta_0$. Since $a$ has $2g$ roots in
$\C^\ast\setminus\bigcup_{m\in M}T_m$, $a_t$ has $2g$ roots in
$\lambda_t\in\C^\ast$.  By construction, the 2-valued function $\mu$ with
$\Delta_0(\lambda)=\mu+\mu^{-1}$ on
$\CP\setminus\bigcup_{m\in M}T_m$ extends to a unique 2-valued
function $\mu_t$ depending on $\lambda_t\in\CP$ with
$\mu_t+\mu_t^{-1}=\Delta_t(\lambda_t)$. The corresponding 1-form
$dh_t$ is meromorphic on the spectral curve induced by
$a_t$ of the form \eqref{eq:def_b} with a unique polynomial
$b_t$. Together with $\lambda_{1,t}$ and $\lambda_{2,t}$ we obtain a
family of spectral data
$(a_t,b_t,\lambda_{1,t},\lambda_{2,t})_{t\in[0,1]}\in\moduli^g$.
By construction $\lambda_t\mapsto a_t(\lambda_t)
b_t(\lambda_t)(\lambda_t-\lambda_{1,t})(\lambda_t-\lambda_{2,t})$ has
pairwise different roots for $t\ne 1$.

Finally we show that this family of spectral data is smooth, since it
solves an ordinary differential equation with smooth coefficients. The
$t$-derivative of $\mu(\beta(t),t)$ at a simple root $\beta(t)$ of $b$
is equal to $\frac{d}{dt}\mu(\beta(t),t)=
\partial_\lambda\mu(\beta(t),t)\dot{\beta}(t)+\partial_t\mu(\beta(t),t)=
\frac{c(\beta(t))}{\nu\beta(t)}$. It does not depend on $\dot{\beta}(t)$ and
is determined by $c(\beta(t))$. For roots $\beta$ of $b$ on $\Sp^1$
the polynomial $c(\lambda)=\frac{\lambda+\beta}{\lambda-\beta}b(\lambda)$
vanishes at all other roots of $b$. For roots
$\beta\in\C^\ast\setminus\Sp^1$ the polynomial
$c(\lambda)=(\frac{C}{\lambda-\beta}-
\frac{\bar{C}\lambda}{1-\bar{\beta}\lambda})b(\lambda)$ vanishes at all roots
of $b$ besides $\beta$ and $\bar{\beta}^{-1}$. Therefore we may change
the values of $\mu$ at the simple roots of $b$
independently. Furthermore, the values of $c$ at the roots of $b$
determine $c$ up to adding to $c$ a real multiple of $\mi b$. The
vector field corresponding to $c=\mi b$ describes the rotations
$\lambda\mapsto e^{\mi t}\lambda$. After adding to $c$ a real
multiple of $\mi b$ the sum of the values of $\frac{c}{b}$ at
$\lambda_1$ and $\lambda_2$ vanishes. Due to \eqref{eq:integrability_2}
this normalisation preserves $\lambda_1\lambda_2^{-1}$. In particular,
the smooth curves
$t\mapsto\Delta_0(\beta_1(t)),\ldots,t\mapsto\Delta_0(\beta_{g+1}(t))$
determine together with the normalisation
$\frac{d}{dt}\lambda_1\lambda_2^{-1}=0$ smooth $t$-dependent $c$ on
the space of spectral data $(a,b,\lambda_1,\lambda_2)$ with pairwise
different roots of $\lambda\mapsto
a(\lambda)b(\lambda)(\lambda-\lambda_1)(\lambda-\lambda_2)$, such
that the values of $\Delta$~\eqref{eq:delta} at the roots of $b$ of the
corresponding solution of
\eqref{eq:integrability_1}-\eqref{eq:integrability_2} follow the given
curves
$t\mapsto\Delta_0(\beta_1(t)),\ldots,t\mapsto\Delta_0(\beta_{g+1}(t))$.
This implies that this family of spectral data is for $t\ne1$ indeed
smooth. In case of two intersecting curves $\beta_i$ and $\beta_j$
like in Figure~3 at $t=1$ the critical values of $A_{m,t}$ at the two
coalescing roots of $b_t$ does not depend smoothly on $t$. But  this
family depends together with the coefficients of $(A_{m,t})_{m\in M}$ for
$t=1-(1-\tilde{t})^2$ smoothly on $\tilde{t}\in[0,1]$.
\end{proof}
The proof even shows that there exists a family of spectral data
$(a_t,b_t,\lambda_{1,t},\lambda_{2,t})_{t\in\mathcal{T}}\in\moduli^g$
parameterised by tuples
$t=(t_1,\ldots,t_{g+1})\in\mathcal{T}\subset[0,1]^{g+1}$ with the values
$h(\beta_i(t_i))=h(\beta_i(-t_i))$ at the roots of $b_t$. Here
$\mathcal{T}$ is due to condition~\eqref{cond reality} characterised
by the equations $t_i=t_j$ on the parameters of conjugated initial roots
$\beta_j(0)=\bar{\beta}_i^{-1}(0)$.

The following Lemma is used in Section~\ref{sec:isolated}.
\begin{lemma}\label{smooth moduli}
Let $\tilde{\moduli}^g$ be the space of $(a,b,\lambda_1,\lambda_2)\in
\C^{2g}[\lambda]\times\C^{g+1}[\lambda]\times\Sp^1\times\Sp^1$ obeying
(i)-(v) in Definition~\ref{spectral data} without the
inequality in \eqref{eq:a_reality}, such that
$\lambda_1\ne\lambda_2$, $b(\lambda_1)\ne0\ne b(\lambda_2)$ and
$\frac{b}{a}$ has first order poles at all common roots of $a$ and
$b$. It is a real $g+1$-dimensional submanifold. The equations
\eqref{eq:integrability_1} and \eqref{eq:integrability_2}
identify $T\tilde{\moduli}^g$ with $c\in\C^{g+1}[\lambda]$ obeying
\eqref{eq:c_reality} and \eqref{eq:normalisation}.
\end{lemma}
\begin{proof}
The space of $(a,b,\lambda_1,\lambda_2)$ obeying condition (i) in Definition~\ref{spectral data} without the inequality in \eqref{eq:a_reality} is a real $(2g+1)+(g+2)+2=3g+5$ dimensional space. Condition~(ii) states that the integral of $dh$ along $g$ independent cycles of $\Sigma$ vanishes. This contains $g$ linear independent conditions. The normalization~(vi) contains two linear independent conditions. Hence conditions~(i)-(ii) and (vi) describe a $2g+3$-dimensional subspace. Condition~(iii) states that the integrals of $dh$ along $g$ further independent cycles of $\Sigma$ takes values in $2\pi\mi\Z$. Condition~(iv) assumes that the integrals of $dh$ along two paths connecting both points over $\lambda_1$ and $\lambda_2$ takes values in $2\pi\mi\Z$. Hence $\tilde{\moduli}^g$ is the subset of a $2g+3$-dimensional space defined by $g+2$ real functions taking values in $2\pi\Z$. We show that all tangent vectors in the kernel of the derivatives of these functions form a $g+1$-dimensional space and apply the implicit function theorem.

For a tangent vector $(\dot{a},\dot{b},\dot{\lambda}_1,\dot{\lambda}_2)$  at $(a,b,\lambda_1,\lambda_2)\in\tilde{\moduli}^g$ in the kernel of the derivatives of these functions the integral of the 1-form $d\dot{h}$ in Definition~\ref{spectral data}~(iii) along any closed cycle of $\Sigma$ vanishes. The derivative $\dot{h}$ defines a single-valued meromorphic function on the corresponding spectral curve $\Sigma$. Due to the properties of $h$ it is of the form~\eqref{eq:def_c} with a polynomial $c$ of degree $g+1$ obeying \eqref{eq:c_reality} and \eqref{eq:integrability_1}. Equation~\eqref{eq:integrability_2} determines $\dot{\lambda}_1$ and $\dot{\lambda}_2$ in terms of $c$. If $a$ and $b$ do not have a common root, then \eqref{eq:integrability_1} uniquely determines $\dot{a}$ and $\dot{b}$ in terms of $c$. Due to the implicit function theorem $\tilde{\moduli}^g$ is at $(a,b,\lambda_1,\lambda_2)$ a submanifold with $T\tilde{\moduli}^g$ isomorphic to the $g+1$ dimensional space of polynomials $c$ obeying~\eqref{eq:c_reality} and \eqref{eq:normalisation}.
\begin{remark}
The relation~\eqref{eq:integrability_1} between $c$ and
$(\dot{a},\dot{b})$ does not encode
Definition~\ref{spectral data}~(iv). Without this condition we can add
to the polynomial $a$ without changing $h$ a double root $\alpha\in\Sp^1$ or a pair
of double roots $\alpha,\bar{\alpha}^{-1}\in\C^\ast\setminus\Sp$
together with the same simple roots of $b$. The movement of
such double roots is not controlled by $c$. Due to
condition~(iv) we can add such roots of $a$ and
$b$ only at points with $\mu=\pm 1$. If $\frac{b}{a}$ has simple poles
at common roots, then $c$ determines $(\dot{a},\dot{b})$.
\end{remark}
It remains to consider the case of common roots of $a$ and $b$ with
simple poles of $\frac{b}{a}$. We describe $\Sigma$ by the function
$\Delta$~\eqref{eq:delta} and use the methods of
Lemma~\ref{common roots} (compare \cite[Section~9.2]{HKS2}.) We
remove small disjoint discs $V_m\subset\C^\ast$ around the roots of
$b$ from the compactified $\lambda$-plane $\CP$ of
$\Delta$~\eqref{eq:delta}. Each $V_m$ contains only one root of
$\lambda\mapsto
a(\lambda)b(\lambda)(\lambda-\lambda_1)(\lambda-\lambda_2)$. On each
$V_m$ we choose a single-valued branch of $h$ and an integer
$n_m$ such that $h-n_m\pi\mi$ vanishes only at the common root of
$a$ and $b$ in $V_m$ (if there is one). Due to~\eqref{eq:poles}
$\frac{h-n_m\pi\mi}{\nu}$ is holomorphic on $V_m$ without
roots. Therefore the roots of $h-n_m\pi\mi$ coincide with the
roots of $a$ in $V_m$. The discs $V_m$ have unique coordinates $z_m$
which vanish at the root of $b$ in $V_m$ with
\begin{align}\label{eq:undeformed_2}
A_m(z_m)&=z_m^{d_m}+a_{m,d_m}=(h-n_m\pi\mi)^2.
\end{align}
Here $d_m-1$ is the order of the root of $b$ in $V_m$. For common
roots, $d_m$ is the order of the root of $a$. $V_m$ is mapped
by $z_m$ biholomorphically onto small discs $W_m\subset\C$.

For all $(\tilde{a},\tilde{b},\tilde{\lambda}_1,\tilde{\lambda}_2)$
in a sufficiently small neighbourhood $O\subset\tilde{\moduli}^g$ of
$(a,b,\lambda_1,\lambda_2)$ the corresponding branches
$\tilde{h}-n_m\pi\mi$ are well defined on the deformed subsets
$\tilde{V}_m$ of the compactified $\lambda$-plane
of the deformed $\tilde{\Delta}$~\eqref{eq:delta}. Furthermore the
function $(\tilde{h}-n_m\pi\mi)^2$ takes for all $m\in M$ on the
boundary $\partial\tilde{V}_m$ the same values as $A_m$ on the
boundary of $W_m$. We glue the complement $\CP\setminus W_m$ of the
compactified $z_k\in\CP$ plane in such a way along the boundary of
$\tilde{V}_m$ that $A_m(z_m)$ coincides on $\partial W_m$ with
$(\tilde{h}-n_m\pi\mi)^2$ on $\partial\tilde{V}_m$. This yields a new
copy of $\CP$. By uniformisation there exists a global coordinate
$\tilde{z}_m$ on the new copy of $\CP$ with a pole at the pole of
$z_m$ in $\CP\setminus W_m$. It is unique up to M\"obius
transformations fixing $\tilde{z}_m=\infty$. The function
$A_m(z_m)=(h-n_m\pi\mi)^2$~\eqref{eq:undeformed_2} on $\CP\setminus W_m$
extends as $(\tilde{h}-n_m\pi\mi)^2$ to $\tilde{W}_m$. This function is
meromorphic on the new copy of $\CP$ with a single pole at the pole of
$\tilde{z}_m$ of degree $d_m$. Therefore this function is a polynomial
$\tilde{A}_m$ with respect to $\tilde{z}_m$ of degree $d_m$. After a
unique M\"obius transformation of the coordinate $\tilde{z}_m$ fixing
$\tilde{z}_k=\infty$ we have
\begin{align}\label{eq:deform_2}
\tilde{A}_m(\tilde{z}_m)&=
\tilde{z}_m^{d_m}+\tilde{a}_{m,2}\tilde{z}^{d_m-2}+\ldots+\tilde{a}_{m,d_m}
=(\tilde{h}-n_m\pi\mi)^2.
\end{align}
The coefficients $\tilde{a}_{m,2},\ldots,\tilde{a}_{m,d_m}$ belong to
a small neighbourhood of the corresponding coefficients of $A_m$. The
map $\lambda\mapsto\bar{\lambda}^{-1}$ interchanges the roots of $b$
and therefore also the discs $(V_m)_{m\in M}$ and the polynomials
$(A_m)_{m\in M}$. Since this symmetry is preserved, the coefficients of
those not necessarily different polynomials have complex conjugate
coefficients, which are interchanged by this symmetry. The sum of the
real degrees of freedom of the coefficients of
$(\tilde{A}_m)_{m\in M}$ is $\deg b=g+1$. We
obtain a space of polynomials $(\tilde{A}_m)_{m\in M}\in B$
of small perturbations of $(A_m)_{m\in M}$ parameterised by a small ball
in $\R^{g+1}$. This yields a map
\begin{align*}%\label{eq:map}
\Phi:O&\to
B\subset\R^{g+1}&
(\tilde{a},\tilde{b},\tilde{\lambda}_1,\tilde{\lambda}_2)&\mapsto
(\tilde{A}_m)_{m\in M}.
\end{align*}
In order to apply the implicit function theorem it remains to show
that $\Phi'$ maps those tangent vectors
$(\dot{a},\dot{b},\dot{\lambda}_1,\dot{\lambda}_2)$ at
$(a,b,\lambda_1,\lambda_2)$ injectively to
$T_{(A_m)_{m\in M}}B=\R^{g+1}$ which correspond to polynomials $c$
of degree $g+1$ obeying \eqref{eq:c_reality} and
\eqref{eq:normalisation}. Due to \eqref{eq:def_c},\eqref{eq:def_b} and
\eqref{eq:deform_2} we have
\begin{align*}
\frac{\partial_t\tilde{h}}{\partial_\lambda h}=
\frac{\lambda c}{b}&=
\frac{\dot{\tilde{A}}_m}{A'_m}\frac{d\lambda}{dz_m}
\end{align*}
on $V_m$. Since $\lambda\mapsto z_m$ is biholomorphic from $V_m$
onto $W_m$, $\frac{d\lambda}{dz_m}$ does not vanish on $V_n$. Hence
the singular parts of $\frac{\lambda c}{b}$ at the roots of $b$ are in
one-to-one correspondence to $(\dot{\tilde{A}}_m)_{m\in M}$.
\end{proof}
%
%%%%%%%%%%%%%%%%%%%%%%%%%%%%%%%%%%%%%%%%%%%%%%%%%%%%%%%%%%%%%%%
%%%%%%%%%%%%%%%%%%%%%%%%%%%%%%%%%%%%%%%%%%%%%%%%%%%%%%%%%%%%%%%
%
%
\section{Paths to spectral genus zero}\label{sec:pathconnected}
In this section we construct a piecewise continuous path (see
Definition~\ref{def piecewise continuous}) in $\moduli_+$ connecting an
arbitrary starting point with an endpoint in $\moduli_+^0$.
In the proof we control the deformation of the spectral data by choosing
piecewise appropriate polynomials $c$.
\begin{theorem}\label{t1}
At any $(a,b,\lambda_1,\lambda_2)\in\moduli^g_+$ there
starts a compact piecewise continuous path $\gamma$ in $\moduli_+$ to
$\moduli_+^0$. Along $\gamma$ the spectral genus increases at most by one
at a multiple root of $a$ in $\Sp^1$.
\end{theorem}
\begin{proof}
We prove this theorem in seven steps. In each step we choose an
appropriate polynomial $c$, such that the corresponding solution of
the ordinary differential equations
\eqref{eq:integrability_1}-\eqref{eq:integrability_2} describes a path
with specified properties. For this purpose we exhibit how
the choice of $c$ controls the movement of the roots of $a$, the roots
of $b$ and the Sym points $\lambda_1$ and $\lambda_2$. Let us assume
that there is given a path of spectral data parameterised by $t$, which
solves~\eqref{eq:integrability_1} with some smooth family of
$c$'s. The corresponding equations~\eqref{eq:curve} define a
family of hyperelliptic algebraic curves. These curves are two-sheeted
coverings over $\lambda\in\CP$. We consider the functions $\mu$ on
these two sheeted coverings as two-valued functions depending on
$\lambda$ and the deformation parameter $t$. The hyperelliptic
involution $\sigma$~\eqref{eq:involutions} interchanges the two
branches of $\mu(\lambda,t)$ and acts by $\mu\mapsto\mu^{-1}$. Now let
$t\mapsto\lambda(t)$ be a path in $\C^\ast$, such that the
function $t\mapsto\mu(\lambda(t),t)$ is constant along the former path
of spectral data. Due to \eqref{eq:def_c} and \eqref{eq:def_b} the
latter path $t\mapsto\lambda(t)$ obeys
\begin{align}\label{eq:vector}
\frac{d\mu(\lambda(t),t)}{dt}=
\frac{\partial\mu(\lambda(t),t)}{\partial\lambda}\dot{\lambda}(t)+
\frac{\partial\mu(\lambda(t),t)}{\partial t}&=0,&
\frac{\dot{\lambda}(t)}{\lambda(t)}&=-\frac{c(\lambda(t))}{b(\lambda(t))}.
\end{align}
The roots of $a$ and the Sym points \eqref{eq:integrability_2} move
along paths on which $\mu$ is constant. Finally, simple roots of $b$,
which are also roots of $c$, also move along paths on which $\mu$ is
constant. In this case the right hand side in the differential
equation should be replaced by the unique holomorphic extension of the
quotient $-\frac{c}{b}$ to the common roots of $b$ and
$c$. Thus~\eqref{eq:vector} describes the movement of several
significant points. The $t$-derivative of $\mu(\beta(t))$ at a simple
root $\beta(t)$ of $b$ is equal to $\frac{d}{dt}\mu(\beta(t),t)=
\partial_\lambda\mu(\beta(t),t)\dot{\beta}(t)+\partial_t\mu(\beta(t),t)=
\frac{c(\beta(t))}{\nu\beta(t)}$. It does not depend on $\dot{\beta}(t)$ and
is determined by $c(\beta(t))$. Changing the values of $\mu$ at the
roots of $b$ implicitly changes the relative positions of the roots of
$a$, the roots of $b$ and the Sym points. At the roots of $a$ and the
Sym points we have $\mu=\pm 1$. We can avoid that the roots of $b$
meet the roots of $a$ and the Sym points by ensuring that the values
of $\mu$ (or $\Delta$) at the roots of $b$ stay away from $\mu=\pm1$
(or $\Delta=\pm2$).

The function $\mu$ transforms as $\sigma^\ast\mu=\mu^{-1}$ and
$\rho^\ast\mu=\bar{\mu}^{-1}$. It takes on the fixed point set
$\lambda\in\Sp^1$ of the involution $\rho$ \eqref{eq:involutions}
unimodular values. Therefore the argument of $\mu$ is on $\Sp^1$ a
function with values in $\R/2\pi\Z$. The two branches of $\arg\mu$
are the negative of each other. The critical points of $\arg\mu$
on $\Sp^1$ are the roots of $b$ on $\Sp^1$. In
particular, at simple roots of $b$ on $\Sp^1$, one branch of this
function has a local maximum and the other branch a local minimum.

We first apply in steps~1-3 small deformations to achieve that $a$ and
$b$ have pairwise different simple roots in
$\C^\ast\setminus\{\lambda_1,\lambda_2\}$. In step~4 we achieve in
addition that $\mu$ takes at the roots of $b$ in $\Sp^1$ pairwise different
values on $\Sp^1\setminus\{-1,1\}$ and at the other roots pairwise
different values on $\C^\ast\setminus\Sp^1$.

\noindent{\bf 1.} In step~1 we divide the polynomials $a$ and $b$
by polynomials $p^2$ and $p$, if $a$ has higher order roots, as
described in Lemma~\ref{jump}. We will meet such $a$ again only in
the last step.

\begin{minipage}[b]{118mm}
\noindent{\bf 2.} In step~2 we choose a small deformation,
which decreases the length of the short arc and eventually separates
the roots of $b$ from the Sym points $\lambda_1$ and $\lambda_2$.
This step is only required if initially $H=0$ or if $b$ vanishes at one
Sym point.

For polynomials $c$ obeying~\eqref{eq:c_reality} the meromorphic function
$-\frac{c}{b}$ takes imaginary values on $\Sp^1$. We choose $c$, which
does not vanish at the Sym points $\lambda_1$ and $\lambda_2$, such
that $-\frac{c}{b}$ has positive imaginary part at the
\end{minipage}
\hspace{1mm}\setlength{\unitlength}{5mm}
\begin{picture}(7,5)
\linethickness{0.05mm}
\put(3.5,3.5){\circle{3}}
\put(4,2.2){\circle*{0.2}}
\put(4.35,2.4){\vector(1,1){.3}}
\put(3.6,2.5){$\lambda_1$}
\put(4,4.8){\circle*{0.2}}
\put(4.35,4.6){\vector(1,-1){.3}}
\put(3.7,4.1){$\lambda_2$}
\put(0.5,3.5){long}
\put(.8,2.9){arc}
\put(5.1,3.5){short}
\put(5.3,2.9){arc}

\put(0,1){Figure~5. Sym points}
\end{picture}\vspace{-1.5mm}

end of the short arc nearby $\lambda_1$ and negative imaginary part
at the end of the short arc nearby $\lambda_2$. Since $c$ has degree
$g+1\ge2$ such polynomials exists. If $b$ does not vanish at
the Sym points, then due to~\eqref{eq:integrability_2} with this
choice of $c$ the Sym points are moved towards the interior of the
short arc.

Since $\arg\mu$ has only isolated critical points on $\Sp^1$,
the restriction of one branch to the short arc has at the boundary of
the short arc a local maximum, and the restriction of the other branch
has a local minimum. Due to \eqref{eq:def_c} the branch of
$\arg\mu$ with a local maximum at the end of the short arc strictly
increases (for increasing $t$) nearby the Sym point, and the
other branch strictly decreases.

In Figure~6 a family of graphs of
the function $\arg\mu$ on $\Sp^1$ with a local maximum at
$\arg\lambda_0$ for $t=t_0$ is shown. The corresponding pre-image
$\{(\arg\lambda,t)\mid\mu(\lambda,t)=\mu_0\}$ is shown in
Figure~7. At $(\lambda_0,t_0)$ we see a bifurcation into two
different paths $t\mapsto\lambda(t)$ in the pre-image
$\{(\lambda,t)\mid\mu(\lambda,t)=\mu_0\}$.
If the Sym point is a root of $b$, then the pre-image
$\{(\lambda,t)\mid\mu(\lambda,t)=\mu_0\}$ might contain several paths
ending and starting at the Sym point. Any such path in $\Sp^1$ (if
they exist) can be chosen as the corresponding Sym point along the
path $t\mapsto(a_t,b_t)$. With our choice of $c$ there
always exist two continuous paths $t\mapsto\lambda_1(t)$ and
$t\mapsto\lambda_2(t)$ for $t\ge 0$ moving the two Sym
points toward the interior of the short arc. By choosing
these paths as the Sym points along the path $t\mapsto(a_t,b_t)$, we
obtain a path of spectral data with strictly decreasing length of the
short arc. This shows that with our choice of $c$ an integral curve of
\eqref{eq:integrability_1} yields a path in $\moduli_+^g$ with
strictly increasing mean curvature. If $a$ and $b$ do not have common
roots, then the smooth vector field \eqref{eq:integrability_1} always
has a local solution. If $a$ and $b$ have common roots, then our
arguments apply for the specified $c$ to the corresponding path
constructed in Lemma~\ref{common roots}
\vspace{-3mm}

\hspace{1mm}\setlength{\unitlength}{5mm}
\begin{picture}(10,10)
\linethickness{0.05mm}
\put(0,5){\vector(1,0){9}}
\put(8.1,5.2){$\arg\lambda$}
\put(5.5,2){\vector(0,1){7}}
\put(5.7,8.6){$\arg\mu$}
\put(5.6,5.2){$\arg\mu_0$}
\put(4.35,3){$\arg\lambda_0$}
\qbezier(0,6)(1,3)(2.5,3)
\qbezier(2.5,3)(3,3)(4,4)
\qbezier(4,4)(5,5)(5.5,5)
\qbezier(5.5,5)(7,5)(8,2)
\qbezier(0,7)(1,4)(2.5,4)
\qbezier(2.5,4)(3,4)(4,5)
\qbezier(4,5)(5,6)(5.5,6)
\qbezier(5.5,6)(7,6)(8,3)
\qbezier(0,8)(1,5)(2.5,5)
\qbezier(2.5,5)(3,5)(4,6)
\qbezier(4,6)(5,7)(5.5,7)
\qbezier(5.5,7)(7,7)(8,4)
\qbezier(0,9)(1,6)(2.5,6)
\qbezier(2.5,6)(3,6)(4,7)
\qbezier(4,7)(5,8)(5.5,8)
\qbezier(5.5,8)(7,8)(8,5)
\put(1.65,3.3){$t\!\!=\!\!t_0$}
\put(1.65,4.3){$t\!\!=\!\!t_1$}
\put(1.65,5.3){$t\!\!=\!\!t_2$}
\put(1.65,6.3){$t\!\!=\!\!t_3$}
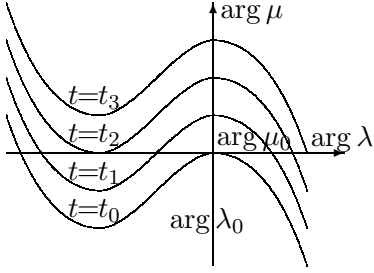
\put(0,1){Figure~6. Family of graphs}
\end{picture}
\setlength{\unitlength}{5mm}
\begin{picture}(10,10)
\linethickness{0.05mm}
\put(0,3){\vector(1,0){9}}
\put(8.1,3.2){$\arg\lambda$}
\put(5.5,2){\vector(0,1){7}}
\put(5.7,8.5){$t$}
\put(4.35,2.3){$\arg\lambda_0$}
\qbezier(0,2)(1,5)(2.5,5)
\qbezier(2.5,5)(3,5)(4,4)
\put(4.5,3.5){\vector(-1,1){0.5}}
\qbezier(4,4)(5,3)(5.5,3)
\qbezier(5.5,3)(7,3)(8,6)
\put(6.5,3.4){\vector(1,1){0.5}}
\put(5.6,3.1){$t_0$}
\put(5.4,3,8){-}\put(5.6,4.1){$t_1$}
\put(5.4,4,8){-}\put(5.6,5.1){$t_2$}
\put(5.4,5,8){-}\put(5.6,6,1){$t_3$}
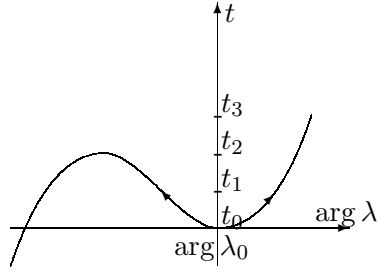
\put(0,1){Figure~7. Pre-image of $\mu=\mu_0$}
\end{picture}\hspace{0mm}
\setlength{\unitlength}{5mm}
\begin{picture}(10,10)
\linethickness{0.05mm}
\put(0,5){\vector(1,0){9}}
\put(8.1,5.2){$\arg\lambda$}
\put(4,2){\vector(0,1){7}}
\put(4.2,8.6){$\arg\mu$}
\put(4.1,5.2){$\arg\mu_0$}
\put(2.85,3){$\arg\lambda_0$}
\qbezier(0,8)(1,8)(4,5)
\qbezier(0,2)(1,2)(4,5)
\qbezier(8,8)(7,8)(4,5)
\qbezier(8,2)(7,2)(4,5)
\qbezier(0,8.2)(1.5,8)(3,6.5)
\qbezier(3,6.5)(3.5,6)(4,6)
\qbezier(4,6)(4.5,6)(5,6.5)
\qbezier(5,6.5)(6.5,8)(8,8.2)
\qbezier(0,1.8)(1.5,2)(3,3.5)
\qbezier(3,3.5)(3.5,4)(4,4)
\qbezier(4,4)(4.5,4)(5,3.5)
\qbezier(5,3.5)(6.5,2)(8,1.8)
\put(4,6){\vector(0,-1){0.6}}
\put(4,4){\vector(0,1){0.6}}
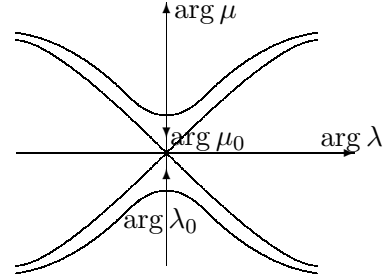
\put(0,1){Figure~8. Coalescing roots of $a$}
\end{picture}\vspace{-3mm}

\noindent{\bf 3.} In step~3 we choose a small deformation which
separates the roots of $b$ from each other. In step~2 we achieved that
the short arc has length smaller then $\pi$. Hence we can follow any
deformation for a short time. The values of
$\partial_t\partial_\lambda h$~\eqref{eq:def_c} at some higher order
root $\beta\in\C^\ast$ of $b$ depends linearly on $c(\beta)$ and
$c'(\beta)$~\eqref{eq:derivative c}. Since $c$ has degree $g+1\ge 1$
there exist $c$ such that $\partial_t\partial_\lambda h$ does not vanish at
$\beta$. With the following lemma we may deform successively all
higher order roots of $b$ into simple roots.
\begin{lemma}\label{higher roots}
Let $c\in\C^{g+1}[\lambda]$ obey \eqref{eq:c_reality} such that the
1-form $d\frac{c}{\nu\lambda}$~\eqref{eq:def_c} does not vanish
at a higher order root of $b$ which is not a root of $a$. Then the
flow of the vector field~\eqref{eq:integrability_1} separates the
higher order root of $b$ after arbitrarily short time into several
simple roots. On $\Sp^1$ the 1-form $d\frac{c}{\nu\lambda}$ is
purely imaginary. One choice of sign of $d\frac{c}{\nu\lambda}$ at a
double root of $b$ on $\Sp^1$ separates the double root of $b$ along
$\Sp^1$. The other choice of sign separates the double root off $\Sp^1$.
\end{lemma}
\begin{proof}
At higher order roots of $b$ there exists a local coordinate $z$ such
that $\partial_\lambda h=z^n$ with $n\ge 2$. If
$\partial_t\partial_\lambda h$ takes the value $C\in\C^\ast$ there, then
$\partial_\lambda h$ is for small $t$ nearby $z=0$ of the form
$$\partial_\lambda h=z^n+Ct+\Order(t^2)+\Order(z).$$
For small $t$ any such function has $n$ distinct simple roots near $z=0$.

The function $h$ and the 1-form $dh$ are purely imaginary on
$\Sp^1$. Therefore $\partial_tdh$ is also purely imaginary on
$\Sp^1$. Due to \eqref{eq:def_c} the 1-form $d\frac{c}{\nu\lambda}$ is
equal to $\partial_tdh$ and purely imaginary on $\Sp^1$. Nearby a
double root of $dh$ on $\Sp^1$ there exists a local coordinate $z$
which is real on $\Sp^1$ such that
$$dh=dh|_{t=0}+t\partial_tdh+\Order(t^2)=
\mi z^2dz+t\left(\mi Cdz+\Order(z)dz\right)+\Order(t^2)
\quad\mbox{with}\quad C\in\R.$$
For small $t>0$ the 1-form $dh$ has for $C>0$ no root and for $C<0$
two roots in $z\in\R$ near $z=0$
\end{proof}

\noindent{\bf 4.}
After steps~1-3 we can assume that the short arc has length smaller than $\pi$,
and all roots of $b$ and $a$ are simple and pairwise distinct, and lie
away from the Sym points. In step~4 we achieve in addition that $\mu$
takes at the roots of $b$ in $\Sp^1$ pairwise different values on
$\Sp^1\setminus\{-1,1\}$ and at the other roots pairwise different
values on $\C^\ast\setminus\Sp^1$. The $t$-derivative of
$h(\beta(t),t)$ at a simple root $\beta(t)$ of $b$ does not depend on
$\dot{\beta}(t)$ and  is determined by $c(\beta(t))$ (see
\eqref{eq:vector}). For roots $\beta$ on $\Sp^1$ the polynomial
$c(\lambda)=\frac{\lambda+\beta}{\lambda-\beta}b(\lambda)$ vanishes at
all other roots of $b$. For roots $\beta\in\C^\ast\setminus\Sp^1$
the polynomials $c(\lambda)=(\frac{C}{\lambda-\beta}-
\frac{\bar{C}\lambda}{1-\bar{\beta}\lambda})b(\lambda)$ vanishes at all roots
of $b$ besides $\beta$ and $\bar{\beta}^{-1}$. Therefore we may change
the values of $\mu$ at the simple roots of $b$ independently.

\noindent{\bf 5.} After the first four steps we have achieved that $a$ and $b$ have pairwise different simple roots in $\C^\ast\setminus\{\lambda_1,\lambda_2\}$, and such that $\mu$ takes at the roots of $b$ pairwise different values in $\C^\ast\setminus\{1,-1\}$. In step~5 we remove all roots of $b$ from the
short arc with possibly one exception. Here we use two Lemmata.

\begin{lemma}\label{l1}
Let $\beta\in\Sp^1$ be a simple root of $b$ in the interior of the
short arc. We choose the sign of the polynomial
$c=\pm\frac{\lambda+\beta}{\lambda-\beta}b$, whose vector field
decreases $\arg\mu$ at the branch having on $\Sp^1$ at $\beta$ a local
maximum, and increases $\arg\mu$ at the branch having on $\Sp^1$ at
$\beta$ a local minimum, respectively. Then this vector field decreases
the length of the short arc.

If $\beta$ belongs to the interior of the long arc,
then the same conclusion holds for the sign, which increases $\arg\mu$
at the branch having on $\Sp^1$ at $\beta$ a local maximum.
\end{lemma}

\begin{proof}\label{p2}
Let $\beta\in\Sp^1$ be a simple root of $b$ in the interior of the
short arc. The vector field corresponding to $c=\mi b$ describes a
rotation of $\lambda$ and does not change the length of the short and
the long arc. Therefore we can add to $c$ a real multiple of $\mi b$
without changing the derivative of the length of the short arc. After
adding to $c$ an appropriate multiple of $\mi b$ the quotient
$\frac{c}{b}$ will have a unique root in the interior of the long
arc. In this case, due to the formula~\eqref{eq:integrability_2}, the
sign of the imaginary parts of $\dot{\lambda}_1/\lambda_1$ and
$\dot{\lambda}_2/\lambda_2$ are different. The branch of $\arg\mu$
having on $\Sp^1$ at $\beta$ a local maximum is monotonically
decreasing for $\arg\lambda>\arg\beta$. Consequently, $\arg\lambda_2$
changes in the same direction as $\arg\mu$ at the maximum over
$\beta$. This implies the claim.
\end{proof}

\begin{lemma}\label{l2}
Let $\beta\in \Sp^1$ be a simple root of $b$ along the integral curve of
the vector field corresponding to
$c=\pm\frac{\lambda+\beta}{\lambda-\beta}b$. Two roots of $a$ can only
coalesce at $\beta$ at some point of the integral curve, if the branch
of $\arg\mu$ having on $\Sp^1$ at $\beta$ a local maximum increases,
and the branch of $\arg\mu$ having on $\Sp^1$ at $\beta$ a local
minimum decreases.
\end{lemma}

\begin{proof}\label{p3}
The number by which a prescribed value is attained by the function
$\mu$ inside a given domain $\Omega\subset\C^\ast$ cannot change,
as long as this value is not attained on the boundary $\partial\Omega$.
The function $\mu$ takes at the roots of $a$ a value
$\mu_0\in\{-1,1\}$ independent of $t$. If two roots coalesce at a
simple root $\beta\in\Sp^1$, then $\arg\mu$ takes the corresponding
value $\arg\mu_0\in2\pi\Z$ twice on $\Sp^1$, once with multiplicity
one at both branches of $\arg\mu$. Therefore $\arg\mu$ takes on a
neighbourhood of $\beta$ in $\Sp^1$ values in the complement of
$(\arg\mu_0-\epsilon,\arg\mu_0+\epsilon)$, as long as $\arg\mu_0$ is
taken at two roots of $a$ in a neighbourhood of $\beta$ away from
$\Sp^1$. More precisely the branch of $\arg\mu$ having at $\beta$ a
local maximum takes at the maximum a value smaller than
$\arg\mu_0$. The other branch takes at the minimum a
value larger then $\arg\mu_0$. Consequently the values of $\arg\mu$
increases on the branch having at $\beta$ a local maximum before the
roots of $a$ coalesce at $\beta$, and the values of $\arg\mu$
decreases on the branch having at $\beta$ a local minimum. In
Figure~8 the graphs of the two branches of $\arg\mu$ nearby a simple
root of $b$ in $\Sp^1$ for coalescing roots of $a$ are shown.
\end{proof}

\noindent{\it Continuation of the proof of the Theorem~\ref{t1}:}
Every simple root $\beta_i$ of $b$ in the short arc is a local
extremum of the restriction of $\arg\mu$ to $\Sp^1$. First we assume
that the short arc contains more than one root of $b$. Then two of the
roots of $b(\lambda)(\lambda-\lambda_1)(\lambda-\lambda_2)$ on $\Sp^1$ are
adjacent neighbours of $\beta_i$. The value $\arg\mu_1$ of $\arg\mu$ at
one of these two neighbours is closer to the value $\arg\mu_0$ of
$\arg\mu$ at $\beta_i$ than the other value $\arg\mu_2$. Then there
exists a smooth embedding $\beta_i:[-1,1]\hookrightarrow\Sp^1$ whose
restriction to $[0,1]$ moves the root $\beta_i$ to this neighbour and
obeys \eqref{eq:symmetry}. This path together with the constant
paths of all other roots of $b$ obey
conditions~\eqref{cond initial}-\eqref{cond reality} of
Proposition~\ref{global existence}. If the neighbour is a Sym point,
then condition~\eqref{cond embedding} is violated for $t=1$. The
construction of Proposition~\ref{global existence} applies to this
situation and yields a path of spectral data. At the end point of
the path the equations~\eqref{eq:integrability_2} becomes singular. As
we have seen in step~2 such singularities are bifurcation points of the
movement of the Sym points. The corresponding path of spectral data is an
integral curve of the vector field described in
Lemma~\ref{l1} and increases the mean curvature. This deformation
changes the values of $\mu$ at the simple root at $\beta_i$ of $b$ and
fixes the values of $\mu$ at all other roots of $b$, $a$ and at the
Sym points. For $t=1$ the root $\beta_i$ either meets one of the Sym
points $\lambda_1$ and $\lambda_2$ or another root of $b$. If
$\beta_i$ meets only one Sym point, then the deformation described in
step~2 moves $\beta_i$ into the long arc. If $\beta_i$
meets another root $\beta_j=\beta_i(1)$ of $b$, then we move with
Lemma~\ref{higher roots} $\beta_i$ and $\beta_j$ by a small
deformation away from $\Sp^1$. Afterwards we continue with the
deformation of another root of $b$ in the short arc.

After finitely many such deformations we arrive at spectral data with
at most one root of $b$ in the short arc. Moreover we can assume that
along the short arc the integral of $dh$
vanishes (and a branch of $h$ takes at $\lambda_1$ and $\lambda_2$
the same value).

If the integral of $dh$ along the short arc does
not vanish, we can move the last root of $b$ inside the short arc
into the long arc without shrinking the short arc to zero. If finally
the short arc does not contain any root of $b$, then the integral of
$dh$ along the short arc does not vanish. Since the
integral of $dh$ has to be a multiple of $2\pi\mi$, there exists
one point in the short arc, such that the integral from both Sym
points to this point of $dh$ are equal. At this point $\mu$ has
to be equal to $\pm1$. With Lemma~\ref{jump} we add to $a$ a double
root and to $b$ a simple root at this point in $\Sp^1$. We obtain new
spectral data in $\moduli^+$. With Lemma~\ref{smooth moduli} we deform the
double root of $a$ on $\Sp^1$ into two roots away from $\Sp^1$ and the
root of $b$ staying on $\Sp^1$ (compare Lemma~\ref{l2}).

After this deformation the sign of $dh$ changes at the root of $b$. Since
initially the integrals of $dh$ along both segments of the short arc
are equal up to sign, the integral of $dh$ along the short arc will be zero
afterwards. Therefore we end with exactly one root of $b$ in the
short arc and the integral of $dh$ over the short arc vanishes as we assume
in the preceding paragraph.

\noindent{\bf 6.} In step~6 we move all roots of $b$
with the exception of the single root in the short arc to the long arc.
To all pairs of roots in $\C^\ast\setminus\Sp^1$ we apply
successively the following lemma.

\begin{lemma}\label{l3}
Let the short arc contain one root and the integral of $dh$ along the
short arc vanish.
If there exists a pair of simple roots of $b$ interchanged by
$\lambda\to\bar{\lambda}^{-1}$ with values of $\mu$ not in $\Sp^1$,
then we can move one of these pairs of roots of $b$ to a point
$\lambda_0$ in the long arc.
\end{lemma}
\begin{proof}\label{p4}
We apply Proposition~\ref{global existence}. First we construct a path
$\beta_i:[-1,1]\to\overline{B(0,1)}$, which together with the path
$\beta_j(t)=\Bar{\beta}_i^{-1}$ obeys conditions~(i)-(iv) in
Proposition~\ref{global existence}. We consider for all $C>1$ the sets
$$S_C=\{(\lambda,\nu)\in\Sigma^\ast\mid\lambda\in B(0,1)\mbox{ and }
|\mu(\lambda,\nu)|=C\}.$$
If $S_C$ does not contain roots of $b$, then it is a one-dimensional
submanifold of $\Sigma^\ast$. On a punctured disc in $\Sigma$ around
the point $(\nu, \lambda) =(\infty,0)$ the function $h$ has a unique
single-valued injective branch with $\sigma^\ast h=-h$. For large
$C>1$, $S_C$ is completely contained in this punctured disc and has
only one connected component. If we decrease $C$, then $S_C$
remains a connected component as long as $S_C$ does not contain a root of
$b$. By assumption $S_C$ contains roots of $b$ for some $C>1$. There
exists a largest $C_1>0$, such that $S_{C_1}$ contains roots of $b$. Now
we choose a path $\gamma$ in $\overline{B(0,1)}$ starting at some point
$\lambda_0$ in the interior of the long arc, which intersects these
submanifolds $(S_C)_{C>1}$ transversally. For $C>C_1$ the sets $S_C$
are connected. Moreover the connected
component of $S_{C_1}$ which intersects $\gamma$ contains a root of
$b$. Therefore there exists a smallest $C_2>1$ such that the connected
component of $S_{C_2}$ which intersects $\gamma$ contains a root of
$b$. For $1<C<C_2$ the connected components of $S_C$ intersecting
$\gamma$ contain no root of $b$ and are either diffeomorphic to
$\arg\mu\in\R$ or to $\arg\mu\in\R/2n\pi\Z$ for some
$n\in\N$. Therefore there exists a smooth path intersecting these
lines transversally from the root $\beta_i$ of $b$ in $S_{C_2}$ to
$\lambda_0$. We may choose this path in such a way that any
branch of $h$ maps this path to a straight line in the $h$-plane. In
this way we get a smooth embedded path $\beta_i:[0,1]\to\overline{B(0,1)}$
from a root $\beta_i=\beta_i(0)$ of $b$, with values of $\mu$ not in
$\Sp^1$, to $\lambda_0=\beta_i(1)\in\Sp^1$.

We extend this path smoothly to $[-1,0]$ in such a way that
\eqref{eq:symmetry} is satisfied. If this extension meets another root
of $b$, then we slightly change $\lambda_0$ on $\Sp^1$ and the
straight line $\beta_i([0,1])$ in $h$-plane from $\lambda_0$ to
$\beta_1$ such that the unique extension to $[-1,0]$ obeying
\eqref{eq:symmetry} does not meet another root of $b$. The map $h$ maps
$\beta_i([-1,0])$ to the segment $\beta_i([0,1])$
($h(\beta_i(-t))=h(\beta_i(t))$).
Now we claim that $\beta_i(-1)$ does not belong to $\Sp^1$ but to another
component of $|\mu|=1$. Otherwise the path $\beta_i([-1,1])$
divides $B(0,1)$ into two connected components. But in any neighbourhood
of a point of $S_C$ with $1<C<C_2$ there starts a straight half-line in
the $h$-plane parallel to the segments $\beta_i([0,1])$ and
$\beta_i([-1,0])$ which does not contain roots of $b$ and ends at
$\lambda=0$. Such half-lines do not intersect $\beta_i([-1,1])$. Hence
$\beta_i([-1,1])$ does not divide $B(0,1)$ and $\beta_i(-1)\not\in\Sp^1$.

Now we apply Proposition~\ref{global existence} to prove the Lemma.
For the conjugated root of $b$ we choose the conjugated
path $\beta_j(t)=\bar{\beta}_i^{-1}(t)$. Along the deformation it could happen
that the mean curvature goes to zero i.e. the length of the short arc goes
to $\pi$. To avoid this case we move the simple root in the short arc.

For the unique root of $b$ in
the short arc we choose a path $\beta_k:[-1,1]\to\Sp^1$ along the
short arc connecting both Sym points $\beta_k(-1)=\lambda_1$ and
$\beta_k(1)=\lambda_2$ and obeying \eqref{eq:symmetry}. Besides these
three paths $\beta_i$, $\beta_j$ and $\beta_k$ we choose constant
paths for all other roots of $b$. For any $s\in[0,1)$ we evaluate
$\beta_k$ at $st$ and restrict the path $\beta_k$ to $[-s,s]$. Then
these paths obey
conditions~\eqref{cond embedding}-\eqref{cond reality}. The
Proposition~\ref{global existence} yields a two-dimensional family of
deformed spectral data $(a_{s,t},b_{s,t},\lambda_{1,s,t},\lambda_{2,s,t}
)_{(t,s)\in[0,1]\times[0,1)}\in\moduli^g$, such that
$\Delta$~\eqref{eq:delta} takes at
the deformed roots $\beta_i$, $\beta_j$ and $\beta_k$ of $b_{s,t}$ the
values $\Delta_0(\beta_i(t))$, $\Delta_0(\beta_j(t))$ and
$\Delta_0(\beta_k(s))$, respectively. This family is again unique and
for $t\in[0,1)$ smooth, if we normalise $\lambda_1\lambda_2^{-1}$
independent of $s$ and $t$. Let $L(s,t)$ denote the corresponding
length of the short arc. Due to Lemma~\ref{l1} $s\mapsto L(s,t)$ is for
all $t\in[0,1]$ monotonically decreasing. With
$L_{\min}=\min\{L(0,t)\mid t\in[0,1]\}$ there exists for all
$t\in[0,1]$ a unique $s(t)\in[0,1)$ such that $L(s,t)=L_{\min}$. First
we take the path of spectral data parameterised by $(s,0)$ with
$s\in[0,s(0)]$ and then the path of spectral data parameterised by
$(t,s(t))$ with $t\in[0,1]$. We obtain a continuous
path of spectral data in $\moduli^g_+$, which deforms the conjugated
pair $(\beta_i,\beta_j)$ of roots of $b$ into a pair of double roots
on $\Sp^1$.
\end{proof}

\noindent{\bf 7.} In step~7 we finally decrease the
genus. First we show

\begin{lemma}\label{l4}
Suppose $a\in\C^{2g}[\lambda]$ has $2g$ simple roots, and
$b\in\C^{g+1}[\lambda]$ has $g+1$ simple roots in $\Sp^1$, and they obey
conditions~(i)-(iii) in Definition~\ref{spectral data}. Then
the set of $\lambda\in\CP$ with $h\in\mi\R$
is the union of $\Sp^1$ with finitely mutually disjoint smooth curves
intersecting $\Sp^1$ exactly once in a root of $b$. Each curve
connects a pair of roots of $a$ interchanged by
$\lambda\to\bar{\lambda}^{-1}$ or $(\lambda=0,\lambda=\infty)$ with
each other.
\end{lemma}
\begin{proof}\label{p5}
The set of $\lambda\in\C^\ast$ where $h$ takes values in
$\mi\R$ is away from the roots of $a$ and $b$ a submanifold
of $\C^\ast$. At each root of $a$ the function $(h-n\mi\pi)^2$
vanishes for some $n\in\Z$ and is a local coordinate. Hence
all roots of $a$ are endpoints of this submanifold. At the roots of
$b$ two such submanifolds intersect transversely. Furthermore $h^{-2}$
is a local coordinate at the marked points $\lambda=0$ and
$\lambda=\infty$. They are also endpoints of this submanifold in
$\CP$. Therefore the intersection with $B(0,1)\subset\C^\ast$
is a disjoint union of smooth curves, which end either at $\lambda=0$,
or at the roots of $a$ inside of  $B(0,1)$, or at the roots of $b$ in
$\partial B(0,1)$. In every connected component of the complement of
these paths in $B(0,1)$ which does not contain $\lambda=0$ the real
part of $h$ vanishes by the maximum principle. Therefore
the complement of these paths in $B(0,1)$ is connected. Hence there is
no path, which connects a root of $b$ in $\partial B(0,1)$ with
another such root. Then all $g+1$ roots of $b$ in $\partial B(0,1)$
are connected either with $\lambda=0$ or with one of the $g$ roots of
$a$ inside of $B(0,1)$. The involution
$\lambda\mapsto\bar{\lambda}^{-1}$ maps these paths in $B(0,1)$ onto
the corresponding paths in $\C^\ast\setminus B(0,1)$. This proves the claim.
\end{proof}
\noindent{\it Conclusion of the proof of the Theorem~\ref{t1}:}
We thus arrive at spectral data $(a,b,\lambda_1,\lambda_2)$ such
that the short arc contains exactly one root of $b$ and the integral
of $dh$ along the short arc vanishes. With Lemma~\ref{l3} we move
successively the pairs of roots of $b$ away from $\Sp^1$ to double
roots in the long arc and separate with Lemma~\ref{higher roots} the
double root into two different roots on $\Sp^1$. Afterwards we continue
with the deformation of another pair of roots of $b$ in
$\C^\ast\setminus\Sp^1$.

We are now in the situation where all roots of $b$ are simple roots
inside the interior of the long arc, and one simple root lies in the
interior of the short arc. Due to Lemma~\ref{l4} all roots of $b$
correspond to a pair of branch points. For any root $\beta_i$ of $b$
in the long arc, which does not correspond to the pair
$(\lambda=0,\lambda=\infty)$, there exists a smooth curve
$\beta_i:[-1,1]\to\C^\ast$ from the root $\beta_i(-1)$ in $B(0,1)$
of $a$ to the root $\beta_i(1)$ outside of $B(0,1)$ along the
curve described in Lemma~\ref{l4}. Furthermore there exists a
branch $h$ along this path and a unique $n_i\in\Z$ such that
$h-n_i\pi\mi$ vanishes at both end points by the reality condition, but
not in between, since $dh$ has only one simple root on this path at
the zero of $b$. Finally we may parameterise this path in such a way
that \eqref{eq:symmetry} is fulfilled. Together with the
constant paths for all other roots of $b$ this path obeys
conditions~\eqref{cond initial}-\eqref{cond reality}. The path meets
two roots of $a$ and does not obey \eqref{cond embedding}. The
function $(h-n_i\mi\pi)^2$ depends on $\lambda$ and does not have
critical points at the roots of $a$. The whole construction of the
proof of Proposition~\ref{global existence} carries over. But in this
case the tubular neighbourhood $V_i$ of the path $\beta_i([-1,1])$
contains two roots of $a$ at the roots of $A_i=(h-n_i\pi\mi)^2$ like
in Lemma~\ref{smooth moduli}. During the whole deformation these two
roots of $A_{i,t}$ are in $W_{i,t}$. For $t=1$ these two roots
form a double root at $\beta_i(1)$, which is a single root for
$t=0$. Due to
Lemma~\ref{l1} the mean curvature increases along this path, and this
path stays in $\moduli_+^g$. At the end point we reduce with
Lemma~\ref{jump} the genus by one. We apply this deformation to all
roots of $b$ in the long arc with the exception of possibly the root
$\beta_j$ which is connected by the curve of Lemma~\ref{l4} with the pair
$(\lambda=0,\lambda=\infty)$. The spectral genus is successively
reduced to at most one.

If finally the unique simple root $\beta_k$ in the short arc is
connected by the curve of Lemma~\ref{l4} with a pair of finite
branch points, we apply the analogous deformation to $\beta_k$.
Due to Lemma~\ref{l1} along this deformation the length of the short
arc increases, and the end point has spectral genus zero but might not
belong to $\moduli_+^0$. In this case we increase afterwards
$\arg(\mu)$ at the local maximum over the last root $\beta_j$
of $b$ in the long arc. Since the spectral genus is zero $h$ is
a meromorphic function and has two roots over $-\beta_j$. If we
increase the argument of $\mu$ at the local maximum over $\beta_j$ by
a large real number $s$ then $b(0)$ becomes very large and both Sym points
move near to the roots of $h$ at $-\beta_j\in\Sp^1$. Therefore
the short arc becomes very small and the path enters again
$\moduli_+^0$. Finally we interchange the order of these two
deformations. We first increase the argument of $\mu$ at the local
maximum over $\beta_j$ by $s$ along the path constructed in
Lemma~\ref{l4} connecting $\beta_j$ with $\lambda=0$ and
$\lambda=\infty$. This deformation decreases the length of the short
arc. Afterwards we increase the length of the short arc along the path
$\beta_k$ of the root in the short arc. As a result the whole path
stays in $\moduli_+$ and reduces the spectral genus to zero. This
completes the proof of Theorem~\ref{t1}.
\end{proof}
%
%%%%%%%%%%%%%%%%%%%%%%%%%%%%%%%%%%%%%%%%%%%%%%%%%%%%%
%%%%%%%%%%%%%%%%%%%%%%%%%%%%%%%%%%%%%%%%%%%%%%%%%%%%%
%
\section{Isolated property of $\mrot$ in $\mae$}\label{sec:isolated}
\begin{theorem}\label{isolated}
{\rm{(i)}} For $g=0,1$: $\mrot^g$ is open and closed in $\moduli_+^g$. \\
{\rm{(ii)}} Let $(a,b,\lambda_1,\lambda_2)\!\in\!\mrot^1$ and
$(\tilde{a},\tilde{b},\lambda_1,\lambda_2)\!\in\!
\moduli_+^{1+\deg p}$ with $\deg p \geq 1$
be as in Lemma~\ref{jump}. At
$(\tilde{a},\tilde{b},\lambda_1,\lambda_2)$ starts a piecewise
continuous path in $\moduli_+$ to
$\moduli_+^0\!\setminus\!\mrot^0$ with decreasing $g$.
\end{theorem}
The first statement shows that a piecewise continuous path
can only enter or leave $\mrot$ at jumps described in
Lemma~\ref{jump} from $(a,b,\lambda_1,\lambda_2)\in\mrot^g$ to
$(\tilde{a},\tilde{b},\lambda_1,\lambda_2)\not\in\mrot^{g+\deg p}$. The
second statement will imply in the proof of Theorem~\ref{thm:main} in
Section~\ref{sec:proof main theorem} that
$(\tilde{a},\tilde{b},\lambda_1,\lambda_2)$ does not belong to
$\mae$ and that $\mrot$ is isolated in $\mae$.

\noindent{\it Proof of }(i): %\begin{proof}
First we show that $\mrot^g$ ($g=0$ or $g=1$) is open in $\moduli^g_+$. We use the implicit function theorem to prove that $\moduli^g$ is at $(a,b,\lambda_1,\lambda_2)\in\mrot^g$ a real submanifold of $\C^{2g}[\lambda]\times \C^{g+1}[\lambda]\times\Sp^1\times\Sp^1$ of dimension $g+1$, the same dimension as $\mrot$. Due to Lemma~\ref{smooth moduli}, $\tilde{\moduli}^g$ is at $(a,b,\lambda_1,\lambda_2)\in\mrot^g$ a submanifold. Due to Theorem~\ref{thm:onesided revolution}, $a$ and $b$ have a common root only in case $g=1$ and $(a,b,\lambda_1,\lambda_2)$ belongs to the image of the map
\begin{align}\label{eq:embedding}
    \mrot^0&\hookrightarrow\mrot^1,&(a,b,\lambda_1,\lambda_2)&\mapsto
    ((\lambda+1)^2a,(\lambda+1)b,\lambda_1,\lambda_2)\,.
\end{align}
On the image of \eqref{eq:embedding}, the polynomial $a$ is equal to $a=-(\lambda+1)^2/16$ and the coefficients $\alpha$ and $H$ in Theorem~\ref{thm:onesided revolution} are local coordinates of $\tilde{M}$. The inequality in \eqref{eq:a_reality} does not allow simple roots of $a$ on $\Sp^1$, which is equivalent to $\alpha\ge 2$. Therefore $\moduli^1$ is at $(a,b,\lambda_1,\lambda_2) \in \mrot^1$ a two-dimensional manifold with image of \eqref{eq:embedding} as boundary. This implies that $\mrot^0$ and $\mrot^1$ are open in $\moduli^0_+$ and $\moduli^1_+$.

For sequences in $\mrot^g$ ($g=0$ or $g=1$) which converge in $\moduli^g_+$ the corresponding sequences of parameters $H$ or $(H,\alpha)$ have to converge. Therefore $\mrot^g$ is closed in $\moduli^g_+$.

\noindent{\it Proof of }(ii):
For $(a,b,\lambda_1,\lambda_2)\in\mrot^1$, the short arc and the long arc contain both one root of $b$. Due to Theorem~\ref{thm:onesided revolution}, $\tilde{a}$ has at least one pair of double roots away from $\Sp^1$. If $\tilde{a}$ has several pairs of double roots we choose one and remove all others. Hence we may assume $(\tilde{a},\tilde{b},\lambda_1,\lambda_2)\in\moduli^3$ with $\tilde{a}$ having a pair of double roots $\beta,\bar{\beta}^{-1}\in\C\setminus\Sp^1$. Since $b$ has only roots at $\lambda=\pm1$ the corresponding $\tilde{b}$ has simple roots at $\beta$ and $\bar{\beta}^{-1}$. Now we follow the path described in Theorem~\ref{t1}. We start with step~2 and open with Lemma~\ref{smooth moduli} two double roots of $\tilde{a}$ at $\beta$ and $\bar{\beta}^{-1}$ and eventually a third double root at $\lambda=-1$ by a small deformation with decreasing mean curvature. Due to Lemma~\ref{l4}, the single root in the short arc is connected along a path with $\mu\in\Sp^1$ to $(\lambda=0,\lambda=\infty)$. We move the roots $\beta$ and $\bar{\beta}^{-1}$ of $\tilde{b}$ along a path described in Lemma~\ref{l3} to the long arc and then separate them by a small deformation described in step~6 of the proof of Theorem~\ref{t1}. We end up with three roots of $\tilde{b}$ in the long arc and one root of $\tilde{b}$ in the short arc. All roots of the long arc are connected with the paths described in Lemma~\ref{l4} with pairs of finite branch points. With the deformations described in Lemma~\ref{l1} we successively shorten these paths from $\Sp^1$ to the finite branch points until all three of them are very small. Finally we end up with spectral data of genus zero with three double roots of $\tilde{a}$ on $\Sp^1$. Hence the end point of the deformation in Theorem~\ref{t1} does not belong to $\mrot^0$ which are classified in Theorem~\ref{thm:onesided revolution}. This completes the proof of Theorem~\ref{isolated}.
\qed%\end{proof}
%
%%%%%%%%%%%%%%%%%%%%%%%%%%%%%%%%%%%%%%%%%%%%%%%%%%%%%%%%%%%%%%%%
%%%%%%%%%%%%%%%%%%%%%%%%%%%%%%%%%%%%%%%%%%%%%%%%%%%%%%%%%%%%%%%%
%
\section{Proof of Theorem~\ref{thm:main}}\label{sec:proof main theorem}
By definition of $\mae$ it suffices to show the equality $\mae=\bigcup_{g\in\N\cup\{0\}}\mae^g=\mrot$, i.e.\ $\mae^0=\mrot^0$, $\mae^1=\mrot^1$ and $\mae^g=\emptyset$ for $g>1$. Indeed we will show that any subset $\mae\subset\moduli_+$, which has the properties~\eqref{prop1}-\eqref{prop3} in Theorem~\ref{characterization} is equal to $\mrot$. We remark that we do not make use of the inclusion $\mrot^1\subset\mae^1$ which is equaivalent to $\bigcup_{(a,b,\lambda_1,\lambda_2)\in\mrot^1}\ann(a,b,\lambda_1,\lambda_2)\subset\annae$. In fact, this inclusion follows from the properties~\eqref{prop1}-\eqref{prop3} and~\eqref{eq:embedding} since $\mrot^1$ is connected.

First we show for pairs $(a,b,\lambda_1,\lambda_2)\in\mrot$ and $(\tilde{a},\tilde{b},\lambda_1,\lambda_2)\not\in\mrot$ as in Lemma~\ref{jump} that $(\tilde{a},\tilde{b},\lambda_1,\lambda_2)$ does not belong to $\mae$. The map \eqref{eq:embedding} guarantees either $(a,b,\lambda_1,\lambda_2)\in\mrot^1$ or $((\lambda+1)^2a,(\lambda+1)b,\lambda_1,\lambda_2)\in\mrot^1$. If in the second case $(\lambda+1)^2$ divides $\tilde{a}$, then the pair $((\lambda+1)^2a,(\lambda+1)b,\lambda_1,\lambda_2)\in\mrot^1$ and $(\tilde{a},\tilde{b},\lambda_1,\lambda_2)\not\in\mrot$ fulfills the assumptions of part~(ii) in Theorem~\ref{isolated}. Otherwise due to Lemma~\ref{jump} $((\lambda+1)^2\tilde{a},(\lambda+1)\tilde{b},\lambda_1,\lambda_2)\in\moduli_+^{1+\deg p}$ and $((\lambda+1)^2a,(\lambda+1)b,\lambda_1,\lambda_2)$ and $((\lambda+1)^2\tilde{a},(\lambda+1)\tilde{b},\lambda_1,\lambda_2)$ fulfills these assumptions. Theorem~\ref{isolated}~(ii) yields a piecewise continuous path either from $(\tilde{a},\tilde{b},\lambda_1,\lambda_2)\not\in\mrot$ or from $((\lambda+1)^2\tilde{a},(\lambda+1)\tilde{b},\lambda_1,\lambda_2)\not\in\mrot$ to $\moduli_+^0\setminus\mrot^0$ with decreasing $g$. For $(\tilde{a},\tilde{b},\lambda_1,\lambda_2)\in\mae$ this path stays due to properties~(ii)-(iii) in $\mae$ in contradiction to (i). This implies $(\tilde{a},\tilde{b},\lambda_1,\lambda_2)\not\in\mae$.

Now we show $\mae\subset\mrot$. Due to Theorem~\ref{t1} there starts at any element of $\mae$ a piecewise continuous path to $\moduli_+^0$. Due to properties~\eqref{prop2}-\eqref{prop3} this path stays in $\mae$. Due to property~\eqref{prop1} it ends in $\mrot^0$. Since $\mrot^g$ is closed in $\moduli_+^g$ (Theorem~\ref{isolated}) at some point the path enters the set $\mrot$ and stays in $\mrot$ until it reaches $\mrot^0$. Due to Theorem~\ref{isolated}~(i) and the first claim this point has to be the initial element of $\mae$ and all elements of $\mae$ belong to $\mrot$.

Conversely, due to properties~\eqref{prop1}-\eqref{prop3} and \eqref{eq:embedding} $\mae$ contains $\mrot$, since $\mrot^1$ is connected.
%
%%%%%%%%%%%%%%%%%%%%%%%%%%%%%%%%%%%%%%%%%%%%%%%%%%%%%%%%%%%%%%%
%%%%%%%%%%%%%%%%%%%%%%%%%%%%%%%%%%%%%%%%%%%%%%%%%%%%%%%%%%%%%%%
%%%%%%%%%%%%%%%%%%%%%%%%%%%%%%%%%%%%%%%%%%%%%%%%%%%%%%%%%%%%%%%
%
\bibliographystyle{amsplain}
%\bibliography{ref}

\begin{thebibliography}{10}

\bibitem{AndL}
B.~Andrews and H.~Li, \emph{Embedded constant mean curvature tori in the
  three-sphere}, J. Differential Geom. \textbf{99} (2015), no.~ 2, 169--189.

\bibitem{Bob:tor}
A.~I.~Bobenko, \emph{All constant mean curvature tori in {$\mathbb{R}^3$},
  {$\mathbb{S}^3$}, {$\mathbb{H}^3$} in terms of theta-functions}, Math. Ann.
  \textbf{290} (1991), 209--245.

\bibitem{Bob:cmc}
\bysame, \emph{Constant mean curvature surfaces and integrable equations},
  Russian Math. Surveys \textbf{46} (1991), 1--45.

\bibitem{Bre2}
S.~Brendle, \emph{Alexandrov immersed minimal tori in {$S^3$}}, Math. Res.
  Lett. \textbf{20} (2013), no.~3, 459--464.

\bibitem{Bre}
\bysame, \emph{Embedded minimal tori in {$S^3$} and the {L}awson conjecture},
  Acta Math. \textbf{211} (2013), no.~2, 177--190.

\bibitem{BurFPP}
F.~E.~Burstall, D.~Ferus, F.~Pedit, and U.~Pinkall, \emph{Harmonic tori in
  symmetric spaces and commuting {H}amiltonian systems on loop algebras}, Ann.
  of Math. \textbf{138} (1993), 173--212.

\bibitem{BurP_adl}
F.~E.~Burstall and F.~Pedit, \emph{Harmonic maps via {A}dler-{K}ostant-{S}ymes
  theory}, Harmonic maps and integrable systems, Aspects of Mathematics, vol.
  E23, Vieweg, 1994.

\bibitem{BurP:dre}
\bysame, \emph{Dressing orbits of harmonic maps}, Duke Math. J. \textbf{80}
  (1995), no.~2, 353--382.

\bibitem{DorPW}
J.~Dorfmeister, F.~Pedit, and H.~Wu, \emph{Weierstrass type representation of
  harmonic maps into symmetric spaces}, Comm. Anal. Geom. \textbf{6} (1998),
  no.~4, 633--668.

\bibitem{HKS1}
L.~Hauswirth, M.~Kilian, and M.~U.~Schmidt, \emph{Finite type minimal annuli in
  {$\mathbb{S}^2 \times \mathbb{R}$}},  Ill. J. Math. \textbf{57} (2013), no.~3, 697-741.

\bibitem{HKS4}
\bysame, \emph{On mean-convex alexandrov embedded surfaces in the 3-sphere},
  arXiv:1407.2150.

\bibitem{HKS2}
\bysame, \emph{Properly embedded minimal annuli in {$\mathbb{S}^2 \times
  \mathbb{R}$}}, arXiv:1210.5953.

\bibitem{Hit:tor}
N.~Hitchin, \emph{Harmonic maps from a 2-torus to the 3-sphere}, J.
  Differential Geom. \textbf{31} (1990), no.~3, 627--710.

\bibitem{Hop}
H.~Hopf, \emph{Differential geometry in the large}, Lecture Notes in
  Mathematics, vol. 1000, Springer-Verlag, 1983.

\bibitem{KilMS}
M.~Kilian, I.~McIntosh, and N.~Schmitt, \emph{New constant mean curvature
  surfaces}, Experiment. Math. \textbf{9} (2000), no.~4, 595--611.

\bibitem{KilSS:equi}
M.~Kilian, M.~U.~Schmidt, and N.~Schmitt, \emph{Flows of constant mean curvature
  tori in the 3-sphere: The equivariant case}, to appear in J. reine angew. Math.

\bibitem{KilSS}
M.~Kilian, N.~Schmitt, and I.~Sterling, \emph{Dressing {CMC} n-{N}oids}, Math.
  Z. \textbf{246} (2004), no.~3, 501--519.

\bibitem{Law:unknot}
H.~B.~Lawson, Jr., \emph{The unknottedness of minimal embeddings}, Invent.
  Math. \textbf{11} (1970), 183--187.

\bibitem{McI:tor}
I.~McIntosh, \emph{Harmonic tori and their spectral data}, Surveys on geometry
  and integrable systems, Adv. Stud. Pure Math., vol.~51, Math. Soc. Japan,
  Tokyo, 2008, pp.~285--314.

\bibitem{MPR}
W.~H.~{Meeks III}, J.~P{\'{e}}rez, and A.~Ros, \emph{Properly embedded minimal
  planar domain}, to appear in Ann. Math.

\bibitem{pap}
C.~D.~Papakyriakopoulos, \emph{On solid tori}, Proc. London Math. Soc. (3)
  \textbf{7} (1957), 281--299.

\bibitem{PinS}
U.~Pinkall and I.~Sterling, \emph{On the classification of constant mean
  curvature tori}, Ann. Math. \textbf{130} (1989), 407--451.

\bibitem{PreS}
A.~Pressley and G.~Segal, \emph{Loop groups}, Oxford Science Monographs, Oxford
  Science Publications, 1988.

\bibitem{St}
N.~Steenrod, \emph{The {T}opology of {F}ibre {B}undles}, Princeton Mathematical
  Series, vol. 14, Princeton University Press, Princeton, N. J., 1951.

\bibitem{TerU}
C.~Terng and K.~Uhlenbeck, \emph{B\"{a}cklund transformations and loop group
  actions}, Comm. Pure and Appl. Math \textbf{LIII} (2000), 1--75.

\end{thebibliography}
\def\cydot{\leavevmode\raise.4ex\hbox{.}} \def\cprime{$'$}
\providecommand{\bysame}{\leavevmode\hbox to3em{\hrulefill}\thinspace}
\providecommand{\MR}{\relax\ifhmode\unskip\space\fi MR }
\providecommand{\MRhref}[2]{%
  \href{http://www.ams.org/mathscinet-getitem?mr=#1}{#2}
}
\providecommand{\href}[2]{#2}

\end{document}